\documentclass{article}
\pdfoutput=1
\usepackage{theorems}
\usepackage{biblio}
\usepackage{format}
\usepackage{diag}
\usepackage{lists}
\usepackage{fonts}
\usepackage{refs}
\usepackage{macros}
\usepackage{textmacros}

\title{Motivic Hodge modules}
\author{Brad Drew}
\date{}

\begin{document}
\maketitle

\begin{abstract}
\noindent
We construct a \qcategorically\ enhanced Grothendieck six-functor formalism on schemes of finite type over the complex numbers.
In addition to satisfying many of the same properties as M.~Saito's derived categories of mixed Hodge modules, this new six-functor formalism receives canonical motivic realization functors compatible with Grothendieck's six functors on constructible objects.
\end{abstract}

\tableofcontents

\section{Introduction}
\nlabel{intro}

\subsection*{Historical background}

According to Deligne, the rational Betti cohomology $\fct[um = {*}, d = {B}]{H}{X, \rational}$ of the finite-type $\complex$-scheme $X$ carries a rational, polarizable mixed Hodge structure (\cite{Deligne_hodgeII,Deligne_hodgeIII}).
M.~Saito's derived categories $\D[u = {b}]{\mhm[u = {p}]{X}}$ of polarizable mixed Hodge modules provide a theory of constructible coefficients for rational Betti cohomology equipped with Deligne's mixed Hodge structure.
These categories admit a formalism of Grothendieck's six functors, $f^*$, $f_*$, $f_!$, $f^!$, $\otimes$, and $\intmor{}{}$, and conservative functors $\D[u = {b}]{\mhm[u = {p}]{X}} \to \D[u = {b}, d = {c}]{\prns{-}^{\tu{an}},\rational}$ compatible with the six functors.
These functors are powerful tools for studying Deligne's mixed Hodge structures, how they vary in families, and how they degenerate.

By the work of J.~Ayoub (\cite{Ayoub_six-operationsI,Ayoub_six-operationsII}), there is also a six-functor formalism associated with the $\tatesphere$-stable $\affine^1$-homotopy category $\SH{X}$ over $X$, where $\tatesphere$ denotes the Tate object.
In a sense that we will not make precise here, $\SH{-}$ should be an initial object in the \qcategory\ of six-functor formalisms that satisfy the properties of excision with respect to the Nisnevich topology, $\affine^1$-invariance, and $\tatesphere$-stability.
In particular, if $\SH[d = {c}]{X} \subseteq \SH{X}$ denotes the full subcategory spanned by the constructible objects, then there should exist a realization functor
\[
\rho^*_X:
  \SH[d = {c}]{X}
    \to \D[u = {b}]{\mhm[u = {p}]{X}}
\]
for each $\complex$-scheme $X$, and these realization functors should commute with Grothendieck's six functors.

At the moment, however, such realization functors are not known to exist.
Over $\spec{\complex}$, if we equip $\SH{\spec{\complex}}$ with the symmetric monoidal structure associated with the smash product $\wedge$, then there is a symmetric monoidal realization functor
\begin{equation}
\nlabel{intro.0.a}
\rho^{*, \otimes}_{\spec{\complex}}:
  \SH[d = {c}]{\spec{\complex}}^{\wedge}
    \to \D[u = {b}]{\mhsp[\rational]}^{\otimes}
\end{equation}
into the bounded derived category of rational, polarizable mixed Hodge $\rational$-structures:
at the level of tensor-triangulated categories, this follows from the work of A.~Huber (\cite{Huber_realization-of-voevodsky's, Huber_corrigendum});
at the level of symmetric monoidal \qcategories, it follows from \cite[Corollary~1.2]{Robalo_K-theory-and-the-bridge} and \cite{Drew_rectification-of-Deligne's}.
Over each smooth, quasi-projective $\complex$-scheme $X$, F.~Ivorra has defined a functor $\SH[d = {c}]{X} \to \D[u = {b}]{\mhm[u = {p}]{X}}$ in \cite{Ivorra_perverse-hodge}, but these functors are not known to commute with any of the six functors in full generality.

\subsection*{An alternative approach}

The fundamental obstacle to the construction of suitable realization functors with values in $\D[u = {b}]{\mhm[u = {p}]{X}}$ is the lack of enhancements of the associated six-functor formalism:
while the triangulated category $\D[u = {b}]{\mhm[u = {p}]{X}}$, as the bounded derived category of an Abelian, underlies a differential graded category, the six functors linking these triangulated categories do not admit obvious enhancements by differential graded functors or more general functors of \qcategories.
The same obstacle rears its head when one attempts to examine mixed Hodge modules on simplicial schemes.

As a workaround, we propose the following course of action.
Using techniques from stable $\affine^1$-homotopy theory and higher algebra, we construct a new six-functor formalism $X \mapsto \DH{X}$ of \emph{motivic Hodge modules}.
By construction, the triangulated categories $\DH{X}$ and the six functors linking them are \qcategorically\ enhanced, receive canonical realization functors from $\SH{-}$ compatible with the six functors when restricted to constructible objects, and the assignment $X \mapsto \DH{X}$ extends naturally to simplicial $\complex$-schemes.

Moreover, for each $\complex$-scheme $X$, the full subcategory $\DH[d = {c}]{X} \subseteq \DH{X}$ spanned by the constructible objects admits a canonical enrichment
\[
\hom[um = {\mhs}, dm = {\DH[d = {c}]{X}}]{}{}:
  \prns{\DH[d = {c}]{X}}\op \times \DH[d = {c}]{X}
    \to \D[u = {b}]{\mhsp[\rational]}.
\]
One recovers Deligne's mixed Hodge structure $\twist{\fct[um = {s}, d = {B}]{H}{X, \rational}}{r}$ from this enrichment: 
there is a canonical isomorphism
\[
\twist{\fct[um = {s}, d = {B}]{H}{X, \rational}}{r}
  \simeq \h^s\hom[um = {\mhs}, dm = {\DH[d = {c}]{X}}]{\1{\DH{X}}}{\tate{\DH{X}}{r}}
  \in \mhsp[\rational]
\]
for each $X$, and each $\prns{r,s} \in \integer^2$, where $\h^s$ is the $s$th cohomology functor associated with the natural t-structure on $\D[u = {b}]{\mhsp[\rational]}$, and where $\tate{\DH{X}}{r} \in \DH{X}$ is the $r$th Tate twist of the monoidal unit.
By the same token, there is a canonical isomorphism 
\begin{equation}
\nlabel{intro.0.b}
\fct[um = {s}, dm = {\mc{H}}]{H}{X, \twist{\rational}{r}}
  \simeq \hom[dm = {\DH{X}}]{\1{\DH{X}}}{\state{\DH{X}}{r}{s}},
\end{equation}
where $\fct[um = {s}, dm = {\mc{H}}]{H}{X, \twist{\rational}{r}}$ denotes the rational absolute Hodge cohomology of $X$, and $\tate{\DH{X}}{r} \in \DH{X}$.

In light of the canonical realization functors $\SH{-} \to \DH{-}$, the task of constructing realization functors with values in $\D[u = {b}]{\mhm[u = {p}]{X}}$ is reduced to that of constructing comparison functors 
\begin{equation}
\nlabel{intro.0.c}
\chi^*_X:
  \DH[d = {c}]{X}
    \to \D[u = {b}]{\mhm[u = {p}]{X}}.
\end{equation}
At first glimmer, this may appear to be a lateral move at best, so let us summarize what we regard as the main advantages of this approach.

First, even in the absence of a comparison with $\D[u = {b}]{\mhm[u = {p}]{X}}$, one still knows that the categories $\DH{X}$ compute the correct data.
In other words, morally speaking, the isomorphisms \nref{intro.0.b} and the compatibilities encoded in the six-functor formalism determine the pseudofunctor $X \mapsto \DH[d = {c}]{X}$ up to pseudonatural equivalence.
Furthermore, the \qcategorical\ enhancement of the six-functor formalism for $\DH{-}$ and the attendant theory of motivic Hodge modules on simplicial schemes are an enormous boon.

Second, the comparison functors $\chi^*_X$ of \nref{intro.0.c} should be fully faithful.
This goes back to our previous remark that the isomorphisms $\nref{intro.0.b}$ should determine the six-functor formalism $\DH[d = {c}]{-}$ completely.

Third, the lack of a reasonable analogue on $\SH{X}$ of the perverse t-structure on $\D[u = {b}, d = {c}]{X^{\tu{an}},\rational}$ is a serious obstacle.
Indeed, even after passing to the $\rational$-linear, \'etale-localized variant $\SH[u = {\'et}]{X}_{\rational}$ of $\SH{X}$, through which the realization $\SH{X} \to \DH{X}$ factors, the existence of a suitable t-structure on $\SH[u = {\'et}]{X}_{\rational}$ is still the subject of notorious conjecture.
With current knowledge, the strategy of constructing the realization functor $\rho^*_X$ by first defining an exact functor between Abelian categories and then deriving it is therefore untenable.
The existence of a reasonable t-structure on $\DH[d = {c}]{-}$, on the other hand, is much more attainable.
Indeed, the Be\u{\i}linson-Soul\'e conjecture is a major hurdle in establishing the existence of the desired t-structure on $\SH[u = {\'et}, d = {c}]{X}_{\rational}$, but the analogue of this conjecture for absolute Hodge cohomology is true.

\subsection*{Main results}

Suppose we have a pseudofunctor $X \mapsto \D{X}$ from the category $\prns{\sch[ft]{\complex}}\op$ to the strict $2$-category of idempotent-complete tensor-triangulated categories, tensor-triangulated functors and tensor-triangulated natural transformations. 
If $\D{-}$ is to serve as a theory of constructible coefficients for mixed Hodge theory, then it ought to have the properties listed below.

\setcounter{thm}{0}
\theoremstyle{defn}
\newtheorem{desiderata}[thm]{Desiderata}
\begin{desiderata}
\nlabel{desiderata}
In order for $\D{-}$ to be a reasonable theory of constructible coefficients, it ought to satisfy the following properties. 
\begin{enumerate}
\item
\nlabel{desideratum.1}
\textbf{Absolute Hodge cohomology:}
For each $X \in \sch[sepft]{\complex}$ and each $\prns{r,s} \in \integer^2$, one has
\[
\absolutehodge[um = {s}]{X, \twist{\rational}{r}}
  \simeq \hom[dm = {\D{X}}]{\1{\D{X}}}{\state{\D{X}}{r}{s}}.
\]
\item
\nlabel{desideratum.2}
\textbf{Six functors:}
For each $X \in \sch[sepft]{\complex}$, the tensor-triangulated category $\D{X}^{\otimes}$ is closed, and, for each morphism $f: X \to Y$ of $\sch[sepft]{\complex}$, there are adjunctions
\[
f^*: \D{Y} \rightleftarrows \D{X}: f_*
\quad\text{and}\quad
f_!: \D{X} \rightleftarrows \D{Y}: f^!
\] 
of triangulated functors, and the functors $f^*$, $f_*$, $f_!$, $f^!$, $\otimes$, $\intmor{}{}$ satisfy the usual barrage of relations, e.g., base change theorems and projection formul\ae. 
\item
\nlabel{desideratum.3}
\textbf{Fiber functors:}
For each $X \in \sch[sepft]{\complex}$, there is a conservative triangulated functor
\[
\omega^*_X:
  \D{X} \to \D[u = {b}, d = {c}]{X^{\tu{an}}, \rational}
\]
that commutes with the six functors.
\item
\nlabel{desideratum.4}
\textbf{Punctual objects:}
There is a fully faithful tensor-triangulated functor
\[
\chi^{*,\otimes}_{\spec{\complex}}:
\D{\spec{\complex}}^{\otimes} \hookrightarrow \D[u = {b}]{\mhsp[\rational]}^{\otimes}
\]
such that $\h^r\fct{\chi^*_{\spec{\complex}}}{\pi_*\1{\D{X}}} \simeq \fct[um = {r}, d = {B}]{H}{X,\rational}$ for each $X \in \sch[sepft]{\complex}$ and each $r \in \integer$.
\item
\nlabel{desideratum.5}
\textbf{Weights:}
For each $X \in \sch[sepft]{\complex}$, $\D{X}$ admits a weight structure compatible with the six functors.
\item
\nlabel{desideratum.6}
\textbf{Realizations:}
For each $X \in \sch[sepft]{\complex}$, there is a triangulated functor
\[
\rho^*_X: 
  \SH[d = {c}]{X} 
    \to \D{X}
\]
that commutes with the six functors, is compatible with weight structures, and factors through the \'etale localization $\SH[d = {c}]{X} \to \SH[u = {\'et}, d = {c}]{X}$.
\item
\nlabel{desideratum.7}
\textbf{t-Structure:}
For each $X \in \sch[sepft]{\complex}$, $\D{X}$ admits a t-structure such that $\omega^*_X$ is t-exact with respect to the perverse t-structure on $\D[u = {b}, d = {c}]{X^{\tu{an}}, \rational}$.
\item
\textbf{Compatibility of t-structure and weight structure:}
\nlabel{desideratum.8}
The intersection of the heart of the weight structure of \nref{desideratum.5} with the heart of the t-structure of \nref{desideratum.7} is a semisimple Abelian category, and the truncation endofunctors $\trun^{\leq r}$ and $\trun^{\geq r}$ associated with the t-structure are weight exact for each $r \in \integer$.
\end{enumerate}
\end{desiderata}

\begin{thm}
\nlabel{mainthm}
There exists a \qcategorically\ enhanced theory of constructible coefficients $\DH[d = {c}]{-}$ satisfying \nref[Desiderata]{desideratum.1} through \varnref{desideratum.6}.
\end{thm}

\begin{proof}
This follows from \nref{ex:DH-coefficient-syst} and \nref{prop:hodge-weights}.
\end{proof}

\subsection*{Forthcoming and future work}

In a forthcoming sequel to this text, we establish Desiderata $(7)$ and $(8)$ for $\DH[d = {c}]{-}$.
After d\'evissage, the proof boils down to establishing a decomposition theorem in $\DH{S}$, i.e., in finding suitable idempotent endomorphisms of $p_*\1{\DH{X}}$ for $p: X \to S$ a projective morphism of $\complex$-schemes of finite type that are compatible with decomposition of the underlying complex of analytic sheaves given by \cite[6.2.5]{Beilinson-Bernstein-Deligne_faisceaux-pervers}. 

Once \nref[Desiderata]{desideratum.7} and \varnref{desideratum.8} have been established, we expect to be able to construct a fully faithful comparison functor $\DH[d = {c}]{X} \to \D[u = {b}]{\mhm{X}}$ compatible with Grothendieck's six operations by first constructing an exact, fully faithful functor from the heart $\DH[d = {c}]{X}^{\heartsuit}$ of the t-structure on $\DH[d = {c}]{X}$ to $\mhm{X}$.
This is work in progress.

\subsection*{Acknowledgements}

The author would like to thank Aravind Asok, Denis-Charles Cisinski, Fr\'ed\'eric D\'eglise, Annette Huber, Shane Kelly, Marc Levine, Marco Robalo, and J\"org Wildeshaus for helpful discussions.
This research was partially supported by the DFG through the SFB Transregio 45 and the SPP 1786.

\subsection*{Notation and conventions}

\textbf{Grothendieck universes:}
We assume that each set is an element of a Grothendieck universe.
We also fix two Grothendieck universes $\mf{U} \in \mf{V}$.
A set will be \emph{small} if it is isomorphic to an element of $\mf{U}$,
\emph{large} if it is isomorphic to an element of $\mf{V}$,
and \emph{very large} otherwise. 
Unless context dictates otherwise, sets, groups, rings, and modules will be small,
and schemes will admit Zariski covers by spectra of small commutative unital rings.

\vspace{.25\baselineskip}

\noindent
\textbf{\Qcategories:}
We will use the language of \qcategories\ as developed in \cite{Lurie_higher-topos, Lurie_higher-algebra}.
We regard each category as a \qcategory\ by identifying it with its nerve.

\vspace{.25\baselineskip}

\noindent
\textbf{Notation for \qcategories\ of \qcategories:}
We will work extensively with the following \qcategories\ of \qcategories:
\begin{longtable}[c]{p{1.1in}>{\raggedright\arraybackslash}p{4.6in}}
$\qcat$, $\QCAT$ & \qcategories\ of small and large \qcategories, respectively \\
\\
$\qcatmon$, $\QCATmon$ & Cartesian symmetric monoidal structure on $\qcat$ and $\QCAT$, respectively \\
$\qcatex$, $\QCATex$ & \qcategories\ of small and large stable \qcategories\ and exact functors \\
$\qcatexmon$, $\QCATexmon$ & symmetric monoidal structure on $\qcatex$ and $\QCATex$, respectively \\
$\pr[u = {L}]$ & \qcategory\ of $\mf{V}$-small \locpres[\mf{U}]\ \qcategories\ and left-adjoint functors\\ 
$\pr[u = {L}, dm = {\kappa}]$ & \qcategory\ of $\mf{V}$-small \locpres[\kappa]\ \qcategories\ and $\kappa$-presentable-object-preserving left-adjoint functors, where $\kappa$ is a small regular cardinal \\
$\pr[u = {L}, dm = {\kappa, \tu{st}}]$ & \qcategory\ of $\mf{V}$-small stable \locpres[\kappa]\ \qcategories\ and $\kappa$-presentable-object-preserving left-adjoint functors for a small regular cardinal $\kappa$ \\
$\pr[u = {L}, d = {st}]$ & \qcategory\ of $\mf{V}$-small stable \locpres[\mf{U}]\ \qcategories\ and left-adjoint functors\\ 
$\pr[um = {\tu{L}, \otimes}, dm = {\spadesuit}]$ & symmetric monoidal structure on $\pr[u = {L}, dm = {\spadesuit}]$ for $\spadesuit \subseteq \brc{\kappa, \tu{st}}$ \\
$\pr[u = {R}]$ & \qcategory\ of $\mf{V}$-small \locpres[\mf{U}]\ \qcategories\ and right-adjoint functors\\ 
$\pr[u = {R}, d = {st}]$ & \qcategory\ of $\mf{V}$-small stable \locpres[\mf{U}]\ \qcategories\ and right-adjoint functors
\end{longtable}

\vspace{.25\baselineskip}

\noindent
\textbf{Symmetric monoidal \qcategories:}
If $\mc{C}^{\otimes}$ is symmetric monoidal structure on one of the \qcategories\ of some flavor of \qcategory\ in the above table, then we identify the \qcategory\ $\calg{\mc{C}^{\otimes}}$ with the \qcategory\ of symmetric monoidal \qcategories\ of that flavor.
For example, $\calg{\pr[um = {\tu{L}, \otimes}, dm = {\kappa, \tu{st}}]}$ is the symmetric monoidal \qcategory\ of stable, \locpres[\kappa]\ symmetric monoidal \qcategories\ and cocontinuous, $\kappa$-presentable-object-preserving symmetric monoidal functors. 

\vspace{.25\baselineskip}

\noindent
\textbf{Standard \qcategorical\ operations:}

\begin{longtable}[c]{p{1.1in}>{\raggedright\arraybackslash}p{4.6in}}
$\1{\mc{C}}$ & unit of the symmetric monoidal \qcategory\ $\mc{C}^{\otimes}$ \\
$\mc{C}_{\kappa}$ & full \subqcategory\ of $\mc{C}$ spanned by the $\kappa$-presentable objects \\
$\mc{C}^{\otimes}\rig$ & full symmetric monoidal \subqcategory\ of $\mc{C}^{\otimes}$ spanned by the $\otimes$-dualizable objects \\
$\mc{C}^{\otimes}$ & a symmetric monoidal \qcategory\ with underlying \qcategory\ $\mc{C}$ \\
$\calg{\mc{C}^{\otimes}}$ & \qcategory\ of commutative algebra objects in the symmetric monoidal \qcategory\ $\mc{C}^{\otimes}$ \\
$\fun{\mc{C}}{\mc{C}}$ & \qcategory\ of functors $\mc{C} \to \mc{D}$ \\
$\fun[u = {ex}]{\mc{C}}{\mc{D}}$ & \qcategory\ of exact functors $\mc{C} \to \mc{D}$ \\
$\fun[u = {L}]{\mc{C}}{\mc{D}}$ & \qcategory\ of left-adjoint functors $\mc{C} \to \mc{D}$ \\
$\fun[u = {R}]{\mc{C}}{\mc{D}}$ & \qcategory\ of right-adjoint functors $\mc{C} \to \mc{D}$ \\
$\ind{\mc{C}}$ & \qcategory\ of ind-objects of $\mc{C}$ \\
$\map[dm = {\mc{C}}]{}{}$ & mapping-space bifunctor associated with $\mc{C}$ \\
$\mod[dm = {A}]{\mc{C}}^{\otimes}$ & symmetric monoidal \qcategory\ of modules over $A \in \calg{\mc{C}^{\otimes}}$ equipped with the symmetric monoidal structure corresponding to the relative tensor product $\prns{-} \otimes_A \prns{-}$ \\
$\intmor[dm = {\mc{C}}]{}{}$ & internal morphisms-object bifunctor associated with the closed symmetric monoidal structure $\mc{C}^{\otimes}$ \\
$\psh{\mc{C}}{\mc{V}}$ & \qcategory\ of $\mc{V}$-valued presheaves on $\mc{C}$ \\
$\hypsh[dm = {\tau}]{\mc{C}}{\mc{V}}$ & \qcategory\ of $\mc{V}$-valued $\tau$-hypersheaves on $\mc{C}$ (\varnref{SH.a2.1.etale})\\
$\yon{\mc{C}}{}$ & \qcategorical\ Yoneda embedding $\mc{C} \hookrightarrow \psh{\mc{C}}{\spc{}}$ \\
\end{longtable}

\vspace{.25\baselineskip}

\noindent
\textbf{Frequently occurring objects:}
\begin{longtable}[c]{p{1.1in}>{\raggedright\arraybackslash}p{4.6in}}
$\DB{X}^{\otimes}$ & modules over the $\rational$-linear Betti cohomology spectrum (\varnref{ex:hodge-module-6ff.2}) \\
$\DH{X}^{\otimes}$ & modules over the $\rational$-linear absolute Hodge spectrum (\varnref{ex:hodge-module-6ff.2})\\
$\varDH{X}^{\otimes}$ & modules over  the $\mhsp[\rational]$-linear absolute Hodge spectrum (\varnref{ex:hodge-module-6ff.2})\\
$\bettimon{}$ & $\rational$-linear Betti cohomology (\varnref{sheaves.13a}) \\
$\enhancedbettimon{}$ & enhanced $\rational$-linear Betti cohomology (\varnref{sheaves.13b}) \\
$\mhsp[\rational]$ & category of polarizable mixed Hodge $\rational$-structures \\
$\mod[dm = {\bs{\Lambda}}]{}$ & category of modules over the commutative unital ring $\bs{\Lambda}$ \\
$\spc{}^{\times}$ & Cartesian symmetric monoidal \qcategory\ of spaces \\
$\spc{}_*^{\wedge}$ & symmetric monoidal \qcategory\ of pointed spaces \\
$\spc[um = {\tau}]{S}^{\times}$ & Cartesian symmetric monoidal \qcategory\ of $\tau$-motivic spaces over $S$ (\varnref{SH.1.D}) \\
$\spc[um = {\tau}]{S, \mc{V}}^{\otimes}$ & symmetric monoidal \qcategory\ of $\mc{V}^{\otimes}$-linear $\tau$-motivic spaces over $S$ (\varnref{SH.1.D}) \\
$\spt{}^{\wedge}$ & symmetric monoidal \qcategory\ of $\sphere^1$-spectra with the smash product \\
$\spt[um = {\tau}, dm = {\tatesphere}]{S}^{\wedge}$ & symmetric monoidal \qcategory\ of $\tau$-motivic $\tatesphere$-spectra over $S$ (\varnref{SH.1.J}) \\
$\spt[um = {\tau}, dm = {\tatesphere}]{S, \mc{V}}^{\otimes}$ & symmetric monoidal \qcategory\ of $\mc{V}^{\otimes}$-linear $\tau$-motivic $\tatesphere$-spectra over $S$ (\varnref{SH.1.J}) \\
$\tatesphere$ & motivic Tate sphere (\varnref{SH.1.I}) \\
$\yona[um = {\tau}]{S}{X}$ & image of the representable presheaf $\yon{S}{X} \in \psh{\sm[ft]{S}}{\spc{}}$ under the localization functor $\psh{\sm[ft]{S}}{\spc{}} \to \spc[um = {\tau}]{S}$ (\varnref{SH.1.F})
\end{longtable}

\section{Hypersheaves}
\nlabel{sheaves}

\setcounter{thm}{-1}

\begin{notation}
\label{sheaves.0}
Throughout this section, we fix the following:
\begin{itemize}
\item
$\mc{C}$, a small category;
\item
$\tau$, a Grothendieck topology on $\mc{C}$.
\end{itemize}
\end{notation}

\begin{motivation}
The goal of this section is to collect some basic facts regarding hypercovers and cohomological descent for presheaves with values a given \qcategory.
The fundamental example of such a presheaf is that of an enhanced version $\enhancedbetti{}$ of the singular-cochain-complex functor $\betti{}$ that assigns to each smooth $\complex$-scheme $X$ of finite type a mixed Hodge complex that encodes the mixed Hodge structure on the Betti cohomology of $X$.
By \cite[3.2]{Drew_rectification-of-Deligne's}, we may regard this functor as a presheaf on $\sm[ft]{\complex}$ with values in the symmetric monoidal bounded derived \qcategory\ of polarizable mixed Hodge $\rational$-structures.
In this section, we check that $\enhancedbetti{}$ inherits the property of descent with respect to \'etale hypercovers (\nref{sheaves.14}) from $\betti{}$.
\end{motivation}

\begin{summary}
Let $\mc{V}$ be a \qcategory.
\begin{itemize}
\item
We begin by recalling the basic definitions, and extend the results of \cite{Dugger-Hollander-Isaksen_hypercovers-and-simplicial} for $\spc{}$-valued presheaves to $\mc{V}$-valued presheaves.
\item
In \nref{sheaves.6}, we show that the localization of the \qcategory\ $\psh{\mc{C}}{\mc{V}}$ with respect to the class of $\tau$-hypercovers is equivalent to the localization of $\psh{\mc{C}}{\mc{V}}$ with respect to the Joyal-Jardine $\tau$-local equivalences, generalizing a result of \cite{Choudhury-Gallauer_homotopy-theory}.
This will allow us to compare our constructions in the following sections with those of \cite{Ayoub_six-operationsII}.
\item
In \nref{prop:descent}, we establish a criterion that implies that the enhanced singular-cochain-complex functor $\enhancedbetti{}$ has the property of descent with respect to \'etale hypercovers.
\item
In \nref{sheaves.12}, we establish sufficient conditions in the cases of interest on the \qcategory\ $\mc{V}$ for the localization of $\psh{\mc{C}}{\mc{V}}$ with respect to $\tau$-hypercovers to be \locpres[\aleph_0].
\end{itemize}
\end{summary}

\begin{defn}
\nlabel{sheaves.1}
Let $\mc{V}$ be a \locpres\ \qcategory, and $F \in \psh{\mc{C}}{\mc{V}}$.
\begin{enumerate}
\item
\nlabel{sheaves.1.1}
Recall that $\yon{\mc{C}}{}: \mc{C} \to \psh{\mc{C}}{\spc{}}$ denotes the Yoneda embedding of \cite[\S5.1.3]{Lurie_higher-topos}.
Let $X \in \mc{C}$, let $p: U_{\bullet} \to \yon{\mc{C}}{X}$ be a $\tau$-hypercover, regarded as a morphism from $U_{\bullet}$ to the constant simplicial object with value $\yon{\mc{C}}{X}$, and let $p_+: \prns{\Delta_+}\op \to \psh{\mc{C}}{\spc{}}$ be the associated augmented simplicial object.
We say that $F$ is \emph{$p$-local} if it satisfies the following equivalent conditions:
\begin{enumerate}
\item
\nlabel{sheaves.1.1.a}
the composite $\overline{F}\prns{p_+}\op: \Delta_+ \to \mc{V}$ of $\prns{p_+}\op$ with the right Kan extension $\overline{F}: \psh{\mc{C}}{\spc{}}\op \to  \mc{V}$ of $F$ along $\prns{\yon{\mc{C}}{}}\op$, i.e., the functor induced by the universal property of $\psh{\mc{C}}{\spc{}}$ given in \cite[5.1.5.6]{Lurie_higher-topos}, is a limit diagram in $\mc{V}$;
\item
\nlabel{sheaves.1.1.b}
the canonical composite morphism
\[
\fct{F}{X}
  \simeq \fct{\overline{F}}{\yon{\mc{C}}{X}}
  \to \lim_{\brk{n} \in \Delta} \fct{\overline{F}}{U_n}
\]
is an equivalence;
\item
\nlabel{sheaves.1.1.c}
$\overline{F}$ sends the canonical morphism 
\[
\colim_{n \in \Delta\op} p_{n}:
  \colim_{\brk{n} \in \Delta\op} U_n
    \to \colim_{n \in \Delta\op} \yon{\mc{C}}{X} \simeq \yon{\mc{C}}{X}
\]
to an equivalence;
\item
\nlabel{sheaves.1.1.d}
$\overline{F}$ factors as $\psh{\mc{C}}{\spc{}}\op \to \psh{\mc{C}}{\spc{}}\brk{\brc{\colim_{\brk{n} \in \Delta\op}p_n}^{-1}}\op \to \mc{V}$. 
\end{enumerate}
Indeed, the equivalence of \nref{sheaves.1.1.a} and \nref{sheaves.1.1.b} follows from the definition of a limit diagram;
the equivalence of \nref{sheaves.1.1.b} and \nref{sheaves.1.1.c} follows from the remark that $\overline{F}$ is continuous, so $\fct{\overline{F}}{\colim_{n \in \Delta\op} U_n} \simeq \lim_{n \in \Delta} \fct{\overline{F}}{U_n}$; and 
the equivalence of \nref{sheaves.1.1.c} and \nref{sheaves.1.1.d} follows from the universal property of the localization.
\item
\nlabel{SH.a2.1.etale}
We say that $F$ is \emph{$\tau$-local} or that $F$ is a \emph{$\mc{V}$-valued $\tau$-hypersheaf} if it satisfies the following equivalent conditions:
\begin{enumerate}
\item
\nlabel{sheaves.1.2.a}
$F$ is $p$-local for each $\tau$-hypercover $p$;
\item 
\nlabel{sheaves.1.2.b}
$\overline{F}$ factors as $\psh{\mc{C}}{\spc{}}\op \to \psh{\mc{C}}{\spc{}}\brk{\mf{W}_{\tau}^{-1}}\op \to \mc{V}$, where $\mf{W}_{\tau}$ is the class of morphisms of the form $\colim_{n \in \Delta\op} p_{n}$ as in \nref{sheaves.1.1.c} with $p$ a $\tau$-hypercover. 
\end{enumerate}
\item
We let $\hypsh[dm = {\tau}]{\mc{C}}{\mc{V}} \subseteq \psh{\mc{C}}{\mc{V}}$ denote the full \subqcategory\ spanned by the $\mc{V}$-valued $\tau$-hypersheaves.
\end{enumerate}
\end{defn}

\begin{rmk}
\nlabel{sheaves.2}
Let $\mc{V}$ be a \locpres\ \qcategory.
In general, the class of $\tau$-hypercovers is not small, so there are some set-theoretic issues with the previous definition.
Specifically, in \nref{sheaves.1.2.a}, we have quantified over a proper class.
Also, in \nref{sheaves.1.2.b}, localization with respect to the class $\mf{W}_{\tau}$ is not \emph{a priori} well behaved.
By \cite[Corollary~7.1]{Dugger-Hollander-Isaksen_hypercovers-and-simplicial} or \cite[6.5.2.8]{Lurie_higher-topos}, however, there exists a small, dense set $\mf{G}_{\tau}$ of $\tau$-hypercovers:
a presheaf $F: \mc{C}\op \to \mc{V}$ is $\tau$-local if and only if it is $p$-local for each hypercover $p$ in this small, dense set $\mf{G}_{\tau}$;
and the localization of $\psh{\sm[ft]{S}}{\spc{}}$ with respect to $\mf{W}_{\tau}$ is equivalent to the localization  with respect to the small set of morphisms $\colim_{n \in \Delta\op} p_n$ for each $p \in \mf{G}_{\tau}$.
\end{rmk}

\begin{defn}
\label{sheaves.3a}
Let $\mc{V}$ be a \locpres\ \qcategory.
By \nref{sheaves.2}, the inclusion $\presheaf{\tau}{}: \hypsh[dm = {\tau}]{\mc{C}}{\mc{V}} \hookrightarrow \psh{\mc{C}}{\mc{V}}$ admits a left adjoint 
\[
\loc[dm = {\tau}]{}:
  \psh{\mc{C}}{\spc{}}
    \to \psh{\mc{C}}{\spc{}}\brk{\mf{W}_{\tau}^{-1}}
    \simeq \hypsh[dm = {\tau}]{\mc{C}}{\spc{}},
\]
which we refer to as the \emph{$\tau$-hypersheafification} functor.
This localization $\loc[dm = {\tau}]{}$ is left exact by \cite[6.5.2.8, 6.2.1.6, 6.2.2.7]{Lurie_higher-topos}.
A morphism $f: F \to F'$ of $\psh{\mc{C}}{\mc{V}}$ is a \emph{$\tau$-local equivalence} if its image under $\sheaf{\tau}{}$ is an equivalence.
\end{defn}

\begin{prop}
\nlabel{sheaves.3}
Let $\mc{V}$ be a \locpres\ \qcategory.
\begin{enumerate}
\item
\nlabel{sheaves.3.1}
There is a canonical equivalence
\[
\hypsh[dm = {\tau}]{\mc{C}}{\mc{V}}
  \simeq \hypsh[dm = {\tau}]{\mc{C}}{\spc{}} \otimes \mc{V}.
\]
In particular, the assignment $\mc{V} \mapsto \hypsh[dm = {\tau}]{\mc{C}}{\mc{V}}$ underlies a functor $\pr[u = {L}] \to \pr[u = {L}]$.
\item
\nlabel{sheaves.3.2}
If $\mc{V}$ is \locpres[\aleph_0], then there is a canonical equivalence
\[
\fun[u = {lex}]{\prns{\mc{V}_{\aleph_0}}\op}{\hypsh[dm = {\tau}]{\mc{C}}{\spc{}}}
  \simeq \hypsh[dm = {\tau}]{\mc{C}}{\mc{V}},
\]
where $\fun[u = {lex}]{\mc{D}}{\mc{D}'} \subseteq \fun{\mc{D}}{\mc{D}'}$ denotes the full \subqcategory\ spanned by the left-exact functors.
\item
\nlabel{sheaves.3.3}
Let $\mc{V}$ and $\mc{W}$ be \locpres[\aleph_0]\ \qcategories, and $\phi^*: \mc{V} \to \mc{W}$ a cocontinuous functor that preserves $\aleph_0$-presentable objects.
There is an essentially commutative diagram
\[
\begin{tikzcd}
\fun[u = {lex}]{\prns{\mc{W}_{\aleph_0}}\op}{\hypsh[dm = {\tau}]{\mc{C}}{\spc{}}}
\ar[r, "\prns{\phi^*}\op" below]
\ar[d]
&
\fun[u = {lex}]{\prns{\mc{V}_{\aleph_0}}\op}{\hypsh[dm = {\tau}]{\mc{C}}{\spc{}}}
\ar[d]
\\
\hypsh[dm = {\tau}]{\mc{C}}{\mc{W}}
\ar[r, "\phi_*" above]
&
\hypsh[dm = {\tau}]{\mc{C}}{\mc{V}}
\end{tikzcd}
\]
with the following properties:
\begin{itemize}
\item
the vertical arrows are the equivalences of \nref{sheaves.3.2};
\item
$\prns{\phi^*}\op$ is induced by composition with the left-exact functor $\prns{\phi^*}\op: \prns{\mc{V}_{\aleph_0}}\op \to \prns{\mc{W}_{\aleph_0}}\op$; and 
\item
the lower horizontal arrow $\phi_*$ is right adjoint to the cocontinuous functor induced by $\phi^*$.
\end{itemize}
\end{enumerate}
\end{prop}

\begin{proof}
There are canonical equivalences
\begin{align*}
\psh{\mc{C}}{\spc{}} \otimes \mc{V}
  &\simeq \fun[u = {R}]{\psh{\mc{C}}{\spc{}}\op}{\mc{V}} 
  &&\text{\cite[4.8.1.17]{Lurie_higher-algebra}} \\
  &\simeq \fun[u = {L}]{\psh{\mc{C}}{\spc{}}}{\mc{V}\op}\op 
  &&\text{\cite[5.2.6.2]{Lurie_higher-topos}}\\
  &\simeq \fun{\mc{C}}{\mc{V}\op}\op \simeq \psh{\mc{C}}{\mc{V}}
  &&\text{\cite[5.5.2.10, 5.1.5.6]{Lurie_higher-topos}}.\\
\end{align*}
Under the equivalence
$
\psh{\mc{C}}{\mc{V}}
  \simeq \fun[u = {R}]{\psh{\mc{C}}{\spc{}}\op}{\mc{V}}
$,
the full \subqcategory\ $\hypsh[dm = {\tau}]{\mc{C}}{\mc{V}}$ corresponds to the full \subqcategory\ spanned by the right-adjoint functors that factor through the opposite of the localization functor $\loc[dm = {\tau}]{}$.
The latter \subqcategory\ is the essential image of the fully faithful functor
\[
\fun[u = {R}]{\hypsh[dm = {\tau}]{\mc{C}}{\spc{}}\op}{\mc{V}}
  \hookrightarrow \fun[u = {R}]{\psh{\mc{C}}{\spc{}}\op}{\mc{V}}
\]
given by composition with $\prns{\loc[dm = {\tau}]{}}\op$.
By \nref{sheaves.2}, $\hypsh[dm = {\tau}]{\mc{C}}{\spc{}}$ is \locpres, so we have a canonical equivalence 
\[
\fun[u = {R}]{\hypsh[dm = {\tau}]{\mc{C}}{\spc{}}\op}{\mc{V}}
  \simeq \hypsh[dm = {\tau}]{\mc{C}}{\spc{}} \otimes \mc{V}
\]
and \nref[Claim]{sheaves.3.1} follows.

Suppose $\mc{V}$ is \locpres[\aleph_0].
By \cite[5.3.5.11]{Lurie_higher-topos}, the canonical cocontinuous functor $\ind{\mc{V}_{\aleph_0}} \to \mc{V}$ is an equivalence.
By \cite[4.2.3.12, 5.3.5.10]{Lurie_higher-topos}, restriction along $\prns{\mc{V}_{\aleph_0}}\op \hookrightarrow \mc{V}\op$ induces an equivalence
\[
\fun[u = {R}]{\mc{V}\op}{\hypsh[dm = {\tau}]{\mc{C}}{\spc{}}}
  \simeq \fun[u = {lex}]{\prns{\mc{V}_{\aleph_0}}\op}{\hypsh[dm = {\tau}]{\mc{C}}{\spc{}}}.
\]
\nref[Claim]{sheaves.3.2} now follows from \nref{sheaves.3.1} and \cite[4.8.1.17]{Lurie_higher-algebra}.

Let $\phi_*: \mc{W} \to \mc{V}$ be right adjoint to $\phi^*$.
\nref[Claim]{sheaves.3.3} follows from the essentially commutativity of the diagram
\[
\begin{tikzcd}
\fun[u = {lex}]{\prns{\mc{W}_{\aleph_0}}\op}{\hypsh[dm = {\tau}]{\mc{C}}{\spc{}}}
\ar[r, "\prns{\phi^*}\op" below]
\ar[d, "\varepsilon" left]
&
\fun[u = {lex}]{\prns{\mc{V}_{\aleph_0}}\op}{\hypsh[dm = {\tau}]{\mc{C}}{\spc{}}}
\ar[d, "\varepsilon" right]
\\
\fun[u = {R}]{\mc{W}\op}{\hypsh[dm = {\tau}]{\mc{C}}{\spc{}}}
\ar[r, "\prns{\phi^*}\op" below]
\ar[d, "\alpha" left]
&
\fun[u = {R}]{\mc{V}\op}{\hypsh[dm = {\tau}]{\mc{C}}{\spc{}}}
\ar[d, "\alpha" right]
\\
\fun[u = {L}]{\hypsh[dm = {\tau}]{\mc{C}}{\spc{}}}{\mc{W}\op}\op
\ar[r, "\prns{\phi_*}\op" below]
\ar[d, "\op" left]
&
\fun[u = {L}]{\hypsh[dm = {\tau}]{\mc{C}}{\spc{}}}{\mc{V}\op}\op
\ar[d, "\op" right]
\\
\fun[u = {R}]{\hypsh[dm = {\tau}]{\mc{C}}{\spc{}}\op}{\mc{W}}
\ar[r, "\phi_*" above]
&
\fun[u = {R}]{\hypsh[dm = {\tau}]{\mc{C}}{\spc{}}\op}{\mc{V}},
\end{tikzcd}
\]
in which $\varepsilon$ is given by right Kan extension
and $\alpha$ by assigning to each right adjoint functor the corresponding left adjoint.
\end{proof}

\begin{prop}
\nlabel{sheaves.4}
Let $\mc{V}$ be a \locpres[\aleph_0]\ \qcategory, $f: F \to F'$ a morphism of $\psh{\mc{C}}{\mc{V}}$, and $\pt \in \mc{C}$ a final object.
The following are equivalent:
\begin{enumerate}
\item
\nlabel{sheaves.4.1}
for each $n \in \integer_{> 0}$, each $0$-simplex $x$ of $\map[dm = {\mc{V}}]{V}{\fct{F}{\pt}}$, and each $V \in \mc{V}_{\aleph_0}$,
the $\tau$-sheafifications of the morphisms of presheaves of sets
\begin{align*}
\pi_0\prns{\map[dm = {\mc{V}}]{V}{\fct{F}{-}}}
  &\to \pi_0\prns{\map[dm = {\mc{V}}]{V}{\fct{F'}{-}}}\\
\text{and}\quad
\pi_n\prns{\map[dm = {\mc{V}}]{V}{\fct{F}{-}}, x}
  &\to
\pi_n\prns{\map[dm = {\mc{V}}]{V}{\fct{F'}{-}}, \fct{f}{x}}
\end{align*}
induced by $f$ are isomorphisms of $\tau$-sheaves;
\item
\nlabel{sheaves.4.2}
for each $V \in \mc{V}_{\aleph_0}$, the morphism
\[
  \sheaf{\tau}{\map[dm = {\mc{V}}]{V}{\fct{F}{-}}}
  \to \sheaf{\tau}{\map[dm = {\mc{V}}]{V}{\fct{F'}{-}}}
\]
of $\hypsh[dm = {\tau}]{\mc{C}}{\spc{}}$ induced by $f$ is an equivalence; and
\item
\nlabel{sheaves.4.3}
the morphism $f: F \to F'$ of $\psh{\mc{C}}{\mc{V}}$ is a $\tau$-local equivalence.
\end{enumerate}
\end{prop}

\begin{proof}
For $\mc{V} = \spc{}$, the equivalence of \varnref{sheaves.4.1} and \varnref{sheaves.4.3} follows from \cite[Theorem~1.2]{Dugger-Hollander-Isaksen_hypercovers-and-simplicial}.

For arbitrary $\mc{V}$, the equivalence of \varnref{sheaves.4.1} and \varnref{sheaves.4.2} follows from the special case in which $\mc{V} = \spc{}$, combined with the fact that the formation homotopy sheaves of sets commutes with $\tau$-sheafification up to natural equivalence.
Indeed, the homotopy sheaves of sets of an object of $\hypsh[dm = {\tau}]{\mc{C}}{\spc{}}$ are, by definition, the homotopy sheaves of the underlying object of $\psh{\mc{C}}{\spc{}}$.
The special case in which $\mc{V} = \spc{}$ results from the equivalence of \varnref{sheaves.4.1} and \varnref{sheaves.4.3} for $\mc{V} = \spc{}$, which we have already mentioned, and the equivalence of \varnref{sheaves.4.2} and \varnref{sheaves.4.3}, which we will now establish for general $\mc{V}$.

Let $V \in \mc{V}_{\aleph_0}$.
This object is classified by a $0$-simplex $\chi_V: \Delta^0 \to \mc{V}$.
Since $\mc{V}$ is cocomplete, $\chi_V$ factors through an essentially unique cocontinuous functor $\overline{\chi}_V: \spc{} \simeq \psh{\Delta^0}{\spc{}} \to \mc{V}$. 
The object $\Delta^0 \in \spc{}$ generates the full \subqcategory\ $\spc{}_{\aleph_0} \subseteq \spc{}$ under iterated finite colimits and retracts, so $\overline{\chi}_V$ restricts to a right-exact functor $\tilde{\chi}_V: \spc{}_{\aleph_0} \to \mc{V}_{\aleph_0}$.
The opposite $\prns{\tilde{\chi}_V}\op: \prns{\spc{}_{\aleph_0}}\op \to \prns{\mc{V}_{\aleph_0}}\op$ is left exact.

We have an essentially commutative square
\begin{equation}
\nlabel{sheaves.4.a}
\begin{tikzcd}
\fun[u = {lex}]{\prns{\mc{V}_{\aleph_0}}\op}{\psh{\mc{C}}{\spc{}}}
\ar[r]
\ar[d]
&
\fun[u = {lex}]{\prns{\spc{}_{\aleph_0}}\op}{\psh{\mc{C}}{\spc{}}}
\ar[d]
\\
\fun[u = {lex}]{\prns{\mc{V}_{\aleph_0}}\op}{\hypsh[dm = {\tau}]{\mc{C}}{\spc{}}}
\ar[r]
&
\fun[u = {lex}]{\prns{\spc{}_{\aleph_0}}\op}{\hypsh[dm = {\tau}]{\mc{C}}{\spc{}}}
\end{tikzcd}
\end{equation}
in which the vertical arrows are given by composition with the left-exact localization $\sheaf{\tau}{}$, and the horizontal arrows by composition with $\prns{\tilde{\chi}_V}\op$.

By \nref{sheaves.3.3}, the upper and lower horizontal arrows of \nref{sheaves.4.a} correspond to the functors
\begin{align*}
F &\mapsto \map[dm = {\mc{V}}]{V}{\fct{F}{-}}:
  \psh{\mc{C}}{\mc{V}} \to \psh{\mc{C}}{\spc{}} \\
\text{and}\quad
F &\mapsto \map[dm = {\mc{V}}]{V}{\fct{F}{-}}:
  \hypsh[dm = {\tau}]{\mc{C}}{\mc{V}} \to \hypsh[dm = {\tau}]{\mc{C}}{\spc{}},
\end{align*}
respectively, under the equivalences of \nref{sheaves.3.2}.
Indeed, the right adjoint of $\overline{\chi}_V: \spc{} \to \mc{V}$ is $\map[dm = {\mc{V}}]{V}{-}$.

The commutativity of \nref{sheaves.4.a} therefore implies that, for each $V \in \mc{V}_{\aleph_0}$, the morphism $\sheaf{\tau}{\map[dm = {\mc{V}}]{V}{f}}$ is an equivalence in $\hypsh[dm = {\tau}]{\mc{C}}{\spc{}}$ if and only if $\map[dm = {\mc{V}}]{V}{\sheaf{\tau}{f}}$ is an equivalence. 
The family of functors $\map[dm = {\mc{V}}]{V}{-}$ indexed by $V \in \mc{V}_{\aleph_0}$ is conservative, so $\map[dm = {\mc{V}}]{V}{\sheaf{\tau}{f}}$ is an equivalence for each $V \in \mc{V}_{\aleph_0}$ if and only if $\sheaf{\tau}{f}: \sheaf{\tau}{F} \to \sheaf{\tau}{F'}$ is an equivalence of $\hypsh[dm = {\tau}]{\mc{C}}{\mc{V}}$, i.e., if and only if $f$ is a $\tau$-local equivalence. 
\end{proof}

\begin{rmk}
\label{sheaves.6}
Let $\mb{M}$ be a combinatorial model category that satisfies the conditions of \cite[4.4.23]{Ayoub_six-operationsII}, and whose underlying \qcategory\ $\mc{M}$ is \locpres[\aleph_0].
It follows from \nref{sheaves.4} that the $\tau$-local projective \resp{injective} model structure on $\psh{\mc{C}}{\mb{M}}$ defined in \cite[4.4.34]{Ayoub_six-operationsII}, whose weak equivalences are the morphisms satisfying \nref{sheaves.4.1}, is the left Bousfield localization of the projective \resp{injective} model structure with respect to $\mf{W}_{\tau}$.
Its underlying \qcategory\ is therefore canonically equivalent to $\hypsh[dm = {\tau}]{\mc{C}}{\mc{M}}$.

In particular, \nref{sheaves.4} generalizes \cite[Theorem~5.11]{Choudhury-Gallauer_homotopy-theory}, which covers the special case in which $\mb{M}$ denotes the category of unbounded cochain complexes of modules over a commutative, unital ring equipped with the projective model structure.
\end{rmk}

\begin{prop}
\nlabel{sheaves.7}
Let $\prns{\mc{C}, \tau}$ be a Verdier site \tu{(\cite[Definition~9.1]{Dugger-Hollander-Isaksen_hypercovers-and-simplicial})} satisfying Conditions $(1)$, $(2)$, and $(3)$ of \tu{\cite[\S10]{Dugger-Hollander-Isaksen_hypercovers-and-simplicial}},
$\mc{V}$ a \locpres\ \qcategory, and $F: \mc{C}\op \to \mc{V}$ a functor.
The following are equivalent:
\begin{enumerate}
\item
\nlabel{sheaves.7.1}
$F$ is $\tau$-local;
\item
\nlabel{sheaves.7.2}
$F$ preserves finite products and is local with respect to internal $\tau$-hypercovers in the sense of \tu{\cite[Definition~10.1]{Dugger-Hollander-Isaksen_hypercovers-and-simplicial}}.
\end{enumerate}
\end{prop}

\begin{proof}
If $\mc{V} = \spc{}$, then this is found in \cite[Theorem~10.2]{Dugger-Hollander-Isaksen_hypercovers-and-simplicial},
where it is proved that a simplicial presheaf is $\tau$-local if and only if it sends the following classes of morphisms to equivalences in $\spc{}$:
the class $\mf{W}_{\tau,\tu{int}}$ consisting of the morphisms of the form 
\[
\colim_{\brk{n} \in \Delta\op} \yon{S}{p_n}:
  \colim_{\brk{n} \in \Delta\op} \yon{S}{U_n} 
    \to \colim_{\brk{n} \in \Delta\op} \yon{S}{X} = \yon{S}{X};
\]
for each internal $\tau$-hypercover $p_{\bullet}: U_{\bullet} \to X$; and
the class $\mf{W}_{\amalg}$ consisting of the canonical morphisms of the form $\coprod_{\alpha = 1}^r \yon{S}{X_{\alpha}} \to \yon{S}{X}$ for each $X = \coprod_{\alpha = 1}^r X_{\alpha}$ in $\mc{C}$ with $r \in \mb{Z}_{\geq 0}$.

The general case follows from the previous case:
the right Kan extension $\overline{F}: \psh{\mc{C}}{\spc{}}\op \to \mc{V}$ of $F$ along $\prns{\yon{\mc{C}}{}}\op$ factors through $\prns{\loc[dm = {\tau}]{}}\op$ if and only if it sends the morphisms of $\mf{W}_{\tau,\tu{int}}$ and $\mf{W}_{\amalg}$ to equivalences.
\end{proof}

\begin{rmk}
In particular, if $S$ is a Noetherian scheme of finite dimension and $\tau$ denotes the Zariski, Nisnevich, or \'etale topology on $\sm[ft]{S}$ or $\sch[ft]{S}$, then \nref{sheaves.7} applies to $\mc{V}$-valued presheaves on the Grothendieck site $\prns{\sm[ft]{S}, \tau}$ as remarked in \cite[p.~41]{Dugger-Hollander-Isaksen_hypercovers-and-simplicial}.
\end{rmk}

\begin{defn}
\nlabel{SH.10}
We say that a symmetric monoidal \qcategory\ $\mc{V}^{\otimes}$ is: 
\begin{enumerate}
\item
\nlabel{SH.10.1}
\emph{rigid} if each object of $\mc{V}$ is $\otimes$-dualizable (\cite[4.6.1.12]{Lurie_higher-algebra}); and
\item
\nlabel{SH.10.2}
\emph{ind-rigid} if $\mc{V}^{\otimes}$ is \locpres[\aleph_0] and the $\aleph_0$-presentable objects of $\mc{V}$ are precisely the $\otimes$-dualizable ones.
\end{enumerate}
\end{defn}

\begin{ex}
\nlabel{SH.10.ex}
The following examples of ind-rigidity will play an essential role in the sequel.
\begin{enumerate}
\item
\nlabel{SH.10.ex.1}
The symmetric monoidal \locpres\ \qcategory\ $\spt[um = {\wedge}]{}$ of $\sphere^1$-spectra in the \qcategory\ $\spc{}$ of spaces is ind-rigid.
\item
\nlabel{SH.10.ex.2}
The category $\ind{\mhsp[\rational]}$ of ind-objects of the Abelian category of polarizable mixed $\rational$-Hodge structures is Grothendieck Abelian.
Let $\D{\ind{\mhsp[\rational]}}$ denote its unbounded derived \qcategory, i.e., the localization of the category of unbounded complexes in $\ind{\mhsp[\rational]}$ with respect to the class of quasi-isomorphisms.
By \cite[4.7]{Drew_verdier-quotients}, $\D{\ind{\mhsp[\rational]}}$ underlies an ind-rigid symmetric monoidal \qcategory\ with the tensor product inherited from $\mhspmon[\rational]]$.
\item
By a similar argument, for each field $\KK$, the symmetric monoidal derived \qcategory\ $\D{\mod[dm = {\KK}]{}}^{\otimes}$ of unbounded cochain complexes of $\KK$-modules is also ind-rigid.
\end{enumerate}
\end{ex}

\begin{rmk}
\nlabel{rep.2}
If the symmetric monoidal \qcategory\ $\mc{V}^{\otimes}$ is rigid, then, as observed in \cite[4.6]{Drew_rectification-of-Deligne's}, the $\otimes$-duality involution $\prns{-}^{\vee} \coloneqq \intmor{-}{\1{\mc{V}}}: \mc{V}\op \to \mc{V}$ underlies a symmetric monoidal equivalence $\prns{-}^{\vee, \otimes}: \mc{V}\op[\otimes] \isom \mc{V}^{\otimes}$, where $\mc{V}\op[\otimes]$ is the symmetric monoidal structure of \tu{\cite[2.4.2.7]{Lurie_higher-algebra}}.
\end{rmk}

\begin{prop}
\nlabel{prop:descent}
Consider the following data:
\begin{itemize}
\item
$\prns{\mc{C}, \tau}$, a Verdier site satisfying Conditions $(1)$, $(2)$, and $(3)$ of \tu{\cite[\S10]{Dugger-Hollander-Isaksen_hypercovers-and-simplicial}};
\item
$\mc{V}^{\otimes}$, a small rigid symmetric monoidal \qcategory;
\item
$F: \mc{C}\op \to \mc{V}$, a functor;
\item
$\check{F}: \mc{C} \to \mc{V}$, the composite of $F\op$ with the equivalence $\prns{-}^{\vee}: \mc{V}\op \isom \mc{V}$ of \nref{rep.2};
\item
$\iota: \mc{V} \hookrightarrow \ind{\mc{V}}$ the Yoneda embedding; and
\item
$p: U_{\bullet} \to X$, an internal $\tau$-hypercover \tu{(\cite[Definition~10.1]{Dugger-Hollander-Isaksen_hypercovers-and-simplicial})}, corresponding to an augmented simplicial diagram $p_+: \prns{\Delta_+}\op \to \mc{C}$.
\end{itemize}
The following are equivalent:
\begin{enumerate}
\item
\nlabel{prop:descent.1}
$F\prns{p_+}\op: \Delta_+ \to \mc{V}$ is a limit diagram;
\item
\nlabel{prop:descent.2}
$\check{F}p_+: \prns{\Delta_+}\op \to \mc{V}$ is a colimit diagram;
\item
\nlabel{prop:descent.3}
$\iota F: \mc{C}\op \to \ind{\mc{V}}$ is $p$-local.
\end{enumerate}
Moreover, if $\omega: \ind{\mc{V}} \to \mc{W}$ is a conservative functor that preserves colimits indexed by $\Delta$, and if $\omega \iota \check{F}p_+: \prns{\Delta_+}\op \to \mc{W}$ is a colimit diagram, then $F\prns{p_+}\op$ is a limit diagram.
\end{prop}

\begin{proof}
Equivalence of \nref[Conditions]{prop:descent.1} and \varnref{prop:descent.2} follows from the remark that the anti-equivalence $\prns{-}^{\vee}$ exchanges colimit diagrams with limit diagrams.

Suppose \nref{prop:descent.3} is satisfied.
The coaugmented cosimplicial object $\iota F \prns{p_+}\op: \Delta_+ \to \ind{\mc{V}}$ factors through $\iota$, so if it is a limit diagram in $\ind{\mc{V}}$, then it is \emph{a fortiori} a limit diagram in $\mc{V}$ by \cite[1.2.13.7]{Lurie_higher-topos}, and \varnref{prop:descent.1} follows.
Conversely, suppose \nref{prop:descent.1} is satisfied.
By definition of $\ind{\mc{V}}$, there is an essentially commutative triangle of fully faithful functors
\[
\begin{tikzcd}
\mc{V}
\ar[r, "\iota" below]
\ar[rr,%
  "\yon{\mc{V}}{}" above,%
  start anchor = north east,%
  end anchor = north west,%
  bend left = 20pt]
&
\ind{\mc{V}}
\ar[r, "i" below]
&
\psh{\mc{V}}{\spc{}}.
\end{tikzcd}
\]
As $\yon{\mc{V}}{}$ preserves small limits that exist in $\mc{V}$ by \cite[5.1.3.2]{Lurie_higher-topos},
$\yon{\mc{V}}{}F\prns{p_+}\op \simeq i \iota F \prns{p_+}\op$ is a limit diagram in $\psh{\mc{V}}{\spc{}}$. 
Full faithfulness of $i$ then implies that $\iota F\prns{p_+}\op$ is a limit diagram in $\ind{\mc{V}}$.

Consider the last assertion.
By the equivalence of \nref[Conditions]{prop:descent.1} and \varnref{prop:descent.2}, it suffices to show that $\check{F}p_+$ is a colimit diagram.
Since $\ind{\mc{V}}$ is cocomplete, and $\omega$ is conservative and preserves $\Delta$-indexed colimits by hypothesis, we find that $\omega \iota \check{F} p_+$ is a colimit diagram if and only if $\iota \check{F} p_+$ is.
The claim now follows from the remark that if $\iota \check{F} p_+$ is a colimit diagram, then $\check{F} p_+$ is too (\cite[1.2.13.7]{Lurie_higher-topos}).
\end{proof}

\begin{defn}
\nlabel{sheaves.13}
The construction of a six-functor formalism for mixed Hodge theory introduced in \nref{ex:DH-coefficient-syst} below begins with the following presheaves.
\begin{enumerate}
\item
\nlabel{sheaves.13a}
A simplified version of \cite[3.6]{Drew_rectification-of-Deligne's} provides a symmetric monoidal functor $\bettimon{}: \prns{\sm[ft]{\complex}}\op[\amalg] \to \D{\mod[dm = {\rational}]{}}^{\otimes}$ described informally by assigning to each object $X$ a cochain complex equivalent to the complex of $\rational$-linear singular cochains on $X^{\tu{an}}$. 
That this functor is symmetric monoidal amounts to the K\"unneth formula for rational Betti cohomology.
\item
\nlabel{sheaves.13b}
With $\D{\ind{\mhsp[\rational]}}^{\otimes}$ as in \nref{SH.10.ex.2},
there is a conservative, cocontinuous symmetric monoidal functor 
\[
\omega^{*, \otimes}:
  \D{\ind{\mhsp[\rational]}}^{\otimes}
    \to \D{\mod[dm = {\rational}]{}}^{\otimes}
\]
that assigns to a complex of ind-objects of $\mhsp[\rational]$ the complex of underlying $\rational$-modules by \cite[1.6]{Drew_rectification-of-Deligne's}.
It follows from \cite[2.5, 2.7, 3.6]{Drew_rectification-of-Deligne's} that there exists a presheaf
\[
\enhancedbettimon{}:
  \prns{\sm[ft]{\complex}}\op[\amalg] \to \D{\ind{\mhsp[\rational]}}^{\otimes}
\]
assigning to each object $X$ a complex of ind-objects of $\mhsp[\rational]$ with the following properties:
\begin{enumerate}
\item
the composite $\omega^{*, \otimes}\enhancedbettimon{}: \prns{\sm[ft]{\complex}}\op[\amalg] \to \D{\mod[dm = {\rational}]{}}^{\otimes}$ is naturally equivalent to $\bettimon{}$; and
\item
for each $r \in \integer$, the cohomology object $\h^r\enhancedbetti{X}$ is canonically isomorphic to the Betti cohomology $\operator[um = {r}, d = {B}]{H}{X, \rational}$ equipped with Deligne's mixed Hodge structure.
\end{enumerate}
\end{enumerate}
\end{defn}

\begin{prop}
\nlabel{sheaves.14}
The functors $\betti{}$ and $\enhancedbetti{}$ of \nref{sheaves.13} are \'etale local.
\end{prop}

\begin{proof}
Consider the assertion for $\betti{}$.
By \cite[Theorem~5.2]{Dugger-Isaksen_topological-hypercovers}, the functor $X \mapsto X^{\tu{an}}: \sm[ft]{\complex} \to \spc{}$ factors through a cocontinuous functor $\spc[u = {\'et}]{\spec{\complex}} \to \spt{}$.
The claim now follows readily.

Consider the assertion for $\enhancedbetti{}$.
First, we claim that $\enhancedbetti{}$ is local with respect to internal \'etale hypercovers.
By the description given in \nref{sheaves.13b}, for each $X \in \sm[ft]{\complex}$, the cohomology of the complex $\enhancedbetti{X}$ is bounded and belongs to the full subcategory $\mhsp[\rational] \subseteq \ind{\mhsp[\rational]}$.
By \cite[4.6]{Drew_verdier-quotients}, $\enhancedbetti{}$ therefore factors through the full \subqcategory\ $\D[u = {b}]{\mhsp[\rational]} \hookrightarrow \D{\ind{\mhsp[\rational]}}$ spanned by the $\aleph_0$-presentable objects.
By \nref{prop:descent}, it suffices to remark that composite $\omega^* \enhancedbetti{} \simeq \betti{}$ is local with respect to \'etale hypercovers.

That $\enhancedbetti{}$ is local with respect to all \'etale hypercovers now follows from the case of internal \'etale hypercovers, the fact that $\enhancedbetti{}$ preserves finite products, and \nref{sheaves.7}.
\end{proof}

\begin{defn}
\nlabel{sheaves.11a}
Let $\prns{\mc{C}, \tau}$ be a Verdier site satisfying Conditions $(1)$, $(2)$, and $(3)$ of \tu{\cite[\S10]{Dugger-Hollander-Isaksen_hypercovers-and-simplicial}}, and
$\mc{V}$ a \locpres[\aleph_0]\ symmetric monoidal \qcategory.
We say that $\mc{V}$ is \emph{$\prns{\mc{C},\tau}$-finite} if:
\begin{enumerate}
\item
\nlabel{sheaves.11a.1}
the \qcategory\ $\hypsh[dm = {\tau}]{\mc{C}}{\mc{V}}$ is \locpres[\aleph_0]; and
\item
\nlabel{sheaves.11a.2}
under the equivalence $\hypsh[dm = {\tau}]{\mc{C}}{\spc{}} \otimes \mc{V} \simeq \hypsh[dm = {\tau}]{\mc{C}}{\mc{V}}$ of \nref{sheaves.3.1},
the object $V \odot \loc[dm = {\tau}]{\yon[um = {\tau}]{\mc{C}}{X}}$ is $\aleph_0$-presentable for each $X \in \mc{C}$ and each $V \in \mc{V}_{\aleph_0}$.
\end{enumerate}
\end{defn}

\begin{prop}
\nlabel{sheaves.12}
\nlabel{SH.1.A}
If $S$ is a Noetherian scheme of finite dimension and $\mc{V}$ is a \locpres[\aleph_0]\ \qcategory, then:
\begin{enumerate}
\item
\nlabel{sheaves.12.1}
$\mc{V}$ is $\prns{\sm[ft]{S}, \tu{Nis}}$-finite;
\item
\nlabel{sheaves.12.2}
if $\mc{V}$ is $\rational$-linear, then $\mc{V}$ is $\prns{\sm[ft]{S}, \tu{\'et}}$-finite.
\end{enumerate}
\end{prop}

\begin{proof}
\nref[Claim]{sheaves.12.1} follows from the finiteness of the Nisnevich cohomological dimension as in \cite[4.5.67]{Ayoub_six-operationsII}.

For \nref[Claim]{sheaves.12.2}, it suffices to treat the case in which $\mc{V} = \D{\mod[dm = {\rational}]{}}^{\otimes}$.
Indeed, if the claim holds in that case, then by \nref{sheaves.3.1} we have equivalences
\begin{align*}
\hypsh[d = {\'et}]{\sm[ft]{S}}{\mc{V}}
&\simeq \hypsh[d = {\'et}]{\sm[ft]{S}}{\spc{}} \otimes \mc{V} \\
&\simeq \prns{\hypsh[d = {\'et}]{\sm[ft]{S}}{\spc{}} \otimes \D{\mod[dm = {\rational}]{}}} \otimes_{\D{\mod[dm = {\rational}]{}}} \mc{V} \\
&\simeq \hypsh[d = {\'et}]{\sm[ft]{S}}{\D{\mod[dm = {\rational}]{}}} \otimes_{\D{\mod[dm = {\rational}]{}}} \mc{V},
\end{align*}
where the relative tensor products over $\mod[dm = {\rational}]{}$ are computed in $\pr[um = {\tu{L}, \otimes}, dm = {\aleph_0, \tu{st}}]$.

By \nref{sheaves.4} and \nref{sheaves.6}, $\hypsh[d = {\'et}]{\sm[ft]{S}}{\mod[dm = {\rational}]{}}$ is canonically equivalent to the \qcategory\ underlying the \'etale-local model structure on the category of presheaves on $\sm[ft]{S}$ with values in unbounded complexes of $\rational$-modules.
As in \cite[3.18]{Ayoub_realisation-etale}, it therefore suffices to remark that the cohomological dimension of $\rational$-linear \'etale sheaves is finite.
\end{proof}

\section{Enriched motivic spectra}
\nlabel{SH}

\setcounter{thm}{-1}

\begin{notation}
In this section, we fix the following notation:
\begin{itemize}
\item
$S$, a Noetherian scheme of finite dimension; and
\item
$\tau$, either the Nisnevich or the \'etale topology.
\end{itemize}
\end{notation}

\begin{motivation}
By a theorem of M.~Robalo (\cite[Corollary~1.2]{Robalo_K-theory-and-the-bridge}), the $\tatesphere$-stable $\tau$-motivic homotopy category $\SH{S}$ is characterized by a universal property with respect to $\tau$-local, $\affine^1$-invariant, $\tatesphere$-stable symmetric monoidal functors $\prns{\sm[ft]{S}}^{\times} \to \mc{C}^{\otimes}$ as soon as one enhances $\SH{S}$ to a \locpres\ symmetric monoidal \qcategory\ $\spt[um = {\tau}, dm = {\tatesphere}]{S}^{\wedge}$.

More generally, for any \locpres\ symmetric monoidal \qcategory\ $\mc{V}^{\otimes}$, there should be a $\mc{V}^{\otimes}$-linear variant of $\spt[um = {\tau}, dm = {\tatesphere}]{S}^{\wedge}$, characterized by an analogous $\mc{V}^{\otimes}$-linear universal property.
For instance, when studying Hodge realizations of mixed motives, we will have reason below to consider the linearization of $\spt[um = {\tau}, dm = {\tatesphere}]{S}^{\wedge}$ over the symmetric monoidal derived \qcategory\ $\D{\ind{\mhsp[\rational]}}^{\otimes}$ of ind-objects in the Abelian category of polarizable mixed Hodge $\rational$-structures.

In this section, we formalize this process of $\mc{V}^{\otimes}$-linearization, its universal property, and the basic attributes of the resulting symmetric monoidal \qcategories\ $\spt[um = {\tau}, dm = {\tatesphere}]{S, \mc{V}}^{\otimes}$.
\end{motivation}

\begin{summary}
Let $\mc{V}^{\otimes}$ be a \locpres\ symmetric monoidal \qcategory.
\begin{itemize}
\item
In \nref{SH.1}, we construct the symmetric monoidal \qcategory\ $\spt[um = {\tau}, dm = {\tatesphere}]{S, \mc{V}}^{\otimes}$ in parallel with the classical construction of $\spt[um = {\tau}, dm = {\tatesphere}]{S}^{\wedge}$. 
\item
In \nref{SH.a2} and \nref{defn:weil-theory}, we recall the notions of $\tau$-local, $\affine^1$-invariant and $\tatesphere$-stable functors, and a generalization of the notion of a mixed Weil theory introduced in \cite{Cisinski-Deglise_mixed-weil}.
\item
In \nref{ex:weil-theory}, we introduce the main examples of such mixed Weil theories that we will consider in later sections.
\item
In \nref{SH.a4}, we extend Robalo's universal property of $\spt[um = {\tau}, dm = {\tatesphere}]{S}^{\wedge}$ to the $\mc{V}^{\otimes}$-linear setting. 
\item
In \nref{cor:SH-is-enriched}, we observe that $\spt[um = {\tau}, dm = {\tatesphere}]{S, \mc{V}}$ is \qcategorically\ enriched in $\mc{V}^{\otimes}$.
\item
In \nref{SH.a8}, we check that $\spt[u = {Nis}, dm = {\tatesphere}]{S, \mc{V}}$ inherits a convenient family of $\aleph_0$-presentable generators if $\mc{V}^{\otimes}$ is \locpres[\aleph_0], as does $\spt[u = {\'et}, dm = {\tatesphere}]{S, \mc{V}}$ if $\mc{V}$ is moreover $\rational$-linear.
\item
In \nref{SH.11}, under additional assumptions, we refine \nref{SH.a8} to obtain a family of $\otimes$-dualizable $\aleph_0$-presentable generators.
\end{itemize}
\end{summary}

\begin{defn}
\nlabel{SH.1}
Let $\mc{V}^{\otimes}$ be a \locpres\ symmetric monoidal \qcategory.
We introduce a $\mc{V}^{\otimes}$-linear \qcategorical\ variant of the construction of the motivic stable homotopy category presented in \cite[\S4.5]{Ayoub_six-operationsII}.
\begin{enumerate}
\item
\nlabel{SH.1.C}
\nlabel{SH.a5.1}
\nlabel{SH.1.B.etale}
\textbf{Hypersheaves.}
One begins by localizing $\psh{\sm[ft]{S}}{\spc{}}$ with respect to the class $\mf{W}_{\tau}$ of morphisms consisting of the form
\[
\colim_{\brk{n} \in \Delta\op} p_n: \colim_{\brk{n} \in \Delta\op} U_n \to \colim_{\brk{n} \in \Delta\op} \yon{S}{X} \simeq \yon{S}{X}
\]
for each $\tau$-hypercover $p_{\bullet}: U_{\bullet} \to \yon{S}{X}$ as in \nref{sheaves}.
As in \nref{sheaves.3a}, we let 
\[
\loc[dm = {\tau}]{}:
  \psh{\sm[ft]{S}}{\spc{}}
    \to \psh{\sm[ft]{S}}{\spc{}}\brk{\mf{W}^{-1}_{\tau}}
    \eqqcolon \hypsh[dm = {\tau}]{\sm[ft]{S}}{\spc{}}
\]
denote the localization functor, which is left exact and reflective.
Left exactness implies that $\loc[dm = {\tau}]{}$ underlies a canonical symmetric monoidal functor with respect to the Cartesian symmetric monoidal structures (\cite[2.4.1.1]{Lurie_higher-algebra}).

Since the \'etale topology is finer than the Nisnevich topology, $\mf{W}_{\tu{Nis}} \subseteq \mf{W}_{\tu{\'et}}$, so $\loc[d = {\'et}]{}$ factors through $\loc[d = {Nis}]{}$, i.e., there is a canonical left-exact reflective localization $\hypsh[d = {Nis}]{\sm[ft]{S}}{\spc{}} \to \hypsh[d = {\'et}]{\sm[ft]{S}}{\spc{}}$.
We denote the latter localization by $\loc[d = {\'et}]{}$:
this is not really abusive notation if we identify $\hypsh[d = {Nis}]{\sm[ft]{S}}{\spc{}}$ with a full \subqcategory\ of $\psh{\sm[ft]{S}}{\spc{}}$ via a right adjoint of $\loc[d = {Nis}]{}$.
\item
\nlabel{SH.1.D}
\nlabel{SH.1.B}
\nlabel{SH.1.E}
\textbf{Homotopy invariance and motivic spaces.}
Next, one localizes with respect to the essentially small set $\mf{W}_{\affine^1}$ of morphisms of the form $\loc[dm = {\tau}]{\yon{S}{\pi}}: \loc[dm ={\tau}]{\yon{S}{\affine^1_X}} \to \loc[dm = {\tau}]{\yon{S}{X}}$ for each $X \in \sm[ft]{S}$, where $\pi: \affine^1_X \to X$ denotes the projection.
The result is the \qcategory\ $\spc[um = {\tau}]{S}$ of \emph{$\tau$-motivic spaces over $S$}.
We shall denote the associated reflective localization functors as follows:
\[
\begin{tikzcd}[row sep = small]
&
\hypsh[d = {Nis}]{\sm[ft]{S}}{\spc{}}
\ar[r, "{\loc[dm={\affine^1}]{}}" below]
\ar[dd, "{\loc[d = {\'et}]{}}" right]
&
\spc[u = {Nis}]{S} \coloneqq \sh[d = {Nis}]{\sm[ft]{S}}{\spc{}}\brk{\mf{W}_{\affine^1}^{-1}}
\ar[dd, "{\loc[d = {\'et}]{}}" right]
\\
\psh{\sm[ft]{S}}{\spc{}}
\ar[ur,%
  "{\loc[d={Nis}]{}}" below,%
  start anchor = north east,%
  end anchor = south west,%
  bend left = 5pt]
\ar[urr, 
  "{\loc[dm={\affine^1,\tu{Nis}}]{}}" below, 
  start anchor = north, 
  end anchor = north west, 
  bend left = 35pt]
\ar[dr,%
  "{\loc[d = {\'et}]{}}" above,%
  start anchor = south east,%
  end anchor = north west,%
  bend right = 5pt]
\ar[drr,%
  "{\loc[dm = {\affine^1, \tu{\'et}}]{}}" above,%
  start anchor = south,%
  end anchor = south west,%
  bend right = 35pt]
&
&
\\
&
\hypsh[d = {\'et}]{\sm[ft]{S}}{\spc{}}
\ar[r, "{\loc[dm = {\affine^1}]{}}" above]
&
\spc[u = {\'et}]{S} \coloneqq \hypsh[d = {\'et}]{\sm[ft]{S}}{\spc{}}\brk{\mf{W}_{\affine^1}^{-1}}
\end{tikzcd}
\]
By \cite[C.6]{Hoyois_quadratic-refinement}, $\loc[dm={\affine^1}]{}$ and, hence, $\loc[dm={\affine^1,\tu{Nis}}]{}$ and $\loc[dm = {\affine^1, \tu{\'et}}]{}$ are left-exact reflective localizations.
They are therefore also symmetric monoidal with respect to the Cartesian symmetric monoidal structures.

The \qcategory\ $\spc[u = {Nis}]{S}$ is equivalent to those denoted by $\fct{\mathcal{H}}{S}$ and $\fct{\tu{H}}{S}$ in \cite[2.4.1]{Robalo_K-theory-and-the-bridge} and \cite[Appendix C]{Hoyois_quadratic-refinement}, respectively.
Its homotopy category $\ho{\spc[u = {Nis}]{S}}$ is equivalent to the unstable motivic homotopy category of F.~Morel and V.~Voevodsky as constructed in \cite{Morel-Voevodsky_A1-homotopy-theory}.

More generally, the \qcategory\ of \emph{$\mc{V}^{\otimes}$-linear $\tau$-motivic spaces over $S$}, denoted by $\spc[um = {\tau}]{S,\mc{V}}$, is the tensor product $\spc[um = {\tau}]{S} \otimes \mc{V}$ in $\pr^{\mr{L},\otimes}$, which we equip with the symmetric monoidal structure $\spc[um = {\tau}]{S,\mc{V}}^{\otimes} \coloneqq \spc[um = {\tau}]{S}^{\times} \otimes \mc{V}^{\otimes}$ given by the coproduct in $\calg{\pr^{\mr{L},\otimes}}$.
Letting $\mc{V}^{\otimes} = \spc{}^{\times}$, we recover $\spc[um = {\tau}]{S}$ up to canonical equivalence.
\item
\nlabel{SH.1.F}
\textbf{Yoneda functors.}
Let $\yon{S, \mc{V}}{}: \sm[ft]{S} \to \psh{\sm[ft]{S}}{\spc{}} \otimes \mc{V}$ denote the composite
\[
\sm[ft]{S} 
  \xrightarrow{\yon{S}{}}
    \psh{\sm[ft]{S}}{\spc{}}
      \simeq \psh{\sm[ft]{S}}{\spc{}} \otimes \spc{} 
        \xrightarrow{\id \otimes \eta_{\mc{V}}}
          \psh{\sm[ft]{S}}{\spc{}} \otimes \mc{V},
\]
where $\eta: \spc{} \to \mc{V}$ classifies the monoidal unit $\1{\mc{V}} \in \mc{V}$.
Informally, $\yon{S, \mc{V}}{}$ is given by $X \mapsto \yon{S}{X} \otimes \1{\mc{V}}$.
The functor $\yona[um = {\tau}]{S}{}: \sm[ft]{S} \to \spc[um = {\tau}]{S}$ is the composite
\[
\yona[um = {\tau}]{S}{}:
  \sm[ft]{S}
    \xrightarrow{\yon{S}{}} \psh{\sm[ft]{S}}{\spc{}}
      \xrightarrow{\loc[dm={\affine^1,\tu{Nis}}]{}}
        \spc[um={\tau}]{S}
\]
and we define $\yona[um = {\tau}]{S, \mc{V}}{}: \sm[ft]{S} \to \spc[um = {\tau}]{S, \mc{V}}$ as the composite
\[
\yona[um={\tau}]{S,\mc{V}}{}:
  \sm[ft]{S}
    \xrightarrow{\yona[um={\tau}]{S}{}} \spc[um={\tau}]{S}
      \simeq \spc[um={\tau}]{S} \otimes \spc{}
        \xrightarrow{\id \otimes \eta_{\mc{V}}} 
          \spc[um={\tau}]{S} \otimes \mc{V} 
            \eqqcolon \spc[um={\tau}]{S, \mc{V}},
\]
given informally by $X \mapsto \yona[um = {\tau}]{S}{X} \otimes \1{\mc{V}}$.
\item
\nlabel{SH.1.G}
\nlabel{SH.1.H}
\textbf{Pointed motivic spaces.}
Recall that a \qcategory\ $\mc{C}$ is \emph{pointed} if it admits an object which is both initial and final.
Passing to the \qcategory\ of pointed objects (\cite[4.8.1.20]{Lurie_higher-algebra}) in $\spc[um = {\tau}]{S}$, we obtain the \qcategory\ $\spc[um = {\tau}, dm={\pt}]{S}$ of \emph{pointed $\tau$-motivic spaces over $S$}.
It inherits a symmetric monoidal structure $\spc[um = {\tau}, dm = {\pt}]{S}^{\wedge}$ from $\spc[um = {\tau}]{S}^{\times}$ and a symmetric monoidal left-adjoint functor $\spc[um = {\tau}]{S}^{\times} \to \spc[um = {\tau}, dm = {\pt}]{S}^{\wedge}$ universal with respect to symmetric monoidal left-adjoint functors $\spc[um = {\tau}]{S}^{\times} \to \mc{C}^{\otimes}$ into pointed \locpres\ symmetric monoidal \qcategories\ $\mc{C}^{\otimes}$ (\cite[Corollary 2.32]{Robalo_K-theory-and-the-bridge}).

Tensoring with $\mc{V}^{\otimes}$, we obtain the symmetric monoidal \qcategory\ $\spc[um = {\tau}, dm = {\pt}]{S,\mc{V}}^{\otimes} \coloneqq \spc[um = {\tau}, dm = {\pt}]{S}^{\wedge} \otimes \mc{V}^{\otimes}$ of \emph{pointed $\mc{V}^{\otimes}$-linear $\tau$-motivic spaces over $S$}.
If $\mc{V}$ is pointed, then $\mc{V} \simeq \mc{V}_{\pt}$, and so 
\[
\spc[um = {\tau}, dm={\pt}]{S} \otimes \mc{V} 
  \simeq \prns{\spc[um = {\tau}]{S} \otimes \mc{V}}_{\pt} 
  \simeq \spc[um = {\tau}]{S} \otimes \mc{V}_{\pt}
  \simeq \spc[um = {\tau}]{S} \otimes \mc{V},
\]
which renders this step unnecessary.

We adopt the following notation for pointed objects.
If $s: \pt \to F$ is a morphism from the final object $\pt \in \spc[um = {\tau}]{S, \mc{V}}$, then we denote the corresponding object of $\spc[um = {\tau}, dm={\pt}]{S, \mc{V}}$ by $\prns{F, s}$.
If $F \in \spc[um = {\tau}]{S, \mc{V}}$, we let $F_+$ denote its image under the universal functor $\spc[um = {\tau}]{S, \mc{V}} \to \spc[um = {\tau}, dm={\pt}]{S, \mc{V}}$, which is informally to be thought of as $F$ with a disjoint base point.
\item
\nlabel{SH.1.I}
\textbf{Tate spheres.}
The standard open immersion $j: \Gm{S} \hookrightarrow \affine^1_S$ induces a morphism $\yona[um = {\tau}]{S}{j}_+$ in $\spc[um = {\tau}, dm={\pt}]{S}$.
The \emph{Tate sphere} $\tatesphere[S]$ is the cofiber of of this morphism:
\[
\tatesphere[S]
  \coloneqq \cofib[size = 1]{\yona[um = {\tau}]{S}{\Gm{S}}_+ \to \yona[um = {\tau}]{S}{\affine^1_S}_+}.
\]
This notation is abusive insofar as it does not specify $\tau$, but confusion is unlikely to result from this.
Let $\sigma_1: S \to \Gm{S}$ and $\sigma_1: S \to \projective^1_S$ denote the unit sections.
As $\yona[um = {\tau}]{S}{\affine^1_S} \simeq \pt$ and as $\projective^1_S$ admits a standard Zariski cover by two copies of $\affine^1_S$ whose intersection is $\Gm{S}$, we have equivalences 
\[
\prns{\yona[um = {\tau}]{S}{\Gm{S}}, \sigma_1} \wedge \sphere^1 
  \simeq \tatesphere[S] 
  \simeq \prns{\yona[um = {\tau}]{S}{\projective^1_S}, \sigma_1},
\]
where $\sphere^1 \coloneqq \Delta^1/\partial \Delta^1$ is the simplicial circle.
\item
\nlabel{SH.1.J}
\textbf{Motivic spectra.}
Following \cite[Definition 2.38]{Robalo_K-theory-and-the-bridge}, we define the symmetric monoidal \qcategory\ $\spt[um = {\tau}, dm = {\tatesphere}]{S}^{\wedge}{}$ of \emph{$\tau$-\motivicspectra\ over $S$} by formally adjoining a $\otimes$-inverse to $\tatesphere[S]$ in $\spc[um = {\tau}, dm = {\pt}]{S}^{\wedge}$.
We similarly define the symmetric monoidal \qcategory\ of \emph{$\mc{V}^{\otimes}$-linear $\tau$-\motivicspectra\ over $S$} to be the tensor product 
\[
\spt[um = {\tau}, dm = {\tatesphere}]{S,\mc{V}}^{\otimes} 
  \coloneqq \spt[um = {\tau}, dm = {\tatesphere}]{S}^{\wedge} \otimes \mc{V}^{\otimes}.
\]
We shall see below (\nref{SH.a8}) that $\spt[um = {\tau}, dm = {\tatesphere}]{S, \mc{V}}^{\otimes}$ is stable and \locpres.
We let
\[
\Sigma^{\infty, \otimes}_{\tatesphere,\mc{V}} 
  \coloneqq \Sigma^{\infty, \wedge}_{\tatesphere} \otimes \id_{\mc{V}^{\otimes}}:
  \spc[um = {\tau}, dm = {\pt}]{S, \mc{V}}^{\otimes}
    \to \spt[um = {\tau}, dm = {\tatesphere}]{S, \mc{V}}^{\otimes}
\]
denote the canonical symmetric monoidal functor, and we call it \emph{infinite $\tatesphere$-suspension}.
By \cite[7.3.2.7]{Lurie_higher-algebra}, it admits a lax symmetric monoidal right adjoint, which we shall denote by $\Omega^{\infty,\otimes}_{\tatesphere,\mc{V}}$.
\item
\nlabel{SH.1.K}
\textbf{Enriched Tate spheres.}
We let $\tatesphere[S, \mc{V}] \coloneqq \tatesphere[S] \otimes \1{\mc{V}} \in \spc[um = {\tau}, dm = {\pt}]{S, \mc{V}}^{\otimes}$ denote the \emph{$\mc{V}^{\otimes}$-linear Tate sphere}.
The infinite $\tatesphere$-suspension $\Sigma^{\infty}_{\tatesphere, \mc{V}}\tatesphere[S, \mc{V}]$ of the $\mc{V}^{\otimes}$-linear Tate sphere is $\otimes$-invertible: it is equivalent to the tensor product $\prns{\Sigma^{\infty}_{\tatesphere}\tatesphere[S]} \otimes \1{\mc{V}} \in \spt[um = {\tau}, dm = {\tatesphere}]{S}^{\wedge} \otimes \mc{V}^{\otimes}$, the first factor of which is $\otimes$-invertible by definition of $\spt[um = {\tau}, dm = {\tatesphere}]{S}^{\wedge}$ and the second factor of which is tautologically $\otimes$-invertible.
It follows that the \emph{Tate object} $\tate{S, \mc{V}}{1} \coloneqq \prns{\Sigma^{\infty}_{\tatesphere, \mc{V}}\tatesphere[S]}\brk{-2}$ is also $\otimes$-invertible.
For each $r \in \integer$, we define $\tate{S, \mc{V}}{r} \coloneqq \prns{\tate{S, \mc{V}}{1}}^{\otimes r}$ to be its $r$th tensor power.
More generally, for each $M \in \spt[um = {\tau}, dm = {\tatesphere}]{S, \mc{V}}$, we define the \emph{$r$th Tate twist} by $\twist{M}{r} \coloneqq M \otimes \tate{S, \mc{V}}{r}$.
The universal property of $\spt[um = {\tau}, dm = {\tatesphere}]{S}^{\wedge}$, recalled below in \nref{SH.a4}, implies that we can alternatively describe $\spt[um = {\tau}, dm = {\tatesphere}]{S, \mc{V}}^{\otimes}$ as the result of formally adjoining a $\otimes$-inverse to $\tatesphere[S, \mc{V}] \in \spc[um = {\tau}, dm = {\pt}]{S, \mc{V}}$.
\end{enumerate}
\end{defn}

\begin{defn}
\nlabel{SH.a2}
Let $\mc{V}$ be a \qcategory
and $F: \sm[ft]{S} \to \mc{V}$ a functor.
\begin{enumerate}
\item
\nlabel{SH.a2.1}
We say that $F$ is \emph{$\tau$-local} if it is local with respect to internal $\tau$-hypercovers and preserves finite coproducts.
\item 
\nlabel{SH.a2.2}
We say that $F$ is \emph{$\affine^1$-invariant} if it sends the projection $\pi: \affine^1_X \to X$ to an equivalence in $\mc{V}$ for each $X \in \sm[ft]{S}$.
\item
\nlabel{SH.a2.3}
Dually, we say that $F: \prns{\sm[ft]{S}}\op \to \mc{V}$ is \emph{$\affine^1$-invariant} whenever the opposite functor $F\op: \sm[ft]{S} \to \mc{V}\op$ is $\affine^1$-invariant.
\end{enumerate}

\noindent
Suppose now that $\mc{V}^{\otimes}$ is a pointed symmetric monoidal \qcategory\ and if $F^{\otimes}: \prns{\sm[ft]{S}}^{\times} \to \mc{V}^{\otimes}$ is a symmetric monoidal functor.
\begin{enumerate}[resume]
\item
\nlabel{SH.a2.4}
We say that $F^{\otimes}$ is \emph{$\tatesphere$-stable} if the cofiber of $\fct{F}{\sigma_1}$ exists and is $\otimes$-invertible, where $\sigma_1: S \to \Gm{S}$ is the unit section.
\item
\nlabel{SH.a2.5}
We say that $F^{\otimes}$ is \emph{$\tau$-local} or \emph{$\affine^1$-invariant} if the underlying functor $F: \sm[ft]{S} \to \mc{V}$ is.
\item
\nlabel{SH.a2.6}
Dually, we say that $F^{\otimes}: \prns{\sm[ft]{S}}\op[\amalg] \to \mc{V}^{\otimes}$ is \emph{$\tau$-local}, \emph{$\affine^1$-invariant} or \emph{$\tatesphere$-stable} if $F\op[\otimes]: \prns{\sm[ft]{S}}^{\times} \to \mc{V}\op[\otimes]$ is, where we equip $\mc{V}\op$ with the opposite symmetric monoidal structure of \cite[2.4.2.7]{Lurie_higher-algebra}.
\end{enumerate}
\end{defn}

\begin{rmk}
In \nref{SH.a2.1}, we define the property of being $\tau$-local via internal $\tau$-hypercovers and preservation of finite coproducts because we do not require $\mc{V}$ to be complete or cocomplete, making it inconvenient to work directly with arbitrary $\tau$-hypercovers.
By \nref{sheaves.7}, when $\mc{V}$ is \locpres, this notion of $\tau$-local is equivalent to the condition of being local with respect to all $\tau$-hypercovers. 
\end{rmk}

\begin{defn}
\nlabel{defn:weil-theory}
Let $\mc{V}^{\otimes}$ be a symmetric monoidal \qcategory.
A \emph{$\mc{V}^{\otimes}$-valued mixed Weil theory over $S$} is a symmetric monoidal functor $\cohomologymon{}: \prns{\sm[ft]{S}}\op[\amalg] \to \mc{V}^{\otimes}$ satisfying the following properties:
\begin{enumerate}
\item
$\cohomology{}$ factors through the inclusion of the full \subqcategory\ of $\mc{V}$ spanned by the $\otimes$-dualizable objects; and
\item
$\cohomologymon{}$ is Nisnevich local, $\affine^1$-invariant and $\tatesphere$-stable in the sense of \nref{SH.a2.6}.
\end{enumerate}
A $\mc{V}^{\otimes}$-valued mixed Weil theory as above is \emph{$\tau$-local} if the functor $\cohomology{}$ is $\tau$-local.
This definition is a generalization of that of \cite[2.1.4]{Cisinski-Deglise_mixed-weil} in a way that we will not bother to make precise here.
\end{defn}

\begin{ex}
\nlabel{ex:weil-theory}
Let $S = \spec{\complex}$.
The following will be the most important examples in the sequel.
\begin{enumerate}
\item
\nlabel{ex:weil-theory.1}
The symmetric monoidal functor $\bettimon{}$ of \nref{sheaves.13a} is \'etale local by \nref{sheaves.14}.
It is $\affine^1$-invariant since singular cohomology is homotopy invariant.
It is $\tatesphere$-stable by the standard computation of the singular cohomology of $\Gm{S}^{\tu{an}} \simeq \complex^{\times}$.
Finally, the $\otimes$-dualizable objects of $\D{\mod[dm = {\rational}]{}}$ are the complexes whose cohomology is bounded and of finite rank over $\rational$.
Finiteness of Betti cohomology therefore implies that $\betti{X}$ is $\otimes$-dualizable for each $X \in \sm[ft]{\complex}$. 
Thus, $\bettimon{}$ is an \'etale-local $\D{\mod[dm = {\rational}]{}}^{\otimes}$-valued mixed Weil theory.
\item
\nlabel{ex:weil-theory.2}
The symmetric monoidal functor $\enhancedbettimon{}$ of \nref{sheaves.13b} is \'etale local by \nref{sheaves.14}.
Since the fiber functor
\[
\omega^*:
\D{\ind{\mhsp[\rational]}}
  \to \D{\mod[dm = {\rational}]{}}
\]
(\cite[1.6]{Drew_rectification-of-Deligne's}) is conservative, $\enhancedbettimon{}$ inherits $\affine^1$-invariance and $\tatesphere$-stability from $\bettimon{}$.
That $\enhancedbetti{X}$ is $\otimes$-dualizable for each $X \in \sm[ft]{\complex}$ follows from \cite[4.7]{Drew_verdier-quotients}.
Thus, the symmetric monoidal functor $\enhancedbettimon{}$ is an \'etale-local $\D{\ind{\mhsp[\rational]}}^{\otimes}$-valued mixed Weil theory.
\end{enumerate}
\end{ex}

\begin{notation}
\nlabel{SH.a3}
Let $\mc{C}^{\otimes}$ and $\mc{D}^{\otimes}$ be symmetric monoidal \qcategories.
Under the identification of symmetric monoidal \qcategories\ with commutative algebra objects in \qcategories\ of \qcategories\ equipped with Cartesian monoidal structures (\cite[4.8.1.9]{Lurie_higher-algebra}), the $0$-simplices of the mapping space
$\map[dm = {\calg{\qcatmon}}]{\mc{C}^{\otimes}}{\mc{D}^{\otimes}}$
classify symmetric monoidal functors $\mc{C}^{\otimes} \to \mc{D}^{\otimes}$, and the $1$-simplices classify equivalences of such.

If $\mc{V}^{\otimes}$ is a \locpres\ symmetric monodial \qcategory\ and $\mc{C}^{\otimes}$ and $\mc{D}^{\otimes}$ are commutative $\mc{V}^{\otimes}$-algebras in $\pr[um = {\tu{L}, \otimes}]$, then the $0$-simplices of the mapping space
$\map[dm = {\calg{\pr[um = {\tu{L}, \otimes}]}_{\mc{V}^{\otimes}/}}]{\mc{C}^{\otimes}}{\mc{D}^{\otimes}}$ classify $\mc{V}^{\otimes}$-linear, cocontinuous symmetric monoidal functors $\mc{C}^{\otimes} \to \mc{D}^{\otimes}$, and the $1$-simplices classify $\mc{V}^{\otimes}$-linear equivalences of such.

If $P$ is a property of symmetric monoidal functors $\mc{C}^{\otimes} \to \mc{D}^{\otimes}$, then we let 
\[
\map[size = 0, dm = {\calg{\QCATmon}}]{\mc{C}^{\otimes}}{\mc{D}^{\otimes}}_P
  \subseteq \map[size = 0, dm = {\calg{\QCATmon}}]{\mc{C}^{\otimes}}{\mc{D}^{\otimes}}
\]
denote the Kan subcomplex spanned by the simplices whose vertices classify symmetric monoidal functors with the property $P$.
\end{notation}

\begin{thm}[M.~Robalo]
\nlabel{SH.a4}
Let $v^{\otimes}: \mc{V}^{\otimes} \to \mc{W}^{\otimes}$ be a cocontinuous symmetric monoidal functor between \locpres\ \qcategories\ such that $\mc{W}^{\otimes}$ is pointed.
Restriction along $\Sigma^{\infty, \otimes}_{\tatesphere, \mc{V}}\prns{-}_+$ induces a weak homotopy equivalence
\[
\map[dm = {\calg{\pr[um = {\tu{L}, \otimes}]}_{\mc{V}^{\otimes}/}}]{\spt[um = {\tau}, dm = {\tatesphere}]{S, \mc{V}}^{\otimes}}{\mc{W}^{\otimes}}
  \to \map[dm = {\calg{\QCATmon}}]{\prns{\sm[ft]{S}}^{\times}}{\mc{W}^{\otimes}}_P,
\]
where $P$ denotes the property of being $\tau$-local, $\affine^1$-invariant and $\tatesphere$-stable.
\end{thm}

\begin{proof}
We have a homotopy-commutative triangle
\[
\begin{tikzcd}[row sep = tiny]
\map[size = 0, dm = {\calg{\pr[um = {\tu{L}, \otimes}]}_{\mc{V}^{\otimes}/}}]{\spt[um={\tau}, dm = {\tatesphere}]{S, \mc{V}}^{\otimes}}{\mc{W}^{\otimes}}
\ar[dr, 
  "\gamma" below, 
  start anchor = east,
  end anchor = north west,
  bend left = 15pt]
\ar[dd, "\alpha" left]
&
&
\\
&
\map[size = 0, dm = {\calg{\QCATmon}}]{\prns{\sm[ft]{S}}^{\times}}{\mc{W}^{\otimes}}.
&
\\
\map[size = 0, dm = {\calg{\pr[um = {\tu{L}, \otimes}]}}]{\spt[um={\tau}, dm = {\tatesphere}]{S}^{\wedge}}{\mc{W}^{\otimes}}
\ar[ur, 
  "\beta" above, 
  start anchor = east,
  end anchor = south west,
  bend right = 15pt]
&
&
\end{tikzcd}
\]
The morphism $\alpha$ is a weak homotopy equivalence by the adjunction
$
\calg{\pr[um = {\tu{L}, \otimes}]} 
  \rightleftarrows \calg{\pr[um = {\tu{L}, \otimes}]}_{\mc{V}^{\otimes}/}
$
between $\prns{-} \otimes \mc{V}^{\otimes}$ and the forgetful functor.
That $\beta$ induces a weak homotopy equivalence with $\map[size = 0, dm = {\calg{\QCATmon}}]{\prns{\sm[ft]{S}}^{\times}}{\mc{W}^{\otimes}}_P$ is just the special case in which $\mc{V}^{\otimes} = \spc{}^{\times}$, which follows from \cite[Corollary~1.2]{Robalo_K-theory-and-the-bridge}.
\end{proof}

\begin{rmk}
\nlabel{SH.a6}
If $\mc{V}^{\otimes}$ is a \locpres\ symmetric monoidal \qcategory\ underlying a $\sset^{\times}$-enriched symmetric monoidal model category, then M.~Robalo's theorem (\cite[Theorem 2.26]{Robalo_K-theory-and-the-bridge}) implies that the homotopy category of $\spt[um = {\tau}, dm = {\tatesphere}]{S, \mc{V}}^{\otimes}$ is equivalent to the $\ho{\mc{V}}^{\otimes}$-linear $\tau$-motivic stable homotopy category constructed in \cite[\S4.5]{Ayoub_six-operationsII} as the homotopy category of a stable symmetric monoidal model category.
\end{rmk}

\begin{rmk}
\nlabel{SH.a7}
As in \nref{sheaves.3.1}, the construction of $\spt[um = {\tau}, dm = {\tatesphere}]{S, \mc{V}}^{\otimes}$ is functorial in $\mc{V}^{\otimes}$.
Indeed, the symmetric monoidal structure on $\pr[u = {L}]$ provides us with a functor
\[
\spt[um = {\tau}, dm = {\tatesphere}, um = {\tau}]{S}^{\wedge} \otimes \prns{-}:
  \calg{\pr[um = {\tu{L}, \otimes}]}
    \to \calg{\pr[um = {\tu{L}, \otimes}]}
\]
given informally by $\mc{V}^{\otimes} \mapsto \spt[um = {\tau}, dm = {\tatesphere}]{S}^{\wedge} \otimes \mc{V}^{\otimes} \eqqcolon \spt[um = {\tau}, dm = {\tatesphere}]{S, \mc{V}}^{\otimes}$.
The construction of $\spt[um = {\tau}, dm = {\tatesphere}]{S, \mc{V}}^{\otimes}$ is also functorial with respect to the base scheme $S$, and we shall explore this further in \nref{funct} below.
\end{rmk}

\begin{prop}[Gepner-Haugseng]
\nlabel{prop:enrichments}
Let $\mc{V}^{\otimes}$ be a symmetric monoidal \qcategory.
\begin{enumerate}
\item
\nlabel{prop:enrichments.1}
If $\mc{V}^{\otimes}$ is \locpres, then $\mc{V}$ admits a $\mc{V}^{\otimes}$-enriched-\qcategory\ structure given informally by the internal morphisms objects $\intmor[dm = {\mc{V}}]{X}{Y}$ for $\prns{X, Y} \in \mc{V}^2$.
\item
\nlabel{prop:enrichments.2}
If $\phi^{\otimes}: \mc{V}^{\otimes} \to \mc{W}^{\otimes}$ is a lax symmetric monoidal functor between symmetric monoidal \qcategories\ and $\mc{C}$ is a $\mc{V}^{\otimes}$-enriched \qcategory, then $\mc{C}$ admits a $\mc{W}^{\otimes}$-enriched-\qcategory\ structure given informally by
\[
\mor[um = {\mc{W}}, dm = {\mc{C}}]{X}{Y}
  \coloneqq \fct{\phi}{\mor[um = {\mc{V}}, dm = {\mc{C}}]{X}{Y}}
\]
for each $\prns{X,Y} \in \mc{C}^2$, where $\mor[um = {\mc{V}}, dm = {\mc{C}}]{}{}$ denotes the morphisms-$\mc{V}$-object bifunctor.
\end{enumerate}
\end{prop}

\begin{proof}
This follows from \cite[7.4.10, 5.7.6]{Gepner-Haugseng_enriched-infty-categories}.
\end{proof}

\begin{cor}
\nlabel{cor:SH-is-enriched}
If $\mc{V}^{\otimes}$ is a \locpres\ symmetric monoidal \qcategory,
then $\spt[um = {\tau}, dm = {\tatesphere}]{S, \mc{V}}^{\otimes}$ admits a $\mc{V}^{\otimes}$-enriched-\qcategory\ structure given informally by 
\[
\mor[um = {\mc{V}}, dm = {\spt[um = {\tau}, dm = {\tatesphere}]{S, \mc{V}}}]{M}{N}
  \coloneqq \fct{\Gamma}{S, \Omega^{\infty}_{\tatesphere, \mc{V}} \intmor[dm = {\spt[um = {\tau}, dm = {\tatesphere}]{S, \mc{V}}}]{M}{N}}
\]
for each $\prns{M, N} \in \spt[um = {\tau}, dm = {\tatesphere}]{S, \mc{V}}$.
\end{cor}

\begin{proof}
We have a sequence of symmetric monoidal left adjoints
\[
\mc{V}^{\otimes}
  \xrightarrow{\cst{}^{\otimes}} \psh{\sm[ft]{S}}{\mc{V}}^{\otimes}
  \xrightarrow{\loc[um = {\otimes}, dm = {\affine^1, \tau}]{}} \spc[um = {\tau}]{S, \mc{V}}^{\otimes}
  \xrightarrow{\prns{-}_+^{\otimes}} \spc[um = {\tau}, dm = {\pt}]{S, \mc{V}}^{\otimes}
  \xrightarrow{\Sigma^{\infty, \otimes}_{\tatesphere, \mc{V}}} \spt[um = {\tau}, dm = {\tatesphere}]{S, \mc{V}}^{\otimes},
\]
where $\cst{}^{\otimes}$ is the constant-diagram functor, and the other functors are as in \nref{SH.1}.
The associated composite right adjoint is lax symmetric monoidal by \cite[7.3.2.7]{Lurie_higher-algebra}.
The claim now follows from \nref{prop:enrichments} and the following observations:
$\cst{}^{\otimes}$ is left adjoint to the global-sections functor $\fct{\Gamma}{S, -}$;
$\loc[dm = {\affine^1, \tau}]{}$ is a reflective localization and therefore left adjoint to the inclusion;
$\prns{-}_+$ is left adjoint to the forgetful functor; and
$\Sigma^{\infty}_{\tatesphere, \mc{V}}$ is left adjoint to $\Omega^{\infty}_{\tatesphere, \mc{V}}$.
\end{proof}

\begin{rmk}
\nlabel{scalars.a3}
Let $\mc{C}$ and $\mc{D}$ be \locpres\ \qcategories.
\begin{enumerate}
\item
\nlabel{scalars.a3.1}
If $\mc{C}$ or $\mc{D}$ is stable, then $\mc{C} \otimes \mc{D}$ is also stable.
Indeed, by \cite[4.8.2.18]{Lurie_higher-algebra}, a \locpres\ \qcategory\ $\mc{E}$ is stable if and only if the tensor product $\Sigma^{\infty}_{\sphere^1} \otimes \id_{\mc{E}}: \spc{} \otimes \mc{E} \to \spt{} \otimes \mc{E}$ of the identity with the infinite $\sphere^1$-suspension functor $\Sigma^{\infty}_{\sphere^1}: \spc{} \to \spt{}$ is an equivalence.
The claim then follows from the associativity of the tensor product:
\[
\begin{tikzcd}[column sep = small, row sep = small]
\spc{} \otimes \prns{\mc{C} \otimes \mc{D}}
\ar[r, "\sim" above]
\ar[d]
&
\prns{\spc{} \otimes \mc{C}} \otimes \mc{D}
\ar[d, "\sim"]
\\
\spt{} \otimes \prns{\mc{C} \otimes \mc{D}}
\ar[r, "\sim" above]
&
\prns{\spt{} \otimes \mc{C}} \otimes \mc{D}.
\end{tikzcd}
\]
\item
\nlabel{scalars.a3.2}
\nlabel{scalars.a3.3}
Suppose that $\mc{C}^{\otimes}$ and $\mc{D}^{\otimes}$ are \locpres\ symmetric monoidal \qcategories.
Similarly, by \cite[4.8.1.19, 3.4.1.7]{Lurie_higher-algebra}, a \locpres\ symmetric monoidal \qcategory\ $\mc{V}^{\otimes}$ is stable if and only if $\mc{V}^{\otimes}$ admits a commutative $\spt{}^{\wedge}$-algebra structure in $\pr[um = {\tu{L}, \otimes}]$.
It follows that if $\mc{C}^{\otimes} \leftarrow \mc{V}^{\otimes} \to \mc{D}^{\otimes}$ is a diagram of \locpres\ symmetric monoidal \qcategories, and if one of $\mc{C}^{\otimes}$, $\mc{D}^{\otimes}$ or $\mc{V}^{\otimes}$ is stable, then the pushout $\mc{C}^{\otimes} \otimes_{\mc{V}^{\otimes}} \mc{D}^{\otimes}$ in $\calg{\pr[um = {\tu{L}, \otimes}]}$ is also stable: 
it inherits a commutative $\spt{}^{\wedge}$-algebra structure in each case. 
\end{enumerate}
\end{rmk}

\begin{rmk}
\nlabel{SH.b8}
Let $\mc{C}$ be a \locpres[\aleph_0]\ \qcategory\ with final object $\pt$.
The \qcategory\ $\mc{C}_{\pt} = \mc{C}_{\pt/}$ of pointed objects in $\mc{C}$ is also \locpres[\aleph_0], and an object of $\mc{C}_*$ is $\aleph_0$-presentable if and only if its image under the forgetful functor $\mc{C}_* \to \mc{C}$ is $\aleph_0$-presentable by \cite[5.4.5.15, 5.5.3.11]{Lurie_higher-topos}.
\end{rmk}

\begin{prop}
\nlabel{SH.a8}
Let $\kappa$ be a small regular cardinal and $\mc{V}^{\otimes}$ a \locpres[\kappa]\ symmetric monoidal \qcategory.
\begin{enumerate}
\item
\nlabel{SH.a8.1}
The symmetric monoidal \qcategory\ $\spt[um = {\tau}, dm = {\tatesphere}]{S, \mc{V}}^{\otimes}$ is stable and \locpres.
\item
\nlabel{SH.a8.2}
The \qcategory\ $\spt[um = {\tau}, dm = {\tatesphere}]{S, \mc{V}}$ is generated under transfinitely iterated small colimits by the objects of the form $V \odot \Sigma^{\infty}_{\tatesphere, \mc{V}}\stwist{\yona[um = {\tau}]{S, \mc{V}}{X}_+}{r}{s}$ with $X \in \sm[ft]{S}$, $\prns{r,s} \in \integer^2$ and $V \in \mc{V}_{\kappa}$.
\item
\nlabel{SH.a8.3}
If $\mc{V}$ is $\prns{\sm[ft]{S}, \tau}$-finite, then $\spt[um = {\tau}, dm = {\tatesphere}]{S, \mc{V}}^{\otimes}$ is \locpres[\aleph_0] and the generators in \nref{SH.a8.2} are $\aleph_0$-presentable.
\end{enumerate}
\end{prop}

\begin{proof}
Consider \nref[Claim]{SH.a8.1}.
As explained in the remarks following \cite[Definition 2.38]{Robalo_K-theory-and-the-bridge}, $\spt[um = {\tau}, dm = {\tatesphere}]{S}^{\wedge}$ is stable.
By \nref{scalars.a3}, $\spt[um = {\tau}, dm = {\tatesphere}]{S, \mc{V}}^{\otimes}$ is also stable.
Also, $\spt[um = {\tau}, dm = {\tatesphere}]{S}^{\wedge}$ is \locpres\ by construction.
That $\spt[um = {\tau}, dm = {\tatesphere}]{S, \mc{V}}^{\otimes}$ is \locpres\ follows from \cite[5.3.2.11]{Lurie_higher-algebra}.

Consider \nref[Claim]{SH.a8.2}.
It follows from the proof of \cite[4.8.1.15]{Lurie_higher-algebra} that the tensor product $\spt[um = {\tau}, dm = {\tatesphere}]{S, \mc{V}} = \spt[um = {\tau}, dm = {\tatesphere}]{S} \otimes \mc{V}$ is generated under transfinitely iterated small colimits by the objects of the form $V \odot M$ with $M \in \spt[um = {\tau}, dm = {\tatesphere}]{S}$ and $V \in \mc{V}$.
By the argument of \cite[C.12.(1)]{Hoyois_quadratic-refinement}, $\spt[um = {\tau}, dm = {\tatesphere}]{S}$ is generated under transfinitely iterated small colimits by the objects $\Sigma^{\infty}_{\tatesphere}\stwist{\yona[um = {\tau}]{S}{X}_+}{r}{s}$ with $X \in \sm[ft]{S}$ and $\prns{r,s} \in \integer^2$.
As $\mc{V}$ is \locpres[\kappa], it is generated under transfinitely iterated small colimits by the objects $V \in \mc{V}_{\kappa}$.
Since tensor products in $\spt[um = {\tau}, dm = {\tatesphere}]{S, \mc{V}}$ are cocontinuous separately in each variable, each $V \odot M \in \spt[um = {\tau}, dm = {\tatesphere}]{S, \mc{V}}$ as above belongs to the full \subqcategory\ generated under transfinitely iterated small colimits by the objects $V \odot \Sigma^{\infty}_{\tatesphere}\stwist{\yona[um = {\tau}]{S}{X}_+}{r}{s}$ with $X \in \sm[ft]{S}$, $\prns{r,s} \in \integer^2$ and $V \in \mc{V}_{\kappa}$.

Consider \nref[Claim]{SH.a8.3}.
By hypothesis, the objects $V \odot \loc[dm = {\tau}]{\yon{S}{X}} \in \hypsh[dm = {\tau}]{\sm[ft]{S}}{\mc{V}}$ with $X \in \sm[ft]{S}$ and $V \in \mc{V}_{\aleph_0}$ are $\aleph_0$-presentable (\nref{sheaves.11a.2}).
The proof of \cite[4.8.1.15]{Lurie_higher-algebra} shows that $\spc[um = {\tau}]{S, \mc{V}} = \spc[um = {\tau}]{S} \otimes \mc{V}$ is the localization of $\hypsh[dm = {\tau}]{\sm[ft]{S}}{\spc{}} \otimes \mc{V}$ with respect to the class of morphisms of the form $\id_V \odot f$ with $f \in \mf{W}_{\affine^1}$ as in \nref{SH.1.D} and with $V \in \mc{V}_{\aleph_0}$.
The domains and codomains of such morphisms are $\aleph_0$-presentable (\cite[5.3.2.11]{Lurie_higher-algebra}).
By \cite[5.5.7.3]{Lurie_higher-topos}, it follows that $\loc[dm = {\affine^1}]{}: \hypsh[dm = {\tau}]{\sm[ft]{S}}{\mc{V}} \to \spc[um = {\tau}]{S, \mc{V}}$ preserves $\aleph_0$-presentable objects.
In particular, each $V \odot \yona[um = {\tau}]{S}{X}$ with $X \in \sm[ft]{S}$ and $V \in \mc{V}_{\aleph_0}$ is $\aleph_0$-presentable.
By \nref{SH.b8}, the image of each such object under $\prns{-}_+: \spc[um = {\tau}]{S, \mc{V}} \to \spc[um = {\tau}, dm = {\pt}]{S, \mc{V}}$ is $\aleph_0$-presentable.
By \cite[Proposition~4.4.2]{Robalo_thesis}, it follows that the image of each such object under $\fct{\Sigma^{\infty}_{\tatesphere^1, \mc{V}}}{-}_+$ is $\aleph_0$-presentable.
The universal property of $\spt[um = {\tau}, dm = {\tatesphere}]{S, \mc{V}}^{\otimes} = \spt[um = {\tau}, dm = {\tatesphere}]{S}^{\wedge} \otimes \mc{V}^{\otimes}$ (\nref{SH.a4}) implies that it is canonically equivalent to the \locpres\ symmetric monoidal \qcategory\ obtained from $\spc[um = {\tau}, dm = {\pt}]{S, \mc{V}}^{\otimes}$ by formally adjoining a $\otimes$-inverse for the object $\tatesphere[S] \otimes \1{\mc{V}}$ (\cite[Definition~2.6]{Robalo_K-theory-and-the-bridge}).
The claim now follows from the argument given in \cite[C.12.(2)]{Hoyois_quadratic-refinement} once we remark that the objects $V \odot \Sigma^{\infty}_{\tatesphere}\stwist{\yona[um = {\tau}]{S}{X}_+}{r}{s}$ are $\aleph_0$-presentable.
\end{proof}

\begin{lemma}
\nlabel{SH.9}
Let $\mc{V}^{\otimes}$ be a \locpres\ \qcategory, $X$ a smooth, projective $S$-scheme, $\prns{r,s} \in \integer^2$, and $V \in \mc{V}$ an $\otimes$-dualizable object.
Then $V \odot \Sigma^{\infty}_{\tatesphere, \mc{V}}\stwist{\yona[um = {\tau}]{S, \mc{V}}{X}_+}{r}{s} \in \spt[um = {\tau}, dm = {\tatesphere}]{S, \mc{V}}$ is $\otimes$-dualizable.
\end{lemma}

\begin{proof}
Letting $\mc{V}^{\otimes} = \spc{}^{\wedge}$ and working with the Nisnevich topology, then the claim follows from \cite[2.2]{Riou_dualite-de-spanier-whitehead}.
In the general case, the canonical symmetric monoidal functor $\spt[u = {Nis}, dm = {\tatesphere}]{S}^{\wedge} \to \spt[um = {\tau}, dm = {\tatesphere}]{S, \mc{V}}^{\otimes}$ sends $\Sigma^{\infty}_{\tatesphere}\twist{\yona[u = {Nis}]{S}{X}_+}{r}$ to $\Sigma^{\infty}_{\tatesphere, \mc{V}}\twist{\yona[um = {\tau}]{S, \mc{V}}{X}_+}{r}$, and symmetric monoidal functors preserve $\otimes$-dualizable objects.
\end{proof}

\begin{prop}
\nlabel{SH.11}
Let $\mc{V}^{\otimes}$ be an ind-rigid symmetric monoidal \qcategory\ and $S = \spec{\kk}$ the spectrum of a perfect field.
Assume one of the following conditions is satisfied:
\begin{itemize}
\item
the field $\kk$ admits resolutions of singularities by blow-ups; or
\item
$\mc{V}^{\otimes}$ is $\integer_{\prns{\ell}}$-linear, where $\ell \in \integer$ is a prime different from the characteristic of $\kk$.
\end{itemize}
Then the following properties hold:
\begin{enumerate}
\item
\nlabel{SH.11.1}
$\spt[um = {\tau}, dm = {\tatesphere}]{S, \mc{V}}$ is generated under transfinitely iterated small colimits by the $\otimes$-dualizable objects of the form $V \odot \Sigma^{\infty}_{\tatesphere, \mc{V}}\stwist{\yona[um = {\tau}]{S, \mc{V}}{X}_+}{r}{s}$ with $X$ a smooth, projective $S$-scheme, $\prns{r,s} \in \integer^2$ and $V \in \mc{V}$ an $\otimes$-dualizable object; and
\item
\nlabel{SH.11.2}
if $\mc{V}$ is $\prns{\sm[ft]{S}, \tau}$-finite, then $\spt[um = {\tau}, dm = {\tatesphere}]{S, \mc{V}}^{\otimes}$ is ind-rigid.
\end{enumerate}
\end{prop}

\begin{proof}
Consider \nref[Claim]{SH.11.1}.
Let $\tau$ be the Nisnevich topology.
Letting $\mc{V}^{\otimes} = \spt{}^{\wedge}$, and assuming resolution of singularites over $\kk$, this is proved in \cite[1.4]{Riou_dualite-de-spanier-whitehead}.
Letting $\mc{V}^{\otimes} = \D{\mod[dm = {\integer_{\prns{\ell}}}]{}}^{\otimes}$, the claim is proved in \cite[B.1]{Levine-Yang-Zhao_algebraic-elliptic}.
For general $\tau$ and $\mc{V}^{\otimes}$, the symmetric monoidal functors
\[
\spt[u = {Nis}, dm = {\tatesphere}]{S}^{\wedge} 
  \to \spt[um = {\tau}, dm = {\tatesphere}]{S, \mc{V}}^{\otimes}
\quad\text{and}\quad
\spt[u = {Nis}, dm = {\tatesphere}]{S, \D{\mod[dm = {\integer_{\prns{\ell}}}]{}}}^{\otimes}
  \to \spt[um = {\tau}, dm = {\tatesphere}]{S, \mc{V}}^{\otimes}
\]
preserve $\otimes$-dualizable objects.
The proof of \nref{SH.a8.2} adapts readily to show that the objects $V \odot \Sigma^{\infty}_{\tatesphere, \mc{V}}\stwist{\yona[um = {\tau}]{S, \mc{V}}{X}_+}{r}{s}$ as above generate $\spt[um = {\tau}, dm = {\tatesphere}]{S, \mc{V}}$ under transfinitely iterated small colimits. 

To prove \nref[Claim]{SH.11.2}, note that the $\otimes$-dualizable objects of $\mc{V}^{\otimes}$ are precisely the $\aleph_0$-presentable ones.
The objects $V \odot \Sigma^{\infty}_{\tatesphere, \mc{V}}\stwist{\yona[um = {\tau}]{S, \mc{V}}{X}_+}{r}{s}$ are therefore $\aleph_0$-presentable by \nref{SH.a8} and it follows that they generate the full \subqcategory\ of $\aleph_0$-presentable objects under iterated finite colimits and retracts, as one can verify using the arguments of \cite[1.4.4.2]{Lurie_higher-algebra} and \cite[5.4.2.4]{Lurie_higher-topos} (cf. \cite[4.6]{Drew_verdier-quotients}).
As finite colimits and retracts of $\otimes$-dualizable objects are $\otimes$-dualizable, $\spt[um = {\tau}, dm = {\tatesphere}]{S, \mc{V}}^{\otimes}$ is ind-rigid.
\end{proof}

\section{Brown representability and reconstruction}
\nlabel{rep}

\setcounter{thm}{-1}

\begin{notation}
\nlabel{rep.0}
In this section, we fix the following notation and hypotheses:
\begin{itemize}
\item
$S$, a Noetherian scheme of finite dimension; and
\item
$\tau$, either the Nisnevich or the \'etale topology.
\end{itemize}
\end{notation}

\begin{motivation}
In this section, we introduce the \qcategory\ $\DH{\spec{\complex}}$ of motivic Hodge modules over $\spec{\complex}$.
As promised in \nref{intro}, we will use $\DH{\spec{\complex}}$ to define a full-fledge six-functor formalism $\DH{-}$ below in \nref{ex:DH-coefficient-syst}.
Here, we content ourselves to explore the situation over the point $\spec{\complex}$, providing the first bits of evidence for the tenability of our definition of $\DH{-}$.

The essential datum in the construction of $\DH{-}$ is that of a commutative algebra object of $\spt[um = {\tau}, dm = {\tatesphere}]{\spec{\complex}}$ that represents absolute Hodge cohomology. 
For the purposes of bookkeeping, it will be convenient to construct this absolute Hodge spectrum from something that, \emph{a priori}, provides more explicit control over the relevant mixed Hodge structures. 

The symmetric monoidal functor $\enhancedbettimon{}$ is \'etale local, $\affine^1$-invariant, and $\tatesphere$-stable (\nref{SH.a2}).
Not for nothing, however, we stress that $\enhancedbetti{}$ is contravariant, so some slight acrobatics are required before we may apply the universal property of $\spt[um = {\tau}, dm = {\tatesphere}]{\spec{\complex}}$ to extract from $\enhancedbettimon{}$ a symmetric monoidal realization functor
\[
\rho^{*, \otimes}_{\tu{Hdg}}:
  \spt[u = {\'et}, dm = {\tatesphere}]{\spec{\complex}}^{\wedge}_{\aleph_0}
    \to \mhc[um = {\tu{p}, \otimes}, dm = {\rational}].
\]
The absolute Hodge spectrum $\hodgespectrum$ that we seek is the image of the monoidal unit under the right adjoint $\rho_{\tu{Hdg},*}$ of $\rho^*_{\tu{Hdg}}$.
We then define $\DH{\spec{\complex}}$ as the \qcategory\ of modules over $\hodgespectrum$ in $\spt[u = {\'et}, dm = {\tatesphere}]{\spec{\complex}}^{\wedge}$.

In order to justify this definition, we conclude this section by showing that the realization functor $\rho^*_{\tu{Hdg}}$ induces a fully faithful functor $\DH{\spec{\complex}} \hookrightarrow \D[u = {b}]{\mhsp[\rational]}$.
\end{motivation}

\begin{summary}
Let $v^{\otimes}: \mc{V}^{\otimes} \to \mc{W}^{\otimes}$ be cocontinuous symmetric monoidal functor between stable symmetric monoidal \qcategories.
\begin{itemize}
\item
In \nref{rep.3}, we show that each $\tau$-local $\mc{W}^{\otimes}$-valued mixed Weil theory $\cohomologymon{}: \prns{\sm[ft]{S}}\op[\amalg] \to \mc{W}^{\otimes}$ induces an essentially unique $\mc{V}^{\otimes}$-linear cocontinuous symmetric monoidal realization functor $\spt[um = {\tau}, dm = {\tatesphere}]{S, \mc{V}}^{\otimes} \to \mc{W}^{\otimes}$.
\item
In \nref{rep.3.2}, we take this a step further to assign to $\cohomologymon{}$ a commutative algebra object $A$ of $\spt[um = {\tau}, dm = {\tatesphere}]{S, \mc{V}}^{\otimes}$ that represents the cohomology theory $\cohomologymon{}$.
By this process, we obtain the fundamental examples of the $\rational$-linear Betti spectrum and the absolute Hodge spectrum over $\spec{\complex}$ in \nref{defn:hodge-spectrum}.
\item
In \nref{rep.6}, we show that, under some technical assumptions, there is a fully faithful symmetric monoidal functor
$
\mod[dm = {A}]{\spt[um = {\tau}, dm = {\tatesphere}]{S, \mc{V}}}^{\otimes} \hookrightarrow \mc{W}^{\otimes},
$
where $A$ represents the cohomology theory $\cohomologymon{}$ as in the previous item.
In particular, the \qcategory\ of modules over the absolute Hodge spectrum is a full \subqcategory\ of $\D{\ind{\mhsp[\rational]}}$, fulfilling \nref[Desideratum]{desideratum.4}.
\item
We conclude this section with a few remarks about functoriality of these constructions (\nref{rep.7}) and enriched representability of mixed Weil theories (\nref{prop:enriched-representability}).
\end{itemize}
\end{summary}

\begin{prop}
\nlabel{rep.3}
\nlabel{rep.3.1}
Let $v^{\otimes}: \mc{V}^{\otimes} \to \mc{W}^{\otimes}$ be a cocontinuous symmetric monoidal functor between \locpres\ symmetric monoidal \qcategories\ with $\mc{W}^{\otimes}$ stable,
and $\iota: \mc{W}\rig \hookrightarrow \mc{W}$ the inclusion of the full \subqcategory\ spanned by the $\otimes$-dualizable objects.
There is an essentially commutative square
\[
\begin{tikzcd}
\map[dm = {\calg{\QCATmon}}]{\prns{\sm[ft]{S}}\op[\amalg]}{\mc{W}\rig^{\otimes}}_{\tu{Weil},\tau}\op
\ar[r, "i" below, hookrightarrow]
\ar[d, "\alpha" left]
&
\map[dm = {\calg{\QCATmon}}]{\prns{\sm[ft]{S}}\op[\amalg]}{\mc{W}\rig^{\otimes}}\op
\ar[d, "\beta" right]
\\
\map[dm = {\pr[um = {\tu{L}, \otimes}]_{\mc{V}^{\otimes}/}}]{\spt[um = {\tau}, dm = {\tatesphere}]{S, \mc{V}}^{\otimes}}{\mc{W}^{\otimes}}
\ar[r, "\prns{\Sigma^{\infty,\otimes}_{\tatesphere,\mc{V}}\prns{-}_+}^*" above, hookrightarrow]
&
\map[dm = {\calg{\QCATmon}}]{\prns{\sm[ft]{S}}^{\times}}{\mc{W}^{\otimes}},
\end{tikzcd}
\]
where $i$ is the inclusion of the Kan subcomplex spanned by the simplices whose vertices classify $\tau$-local $\mc{W}^{\otimes}\rig$-valued mixed Weil theories
and $\beta$ is composition with the involution $\prns{-}^{\vee, \otimes}$ of \nref{rep.2} and $\iota^{\otimes}$.
\end{prop}

\begin{proof}
By \nref{SH.a4}, $\prns{\Sigma^{\infty,\otimes}_{\tatesphere,\mc{V}}\prns{-}_+}^*$ is equivalent to the inclusion of the Kan subcomplex spanned by the simplices whose vertices classify $\tau$-local, $\affine^1$-local, $\tatesphere$-stable symmetric monoidal functors $\prns{\sm[ft]{S}}^{\times} \to \mc{W}^{\otimes}$.
To obtain the desired morphism $\alpha$, it therefore suffices to show that the essential image of $\beta i$ is contained in this Kan subcomplex.
 
Suppose a $0$-simplex of $\map[dm = {\calg{\QCATmon}}]{\prns{\sm[ft]{S}}\op[\amalg]}{\mc{W}\rig^{\otimes}}_{\tu{Weil}}\op$ classifies a $\mc{W}\rig^{\otimes}$-valued mixed Weil theory $\cohomologymon{}$.
Let $\homologymon{}$ denote the composite 
\[
\homologymon{}:
  \prns{\sm[ft]{S}}^{\times}
    \xrightarrow{\cohomology[um = {\varop[\otimes]}]{}} \mc{W}\rig\op[\otimes]
    \xrightarrow{\prns{-}^{\vee,\otimes}} \mc{W}^{\otimes}\rig
    \xrightarrow{\iota^{\otimes}} \mc{W}^{\otimes}.
\]
Note that $\homologymon{}$ is $\tau$-local by \nref{prop:descent}, and also $\affine^1$-invariant and $\tatesphere$-stable.
The $0$-simplex classifying $\homologymon{}$ is equivalent to the image of the $0$-simplex classifying $\cohomologymon{}$ under $\beta i$ and the claim follows.
\end{proof}

\begin{defn}
By \nref{rep.3}, each $\mc{W}^{\otimes}$-valued mixed Weil theory $\cohomologymon{}$ induces an essentially unique $\mc{V}^{\otimes}$-linear cocontinuous symmetric monoidal functor $\rho^{*, \otimes} \coloneqq \fct{\alpha}{\cohomologymon{}}: \spt[um = {\tau}, dm = {\tatesphere}]{S, \mc{V}}^{\otimes} \to \mc{W}^{\otimes}$ such that $\rho^*\Sigma^{\infty}_{\tatesphere, \mc{V}}\prns{X}_+ \simeq \cohomology{X}^{\vee}$ for each $X \in \sm[ft]{S}$.
We call $\rho^{*, \otimes}$ the \emph{symmetric monoidal realization functor associated with $\cohomologymon{}$}.
\end{defn}

\begin{rmk}
\nlabel{rep.4}
We intend to apply \nref{rep.3} to the functors $\bettimon{}$ and $\enhancedbettimon{}$ of \nref{sheaves.13}.
Thus, in the examples of interest below, the functor $v^{\otimes}: \mc{V}^{\otimes} \to \mc{W}^{\otimes}$ in \nref{rep.3} will be the symmetric monoidal functor 
\[
\alpha^{*, \otimes}:
\D{\mod[dm = {\rational}]{}}^{\otimes}
 \to \D{\ind{\mhsp[\rational]}}^{\otimes}
\]
induced by the functor that assigns to each $\rational$-module $V$ of finite rank $r$ the direct sum of $r$ copies of $\1{\mhsp[\rational]}$, 
or one of the identities
$\id_{\D{\mod[dm = {\rational}]{}}}^{\otimes}$
or
$\id_{\D{\ind{\mhsp[\rational]}}}^{\otimes}$.
\end{rmk}

\begin{defn}
\nlabel{ex:realizations}
Applying \nref{rep.3} to the mixed Weil theories $\bettimon{}$ and $\enhancedbettimon{}$ of \nref{ex:weil-theory}, we obtain realization functors over $\spec{\complex}$.
\begin{enumerate}
\item
Taking $v^{\otimes} = \id_{\D{\mod[dm = {\rational}]{}}}^{\otimes}$ and $\cohomologymon{} = \bettimon{}$ in \nref{rep.3}, we obtain a $\D{\mod[dm = {\rational}]{}}^{\otimes}$-linear cocontinuous symmetric monoidal functor
\[
\rho^{*, \otimes}_{\tu{B}}:
\spt[u = {\'et}, dm = {\tatesphere}]{S, \D{\mod[dm = {\rational}]{}}}^{\otimes}
  \to \D{\mod[dm = {\rational}]{}}^{\otimes},
\]
which we refer to as the \emph{$\rational$-linear Betti realization over $\spec{\complex}$}.
\item
Taking $v^{\otimes} = \alpha^{*, \otimes}$ as defined in \nref{rep.4} and $\cohomologymon{} = \enhancedbettimon{}$ as in \nref{rep.3}, we obtain a $\D{\mod[dm = {\rational}]{}}^{\otimes}$-linear cocontinuous symmetric monoidal functor
\[
\rho^{*, \otimes}_{\tu{{Hdg}}}:
\spt[u = {\'et}, dm = {\tatesphere}]{S, \D{\mod[dm = {\rational}]{}}}^{\otimes}
  \to \D{\ind{\mhsp[\rational]}}^{\otimes},
\]
which we refer to as the \emph{mixed Hodge realization over $\spec{\complex}$}.
\item
Taking $v^{\otimes} = \id_{\D{\ind{\mhsp[\rational]}}}^{\otimes}$ and $\cohomologymon{} = \enhancedbettimon{}$, we obtain a $\D{\ind{\mhsp[\rational]}}^{\otimes}$-linear cocontinuous symmetric monoidal functor
\[
\rho^{*, \otimes}_{\tb{{Hdg}}}:
\spt[u = {\'et}, dm = {\tatesphere}]{S, \D{\ind{\mhsp[\rational]}}}^{\otimes}
  \to \D{\ind{\mhsp[\rational]}}^{\otimes},
\]
which we refer to as the \emph{$\mhsp[\rational]$-linear mixed Hodge realization over $\spec{\complex}$}.
\end{enumerate}
\end{defn}

\begin{defn}
\nlabel{rep.1}
Let $\mc{W}^{\otimes}$ be a stable symmetric monoidal \qcategory\ and $\cohomologymon{}: \prns{\sm[ft]{S}}\op[\amalg] \to \mc{W}^{\otimes}$ a symmetric monoidal functor. 
If $\sigma_1: S \to \Gm{S}$ denotes the unit section, then we let $\tate{\mc{W}}{-1}_{\cohomology{}} \in \mc{W}$ denote the $\sphere^1$-suspension $\fib{\cohomology{\sigma_1}}\sus{1}$ of the fiber of $\cohomology{\sigma}$, so we have a fiber sequence
\[
\fib{\cohomology{\sigma_1}}
  \to \cohomology{S}
  \xrightarrow{\cohomology{\sigma_1}} \cohomology{\Gm{S}}.
\]
We refer to $\tate{\mc{W}}{-1}_{\cohomology{}}$ as the \emph{Tate object in $\mc{W}^{\otimes}$ with respect to $\cohomologymon{}$}.
By \nref{SH.a2}, $\cohomologymon{}$ is $\tatesphere$-stable if and only if $\tate{\mc{W}}{-1}_{\cohomology{}}$ is $\otimes$-invertible.
In that case, we set
\[
\twist{W}{r}_{\cohomology{}} \coloneqq W \otimes \prns{\tate{\mc{W}}{-1}_{\cohomology{}}}^{\otimes(-r)}
\]
for each $W \in \mc{W}$ and each $r \in \integer$.
\end{defn}

\begin{ex}
\nlabel{ex:tate-twist}
Applying \nref{rep.1} to the functors of \nref{sheaves.13}, we recover well-known objects by the usual computations of the cohomology of $\Gm{\complex}$.
\begin{enumerate}
\item
The Tate object $\tate{\mod[dm = {\rational}]{}}{-1}_{\betti{}}$ associated with $\bettimon{}$ as in \nref{sheaves.13a} is isomorphic to a $\rational$-module of rank $1$ regarded as a cochain complex concentrated in degree $0$ by the usual computation of the Betti cohomology of $\Gm{\complex}$. 
\item
The Tate object $\tate{\mhsp[\rational]}{-1}_{\enhancedbetti{}}$ associated with $\enhancedbettimon{}$ as in \nref{sheaves.13b} is isomorphic to the Tate-Hodge structure $\twist{\rational}{-1}$ of weight $2$ regarded as a cochain complex concentrated in degree $0$.
\end{enumerate}
\end{ex}

\begin{defn}
\nlabel{defn:absolute-theory}
Let $\mc{W}^{\otimes}$ be a stable \locpres\ symmetric monoidal \qcategory, $W \in \mc{W}$, and $\cohomologymon{}: \prns{\sm[ft]{S}}\op[\amalg] \to \mc{W}^{\otimes}$ a $\mc{W}^{\otimes}$-valued mixed Weil theory.
We obtain a bigraded \emph{absolute cohomology theory associated with $\cohomologymon{}$ with coefficients in $W$} given by 
\[
\absolutecohomology[um = {s}]{X, \twist{W}{r}}
  \coloneqq \pi_0\map[dm = {\mc{W}}]{\twist{W}{-r}_{\cohomology{}}}{\cohomology{X}\sus{s}}.
\]
\end{defn}

\begin{ex}
\nlabel{ex:absolute-theory}
Since $\mod[dm = {\rational}]{}$ is semisimple and the Tate objects associated with $\bettimon{}$ are equivalent to $\1{\mod[dm = {\rational}]{}}$, $\rational$-linear Betti cohomology is naturally equivalent to the associated absolute cohomology theory with coefficients in $\rational$ once we fix the degree $r$ of the twist.
On the other hand, $\mhsp[\rational]$ is not semisimple and the Tate-Hodge objects of different degrees are not mutually isomorphic.
The absolute cohomology theories associated with $\enhancedbettimon{}$ are therefore interesting.
By definition, they recover the classical definition of absolute Hodge cohomology given in \cite{Beilinson_absolute-hodge}.
\end{ex}

\begin{prop}
\nlabel{rep.3.2}
Let $v^{\otimes}: \mc{V}^{\otimes} \to \mc{W}^{\otimes}$ be a cocontinuous symmetric monoidal functor between \locpres\ symmetric monoidal \qcategories\ with $\mc{W}^{\otimes}$ stable,
$\cohomologymon{}: \prns{\sm[ft]{S}}\op[\amalg] \to \mc{W}^{\otimes}$ a $\tau$-local $\mc{W}^{\otimes}$-valued mixed Weil theory,
and $\rho^{*, \otimes}: \spt[um = {\tau}, dm = {\tatesphere}]{S, \mc{V}}^{\otimes} \to \mc{W}^{\otimes}$ the associated symmetric monoidal realization functor.
The following properties hold:
\begin{enumerate}
\item
\nlabel{rep.3.2a}
$\rho^{*,\otimes}$ admits a lax symmetric monoidal right adjoint $\rho^{\otimes}_*$; and
\item
\nlabel{rep.3.2b}
if $A \coloneqq \rho_*\1{\mc{W}} \in \calg{\spt[um = {\tau}, dm = {\tatesphere}]{S,\mc{V}}^{\otimes}}$, then there is a natural equivalence
\begin{equation}
\nlabel{rep.3.a1}
\map[dm = {\spt[um = {\tau}, dm = {\tatesphere}]{S, \mc{V}}}]{V \odot \Sigma^{\infty}_{\tatesphere}\yona[um = {\tau}]{S}{X}}{\stwist{A}{r}{s}}
\simeq
\map[dm = {\mc{W}}]{\fct{v}{V}}{\twist{\cohomology{X}}{r}_{\cohomology{}}\sus{s}}
=
\
\end{equation}
for each $X \in \sm[ft]{S}$, each $V \in \mc{V}$, and each $\prns{r,s} \in \integer^2$.
\end{enumerate}
In particular, $A$ represents the absolute cohomology theory associated with $\cohomologymon{}$ with coefficients in $\1{\mc{W}}$, i.e., there is a natural equivalence
\[
\absolutecohomology[um = {s}]{X, \tate{\mc{W}}{r}}
  \simeq \pi_0\map[dm = {\mc{W}}]{\tate{\mc{W}}{-r}_{\cohomology{}}}{\cohomology{X}\sus{s}}
\]
for each $X \in \sm[ft]{S}$ and each $\prns{r,s} \in \integer^2$.
\end{prop}

\begin{proof}
\nref[Claim]{rep.3.2a} follows from the Adjoint Functor Theorem (\cite[5.5.2.9]{Lurie_higher-topos}) and \cite[7.3.2.7]{Lurie_higher-algebra}.

Consider \nref[Claim]{rep.3.2b}.
We start with the following observation.
If $\sigma_1: S \to \Gm{S}$ denotes the unit section, then, by construction of $\rho^*$, we have
\begin{align}
\nlabel{rep.3.a}
\begin{split}
\fct{\rho^*}{\state{S, \mc{V}}{1}{1}}
  & = \fct{\rho^*}{ \cofib{\prns{\Sigma^{\infty}_{\tatesphere, \mc{V}} \yona[um = {\tau}]{S, \mc{V}}{\sigma_1}}}}
\\
&\simeq \cofib{\fct{\rho^*}{\Sigma^{\infty}_{\tatesphere, \mc{V}}\yona[um = {\tau}]{S, \mc{V}}{\sigma_1} }}
\\
&\simeq \cofib{ \homology{\sigma_1} }
\\
&\simeq \fib{\cohomology{\sigma_1}}^{\vee}
\\
&\simeq \prns{\tate{\mc{W}}{-1}_{\cohomology{}} \brk{-1}}^{\vee}
\\
&\simeq \tate{\mc{W}}{1}_{\cohomology{}}\brk{1}.
\end{split}
\end{align} 
As $\rho^{*,\otimes}$ is symmetric monoidal and, hence, preserves $\otimes$-inverses, this implies that the object $\fct{\rho^*}{ \tate{S, \mc{V}}{r} } \simeq \tate{\mc{W}}{r}_{\cohomology{}}$ is $\otimes$-invertible for each $r \in \integer$.
It follows that $\rho^*$ sends the object $V \odot \Sigma^{\infty}_{\tatesphere, \mc{V}}\twist{\yona[um = {\tau}]{S, \mc{V}}{X}}{r}$ to the object $\fct{v}{V} \otimes \twist{\homology{X}}{r}_{\cohomology{}}$ for each $X \in \sm[ft]{S}$, each $r \in \integer$, and each $V \in \mc{V}$.

Let $X \in \sm[ft]{S}$ and $\prns{r,s} \in \integer^2$.
We have
\begin{align*}
& \map[dm = {\spt[um = {\tau}, dm = {\tatesphere}]{S, \mc{V}}}]{V \odot \Sigma^{\infty}_{\tatesphere, \mc{V}} \yona[um = {\tau}]{S, \mc{V}}{X}}{\stwist{A}{r}{s}}
\\
& \qquad \simeq \map[dm = {\spt[um = {\tau}, dm = {\tatesphere}]{S, \mc{V}}}]{V \odot \Sigma^{\infty}_{\tatesphere, \mc{V}} \stwist{\yona[um = {\tau}]{S, \mc{V}}{X}}{-r}{-s}}{\rho_*\1{\mc{W}}}
\\
& \qquad \simeq \map[size = 0, dm = {\mc{W}}]{\fct{\rho^*}{V \odot \Sigma^{\infty}_{\tatesphere, \mc{V}} \stwist{\yona[um = {\tau}]{S, \mc{V}}{X}}{-r}{-s}}}{\1{\mc{W}}}
&&
\text{adjunction}
\\
& \qquad \simeq \map[size = 0, dm = {\mc{W}}]{\fct{v}{V} \otimes \fct{\rho^*}{ \Sigma^{\infty}_{\tatesphere, \mc{V}} \yona[um = {\tau}]{S, \mc{V}}{X}} \otimes \fct{\rho^*}{\state{S}{-r}{-s}}}{\1{\mc{W}}}
&&
\rho^{*, \otimes} \text{ monoidal, $\mc{V}^{\otimes}$-linear}
\\
& \qquad \simeq \map[dm = {\mc{W}}, size = 0]{\fct{v}{V} \otimes \twist{\homology{X}}{-r}_{\cohomology{}}\sus{-s}}{\1{\mc{W}}}
&&
\text{\nref{rep.3.a}, \nref{rep.3.1}}
\\
&\qquad \simeq \map[dm = {\mc{W}}]{\fct{v}{V} \otimes \twist{\cohomology{X}^{\vee}}{-r}_{\cohomology{}}\sus{-s}}{\1{\mc{W}}}
&&
\text{construction of }\homology{}
\\
& \qquad \simeq \map[dm = {\mc{W}}]{\fct{v}{V}}{\twist{\cohomology{X}}{r}_{\cohomology{}}\sus{s}}
&&
\text{duality}
\end{align*}
and the claim follows.
\end{proof}

\begin{defn}
\nlabel{defn:hodge-spectrum}
Applying \nref{rep.3.2} to the realization functors constructed in \nref{ex:realizations}, we obtain the following commutative algebra objects:
\begin{enumerate}
\item
the \emph{$\rational$-linear Betti spectrum over $\spec{\complex}$} given by
\[
\bettispectrum
  \coloneqq \rho_{\tu{B},*}\1{\mod[dm = {\rational}]{}}
  \in \calg{\spt[u = {\'et}, dm = {\tatesphere}]{S, \D{\mod[dm = {\rational}]{}}}^{\otimes}};
\]
\item
the \emph{absolute Hodge spectrum over $\spec{\complex}$} given by
\[
\hodgespectrum
  \coloneqq \rho_{\tu{Hdg},*}\1{\mhsp[\rational]}
  \in \calg{\spt[u = {\'et}, dm = {\tatesphere}]{S, \D{\mod[dm = {\rational}]{}}}^{\otimes}}; \quad\text{and}
\] 
\item
the \emph{$\mhsp[\rational]$-linear absolute Hodge spectrum over $\spec{\complex}$} given by
\[
\varhodgespectrum
  \coloneqq \rho_{\tb{Hdg},*}\1{\mhsp[\rational]}
  \in \calg{\spt[u = {\'et}, dm = {\tatesphere}]{S, \D{\ind{\mhsp[\rational]}}}^{\otimes}}.
\]
\end{enumerate}
Both $\hodgespectrum$ and $\varhodgespectrum$ represent absolute Hodge cohomology by \nref{rep.3.2b}.
\end{defn}

\begin{lemma}[Reconstruction]
\nlabel{rep.5}
Let $\phi^{*,\otimes}: \mc{V}^{\otimes} \to \mc{W}^{\otimes}$ be a cocontinuous symmetric monoidal functor of \locpres\ symmetric monoidal \qcategories, and $\phi^{\otimes}_*: \mc{W}^{\otimes} \to \mc{V}^{\otimes}$ a lax symmetric monoidal functor right adjoint to $\phi^{*,\otimes}$.
The following properties hold:
\begin{enumerate}
\item
\nlabel{rep.5.1}
$\phi^{*,\otimes}$ factors canonically as
$
\mc{V}^{\otimes}
  \xrightarrow{\prns{-}\otimes \phi_*\1{\mc{W}}}
    \mod[dm = {\phi_*\1{\mc{W}}}]{\mc{V}}^{\otimes}
  \xrightarrow{\tilde{\phi}^{*,\otimes}}
    \mc{W}^{\otimes};
$
\item 
\nlabel{rep.5.2}
the restriction of $\tilde{\phi}^*$ to the full \subqcategory\ of $\mod[dm = {\phi_*\1{\mc{W}}}]{\mc{V}}$ spanned by $\otimes$-dualizable objects is fully faithful; and
\item
\nlabel{rep.5.3}
if $\mc{V}^{\otimes}$ is ind-rigid and $\phi^*$ preserves $\aleph_0$-presentable objects, then $\tilde{\phi}^*$ is fully faithful.
\end{enumerate}
\end{lemma}

\begin{proof}
The lax symmetric monoidal functors $\phi^{*,\otimes}$ and $\phi^{\otimes}_*$ induce functors between $\calg{\mc{V}^{\otimes}}$ and $\calg{\mc{W}^{\otimes}}$.
In particular, we can promote $\phi^*\phi_*\1{\mc{W}}$ to an object of $\calg{\mc{W}^{\otimes}}$.
We deduce a homotopy-commutative diagram
\[
\begin{tikzcd}
\mc{V}^{\otimes}
\ar[r, "\phi^{*,\otimes}" above]
\ar[d, "\prns{-}\otimes \phi_*\1{\mc{W}}" left]
&
\mc{W}^{\otimes}
\ar[drr, "\prns{-}\otimes \1{\mc{W}}" above, "\sim" below, bend left=15pt]
\ar[d, "\prns{-}\otimes \phi^*\phi_*\1{\mc{W}}" left]
&
&
\\
\mod[dm = {\phi_*\1{\mc{W}}}]{\mc{V}}^{\otimes}
\ar[r, "\phi^{*,\otimes}" below]
&
\mod[dm = {\phi^*\phi_*\1{\mc{W}}}]{\mc{W}}^{\otimes}
\ar[rr, "\prns{-}\otimes_{\phi^*\phi_*\1{\mc{W}}}\1{\mc{W}}" below]
&
&
\mod[dm = {\1{\mc{W}}}]{\mc{W}}^{\otimes}
\end{tikzcd}
\]
of symmetric monoidal functors, where $\1{\mc{W}}$ is equipped with the commutative $\phi^*\phi_*\1{\mc{W}}$-algebra structure induced by the counit of $\phi^*\dashv \phi_*$.
The functor $\prns{-}\otimes \1{\mc{W}}$ is an equivalence by \cite[3.4.2.1]{Lurie_higher-algebra}, and \nref[Claim]{rep.5.1} follows, letting $\tilde{\phi}^{*,\otimes}$ denote the composite of the two lower horizontal arrows and an inverse to $\prns{-}\otimes \1{\mc{W}}$.

Let $\prns{M, N} \in \mod[dm = {\phi_*\1{\mc{W}}}]{\mc{V}}^2$ with $N$ an $\otimes$-dualizable object.
As $\tilde{\phi}^*$ preserves tensor products and $\otimes$-duals, we have a homotopy-commutative square
\[
\begin{tikzcd}
\map[size = 0, dm = {\mod[dm = {\phi_*\1{\mc{W}}}]{\mc{V}}^{\otimes}}]{M\otimes_{\phi_*\1{\mc{W}}} \prns{N^{\vee}}}{\phi_*\1{\mc{W}}}
\ar[r, "\tilde{\phi}^*" below]
\ar[d]
&
\map[size = 0, dm = {\mc{W}}]{\tilde{\phi}^*M\otimes \prns{\tilde{\phi}^*N^{\vee}}}{\1{\mc{W}}}
\ar[d]
\\
\map[dm = {\mod[dm = {\phi_*\1{\mc{W}}}]{\mc{V}}^{\otimes}}]{M}{N}
\ar[r, "\tilde{\phi}^*" above]
&
\map[dm = {\mc{W}}]{\tilde{\phi}^*M}{\tilde{\phi}^*N}
\end{tikzcd}
\]
in which the vertical arrows are the equivalences by duality.
Moreover, the upper horizontal arrow is an equivalence by adjunction, since $\1{\mod[dm = {\phi_*\1{\mc{W}}}]{\mc{V}}} \simeq \phi_*\1{\mc{W}} \simeq \tilde{\phi}_*\1{\mc{W}}$ in $\mod[dm = {\phi_*\1{\mc{W}}}]{\mc{V}}$.
This establishes \nref[Claim]{rep.5.2}.

If $\mc{V}^{\otimes}$ is ind-rigid, then $\mod[dm = {\phi_*\1{\mc{W}}}]{\mc{V}}$ is the ind-completion of its full \subqcategory\ spanned by the $\aleph_0$-presentable objects, each of which is $\otimes$-dualizable, so $\mod[dm = {\phi_*\1{\mc{W}}}]{\mc{V}}^{\otimes}$ is also ind-rigid.
By \nref[Claim]{rep.5.2}, the restriction of $\tilde{\phi}^*$ to the full \subqcategory\ of $\aleph_0$-presentable objects is fully faithful, so \nref[Claim]{rep.5.3} follows from \cite[5.3.5.11(1)]{Lurie_higher-topos}.
\end{proof}

\begin{cor}
\nlabel{rep.6a}
With the notation and hypotheses of \tu{\varnref{rep.3.2}}, $\rho^{*, \otimes}$ factors canonically as
\[
\spt[um = {\tau}, dm = {\tatesphere}]{S, \mc{V}}^{\otimes}
  \xrightarrow{\prns{-} \otimes A} \mod[dm = {A}]{\spt[um = {\tau}, dm = {\tatesphere}]{S, \mc{V}}}^{\otimes}
  \xrightarrow{\tilde{\rho}^{*, \otimes}} \mc{W}^{\otimes}
\]
with $\tilde{\rho}^{*, \otimes}$ a cocontinuous symmetric monoidal functor.
\end{cor}

\begin{proof}
This is a special case of \nref{rep.5.1}.
\end{proof}

\begin{cor}
\nlabel{rep.6}
With the notation and hypotheses of \tu{\varnref{rep.3.2}} and \tu{\nref{rep.6a}}, assume further that:
\begin{itemize}
\item
$S = \spec{\kk}$ is the spectrum of a perfect field;
\item
$\mc{V}^{\otimes}$ is ind-rigid and $\prns{\sm[ft]{S}, \tau}$-finite; and
\item
if $\kk$ is of characteristic $p > 0$, then $\mc{V}^{\otimes}$ is $\integer_{\prns{\ell}}$-linear with $\ell \in \integer$ prime and distinct from $p$.
\end{itemize}
Then the following properties hold:
\begin{enumerate}
\item
\nlabel{rep.6.1}
the symmetric monoidal functor $\tilde{\rho}^{*, \otimes}: \mod[dm = {A}]{\spt[um = {\tau}, dm = {\tatesphere}]{S, \mc{V}}}^{\otimes} \to \mc{W}^{\otimes}$ is fully faithful; and
\item
\nlabel{rep.6.2}
$\tilde{\phi}^{*, \otimes}$ is moreover an equivalence if $v: \mc{V} \to \mc{W}$ is essentially surjective.
\end{enumerate}
\end{cor}

\begin{proof}
\nref[Claim]{rep.6.1} follows immediately from \nref{rep.3}, \nref{rep.5}, and \nref{SH.11}.

For the essential surjectivity in \nref[Claim]{rep.6.2}, notice that $\fct{\tilde{\rho}^*}{V \odot \1{S,A}} \simeq V \odot \fct{\tilde{\rho}^*}{\1{S,A}} \simeq \fct{v}{V}$ for each $V \in \mc{V}$, where $\1{S,A} \in \mod[dm = {A}]{\spt[um = {\tau}, dm = {\tatesphere}]{S, \mc{V}}}$ denotes the monoidal unit.
Indeed, $\tilde{\rho}^*$ is $\mc{V}^{\otimes}$-linear and symmetric monoidal.
\end{proof}

\begin{defn}
\nlabel{defn:hodge-modules}
Note that $\D{\mod[dm = {\rational}]{}}^{\otimes}$ and $\D{\ind{\mhsp[\rational]}}^{\otimes}$ are both $\prns{\sm[ft]{S}, \tau}$-finite by \nref{sheaves.12}.
Applying \nref{rep.6} to the mixed Weil theories $\bettimon{}$ and $\enhancedbettimon{}$, we obtain the following ind-rigid symmetric monoidal \qcategories:
\begin{enumerate}
\item
$\DB{\spec{\complex}}^{\otimes} \coloneqq \mod[dm = {\bettispectrum}]{\spt[u = {\'et}, dm = {\tatesphere}]{S, \D{\mod[dm = {\rational}]{}}}}^{\otimes}$, the symmetric monoidal \qcategory\ of \emph{motivic Betti modules over $\spec{\complex}$};
\item
$\DH{\spec{\complex}}^{\otimes} \coloneqq \mod[dm = {\hodgespectrum}]{\spt[u = {\'et}, dm = {\tatesphere}]{S, \D{\mod[dm = {\rational}]{}}}}^{\otimes}$, the symmetric monoidal \qcategory\ of \emph{motivic Hodge modules over $\spec{\complex}$}; and
\item
$\varDH{\spec{\complex}}^{\otimes} \coloneqq \mod[dm = {\varhodgespectrum}]{\spt[u = {\'et}, dm = {\tatesphere}]{S, \D{\ind{\mhsp[\rational]}}}}^{\otimes}$, the symmetric monoidal \qcategory\ of \emph{$\mhsp[\rational]$-motivic Hodge modules over $\spec{\complex}$}.
\end{enumerate}
By \nref{rep.6}, the realization functors of \nref{ex:realizations} induce fully faithful functors
\begin{align*}
\tilde{\rho}^*_{\tu{B}}:
\DB{\spec{\complex}} 
&\hookrightarrow \D{\mod[dm = {\rational}]{}} \\
\tilde{\rho}^*_{\tu{Hdg}}:
\DH{\spec{\complex}}
&\hookrightarrow \D{\ind{\mhsp[\rational]}} \\
\tilde{\rho}^*_{\tb{Hdg}}:
\varDH{\spec{\complex}}
&\hookrightarrow \D{\ind{\mhsp[\rational]}}.
\end{align*}
Furthermore, $\tilde{\rho}^*_{\tb{Hdg}}$ is an equivalence by \nref{rep.6.2}.
Also, $\D{\mod[dm = {\rational}]{}}$ is generated under transfinitely iterated colimits by the monoidal unit, which belongs to the essential image of the cocontinuous functor $\tilde{\rho}^*_{\tu{B}}$, and it follows that $\tilde{\rho}^*_{\tu{B}}$ is also an equivalence.
\end{defn}

\begin{rmk}[Naturality]
\nlabel{rep.7}
The constructions of \nref[Propositions]{rep.3} and \varnref{rep.3.2} are natural in the following sense.
Consider a commutative square
\[
\begin{tikzcd}
\mc{V}^{\otimes}
\ar[r, "v^{\otimes}" below]
\ar[d, "\phi^{*, \otimes}" left]
&
\mc{W}^{\otimes}
\ar[d, "\psi^{*, \otimes}" right]
\\
\mc{V}'^{\otimes}
\ar[r, "v'^{\otimes}" above]
&
\mc{W}'^{\otimes}
\end{tikzcd}
\]
of cocontinuous symmetric monoidal functors between \locpres\ symmetric monoidal \qcategories.
Suppose both $\mc{W}^{\otimes}$ and $\mc{W}'^{\otimes}$ are stable.
Consider furthermore a $\mc{W}^{\otimes}$-valued mixed Weil theory $\cohomologymon{}: \prns{\sm[ft]{\complex}}\op[\amalg] \to \mc{W}^{\otimes}$.
Note that $\psi^{*, \otimes}\cohomologymon{}$ is a $\mc{W}'^{\otimes}$-valued mixed Weil theory.
Let $A \in \calg{\spt[um = {\tau}, dm = {\tatesphere}]{S, \mc{V}}^{\otimes}}$ and $B \in \calg{\spt[um = {\tau}, dm = {\tatesphere}]{S, \mc{V}'}^{\otimes}}$ be the commutative algebra objects associated with $\cohomologymon{}$ and $\psi^{*, \otimes}\cohomologymon{}$, respectively, as in \nref{rep.3.2}.
\begin{enumerate}
\item
The symmetric monoidal functor $\phi^{*, \otimes}: \mc{V}^{\otimes} \to \mc{V}'^{\otimes}$ induces a symmetric monoidal functor
\[
\phi^{*, \otimes}:
  \spt[um = {\tau}, dm = {\tatesphere}]{S, \mc{V}}^{\otimes}
    \to \spt[um = {\tau}, dm = {\tatesphere}]{S, \mc{V}'}^{\otimes}
\]
as observed in \nref{SH.a7}.
By the universal property (\nref{SH.a4}), the functors constructed in \nref{rep.3.1} provide the horizontal arrows in the following essentially commutative square
\[
\begin{tikzcd}
\spt[um = {\tau}, dm = {\tatesphere}]{S, \mc{V}}^{\otimes}
\ar[r, "\rho^{*, \otimes}" below]
\ar[d, "\phi^{*, \otimes}" left]
&
\mc{W}^{\otimes}
\ar[d, "\psi^{*, \otimes}" right]
\\
\spt[um = {\tau}, dm = {\tatesphere}]{S, \mc{V}'}^{\otimes}
\ar[r, "G^{*, \otimes}" above]
&
\mc{W}'^{\otimes}
\end{tikzcd}
\]
There is a canonical morphism $\rho: A \to \phi_*B$ in $\calg{\spt[um = {\tau}, dm = {\tatesphere}]{S, \mc{V}}^{\otimes}}$, where $\phi_*$ underlies a lax symmetric monoidal functor right adjoint to $\phi^{*, \otimes}$.
Specifically, letting $\phi_*$ be a right adjoint for $\phi^*$, $\rho$ is the composite
\[
\rho:
A \coloneqq \rho_*\1{\mc{W}}
  \simeq \rho_*\rho^*\1{S, \mc{V}}
  \to \rho_* \psi_*\psi^*\rho^*\1{S, \mc{V}}
  \simeq \phi_*G_* G^*\phi^*\1{S, \mc{V}}
  \simeq \phi_*G_*\1{\mc{W}'}
  \eqqcolon \phi_*B
\] 
of the canonical equivalences with the unit of the adjunction $\psi^* \dashv \psi_*$.
The structure of a morphism of commutative algebras results from the lax symmetric monoidal structure $\psi_*$ inherits from $\psi^{*, \otimes}$.
We now have a sequence of symmetric monoidal functors
\[
\mod[dm = {A}]{\spt[um = {\tau}, dm = {\tatesphere}]{S,\mc{V}}}^{\otimes}
  \xrightarrow{\rho^{\otimes}} \mod[dm = {\phi_*B}]{\spt[um = {\tau}, dm = {\tatesphere}]{S,\mc{V}}}^{\otimes}
  \xrightarrow{\phi^{*, \otimes}} \mod[dm = {\phi^*\phi_*B}]{\spt[um = {\tau}, dm = {\tatesphere}]{S,\mc{V}'}}^{\otimes}
  \xrightarrow{\varepsilon^{\otimes}} \mod[dm = {B}]{\spt[um = {\tau}, dm = {\tatesphere}]{S,\mc{V}'}}^{\otimes},
\]
where $\varepsilon: \phi^*\phi_*B \to B$ denotes the counit of the adjunction $\phi^* \dashv \phi_*$.
\item
If $\psi^{*, \otimes}$ is fully faithful, the unit $\id \to \psi_*\psi^*$ is an equivalence and $\rho: A \to \phi_*B$ is therefore an equivalence.
In particular, it follows that the scalar-extension functor
\[
\mod[dm = {A}]{\spt[um  = {\tau}, dm = {\tatesphere}]{S, \mc{V}}}^{\otimes}
  \to \mod[dm = {\phi_*B}]{\spt[um = {\tau}, dm = {\tatesphere}]{S, \mc{V}}}^{\otimes}
\]
is an equivalence.
If $S$, $\mc{V}^{\otimes}$ and $\mc{V}'^{\otimes}$ satisfy moreover the hypotheses of \nref{rep.6}, then the composite functor
\[
\mod[dm = {A}]{\spt[um = {\tau}, dm = {\tatesphere}]{S,\mc{V}}}^{\otimes}
  \to \mod[dm = {\phi_*B}]{\spt[um = {\tau}, dm = {\tatesphere}]{S, \mc{V}}}^{\otimes}
  \to \mod[dm = {B}]{\spt[um = {\tau}, dm = {\tatesphere}]{S,\mc{V}'}}^{\otimes},
\]
constructed above is fully faithful.
Indeed, the functor
\[
\spt[um = {\tau}, dm = {\tatesphere}]{S, \mc{V}}^{\otimes}
  \to \mod[dm = {B}]{\spt[um = {\tau}, dm = {\tatesphere}]{S,\mc{V}'}}^{\otimes}
\]
is a $\tau$-local $\mod[dm = {B}]{\spt[um = {\tau}, dm = {\tatesphere}]{S,\mc{V}'}}^{\otimes}$-valued mixed Weil theory, so the claim follows from \nref{rep.6} and the universal property (\nref{SH.a4}).
\end{enumerate}
\end{rmk}

\begin{ex}
\nlabel{ex:naturality}
Applying \nref{rep.7} to the square
\[
\begin{tikzcd}
\D{\mod[dm = {\rational}]{}}^{\otimes}
\ar[r]
\ar[d]
&
\D{\ind{\mhsp[\rational]}}^{\otimes}
\ar[d, equals]
\\
\D{\ind{\mhsp[\rational]}}^{\otimes}
\ar[r, equals]
&
\D{\ind{\mhsp[\rational]}}^{\otimes},
\end{tikzcd}
\]
if $\phi_*: \spt[u = {\'et}, dm = {\tatesphere}]{S, \D{\ind{\mhsp[\rational]}}} \to \spt[u = {\'et}, dm = {\tatesphere}]{S, \D{\mod[dm = {\rational}]{}}}$ is right adjoint to the functor induced by $\alpha^{*, \otimes}$ (\nref{rep.4}), then we obtain a morphism of commutative algebras $\hodgespectrum \to \phi_*\varhodgespectrum$ as well as a fully faithful, cocontinuous symmetric monoidal functor $\DH{\spec{\complex}}^{\otimes} \hookrightarrow \varDH{\spec{\complex}}^{\otimes}$.

Similar arguments provide us with a morphisms of commutative algebras $\hodgespectrum \to \bettispectrum$ and cocontinuous symmetric monoidal functors
\[
\DH{\spec{\complex}}^{\otimes}
  \hookrightarrow \varDH{\spec{\complex}}^{\otimes}
  \to \DB{\spec{\complex}}^{\otimes}.
\]
Identifying $\DH{\spec{\complex}}$ and $\DB{\spec{\complex}}$ with full \subqcategories\ of $\D{\ind{\mhsp[\rational]}}$ and $\D{\mod[dm = {\rational}]{}}$, respectively, the composite corresponds to the fiber functor $\omega^*: \D{\ind{\mhsp[\rational]}} \to \D{\mod[dm = {\rational}]{}}$.
\end{ex}

\begin{prop}
\nlabel{prop:enriched-representability}
With the notation and hypotheses of \tu{\nref{rep.3.2}} and \tu{\nref{rep.6a}}:
\begin{enumerate}
\item
\nlabel{prop:enriched-representability.1}
$\spt[um = {\tau}, dm = {\tatesphere}]{S, \mc{V}}$ and $\mod[dm = {A}]{\spt[um = {\tau}, dm = {\tatesphere}]{S, \mc{V}}}$ inherit $\mc{W}^{\otimes}$-enriched-\qcategory\ structures from $\rho^{*, \otimes}$ and $\tilde{\rho}^{*, \otimes}$, respectively; and
\item
\nlabel{prop:enriched-representability.2}
under the hypotheses of \tu{\nref{rep.6}},
the $\mc{W}^{\otimes}$-valued mixed Weil theory $\cohomology{}: \prns{\sm[ft]{S}}\op \to \mc{W}$ is representable in $\mod[dm = {A}]{\spt[um = {\tau}, dm = {\tatesphere}]{S, \mc{V}}}$ in the following $\mc{W}^{\otimes}$-enriched sense:
there is a natural equivalence
\[
\twist{\cohomology{X}}{r}_{\cohomology{}}
  \simeq \mor[um = {\mc{W}}, dm = {\mod[dm = {A}]{\spt[um = {\tau}, dm = {\tatesphere}]{S, \mc{V}}}}]{A \otimes \Sigma^{\infty}_{\tatesphere, \mc{V}} \yona[um = {\tau}]{S}{X}}{\twist{A}{r}}
\]
for each $X \in \sm[ft]{S}$ and each $r \in \integer$.
\end{enumerate}
\end{prop}

\begin{proof}
\nref[Claim]{prop:enriched-representability.1} follows from \nref{rep.3}, \nref{prop:enrichments}, and \nref{rep.6a}.

Consider \nref[Claim]{prop:enriched-representability.2}.
We have natural equivalences
\begin{align*}
&\mor[um = {\mc{W}}, dm = {\mod[dm = {A}]{\spt[um = {\tau}, dm = {\tatesphere}]{S, \mc{V}}}}]{A \otimes \Sigma^{\infty}_{\tatesphere, \mc{V}} \yona[um = {\tau}]{S, \mc{V}}{X}}{\twist{A}{r}} \\
&\qquad \simeq \tilde{\rho}^*\intmor[dm = {\mod[dm = {A}]{\spt[um = {\tau}, dm = {\tatesphere}]{S, \mc{V}}}}]{A \otimes \Sigma^{\infty}_{\tatesphere, \mc{V}} \yona[um = {\tau}]{S, \mc{V}}{X}}{\twist{A}{r}} 
&& \text{\nref{prop:enrichments}}\\
&\qquad \simeq \fct{\tilde{\rho}^*}{\prns{A \otimes \Sigma^{\infty}_{\tatesphere, \mc{V}} \yona[um = {\tau}]{S, \mc{V}}{X}}^{\vee} \otimes_A \twist{A}{r}} 
&& \text{\nref{SH.11.2}}\\
&\qquad \simeq \fct{\tilde{\rho}^*}{\prns{A \otimes \Sigma^{\infty}_{\tatesphere, \mc{V}} \yona[um = {\tau}]{S, \mc{V}}{X}}^{\vee}} \otimes \tilde{\rho}^*\twist{A}{r}
&& \text{\nref{rep.3}, \nref{rep.6a}} \\
&\qquad \simeq \twist{\cohomology{X}}{r}_{\cohomology{}}
&& \text{\nref{rep.3.a}}
\end{align*}
as required.
\end{proof}

\begin{ex}
\nlabel{ex:enriched-representability}
\nref{prop:enriched-representability} implies that the objects $\hodgespectrum$ and $\varhodgespectrum$ both represent the mixed Weil theory $\enhancedbetti{}$ in the $\D{\ind{\mhsp[\rational]}}^{\otimes}$-enriched sense, and that $\DH{\spec{\complex}}$ and $\varDH{\spec{\complex}}$ admit $\D{\ind{\mhsp[\rational]}}^{\otimes}$-enriched-\qcategory\ structures from the associated Hodge realization functors.
\end{ex}

\begin{rmk}
\nlabel{rep.11}
The category $\D{\ind{\mhsp[\rational]}}$ admits a natural t-structure.
It also admits a weight structure characterized by the property that each polarizable pure $\rational$-Hodge structure $V$ of weight $k$, regarded as a complex concentrated in degree $0$, is pure of weight $k$ with regard to the weight structure.
Moreover, this weight structure is compatible with the t-structure, i.e., the t-structure and the weight structure are mutually transversal in the sense that the intersection of the hearts of the t-structure and the weight structure forms a semisimple Abelian category.

Combining these remarks regarding the t-structure and weight structure on $\D{\ind{\mhsp[\rational]}}$ with the results of this section, one can check that each of the desiderata from \varnref{desiderata} except for \varnref{desideratum.2} is satisfied when we take $X = \spec{\complex}$ and $\D{-} = \varDH{-}$.
With a little more effort, one can establish the analogous statement for $X = \spec{\complex}$ and $\D{-} = \DH{-}$.
The sequel is devoted to establishing \nref[Desiderata]{desideratum.1} through \varnref{desideratum.6} in full generality by establishing a suitable six-functor formalism for $\DH{\spec{\complex}}$ and $\varDH{\spec{\complex}}$.
We defer discussion of the t-structures required in \nref[Desiderata]{desideratum.7} and \varnref{desideratum.8} over general bases to a forthcoming sequel.
\end{rmk}

\section{Coefficient systems}
\nlabel{funct}

\setcounter{thm}{-1}

\begin{notation}
\nlabel{functors.0}
In this section, we fix the following notation a Noetherian scheme $S$ of finite dimension.
\end{notation}

\begin{motivation}
The \qcategories\ $\spt[um = {\tau}, dm = {\tatesphere}]{X}$ for $X \in \sch[ft]{S}$ are related by a formalism of Grothendieck's six functors $f^*$, $f_*$, $f_!$, $f^!$, $\otimes$ and $\intmor{}{}$.
An abstract framework for establishing such a formalism is established in \cite{Ayoub_six-operationsI} and \cite{Cisinski-Deglise_triangulated-categories} in the languages of derivators and model categories, respectively.
The purpose of this section is to provide a \qcategorical\ axiomatization of this framework, leading to the notion of a \qcategorically-enhanced \coeffsyst.
Similar axiomatizations have been proposed elsewhere, but with implicit reference to results in the theory of $(\infty,2)$-categories for which the author of the present text was unable to locate references.
In any case, we make no claim to originality here.

Of course, one would also like to be able to apply the results of \cite{Ayoub_six-operationsI, Ayoub_six-operationsII, Cisinski-Deglise_triangulated-categories} in this \qcategorical\ context.    
The other goal of this section is thus to show that each \qcategorically\ enhanced \coeffsyst\ admits an underlying symmetric monoidal stable homotopy $2$-functor (\cite[1.4.1, 2.3.1]{Ayoub_six-operationsI}) or motivic category (\cite[2.4.45]{Cisinski-Deglise_triangulated-categories}), given by passing to homotopy categories.
\end{motivation}

\begin{summary}\
\begin{itemize}
\item
In \nref{functors.3}, we introduce the axiomatic notion of a \qcategorical\ \coeffsyst which by definition encodes the functors $f^*$, $f_*$, $p_{\sharp}$ for $p$ smooth, $\otimes$, and $\intmor{}{}$.
In \nref{functors.a6}, we define a \qcategory\ of such \coeffsysts.
\item
In \nref{functors.a8}, we show that the homotopy-category functor induces a functor from the \qcategory\ of \coeffsysts\ to the $(2,1)$-category of symmetric monoidal stable homotopy $2$-functors.
\item
In \nref{functors.a9}, we introduce the fundamental example of a \coeffsyst:
$\spt[um = {\tau}, dm = {\tatesphere}]{-}^{\wedge}$.
\item
In \nref{functors.a11} and \nref{functors.a13}, we apply results of \cite{Ayoub_six-operationsI, Liu-Zheng_gluing-restricted, Robalo_K-theory-and-the-bridge} to obtain a theory of \qcategorically\ enhanced exceptional functors $f_!$ and $f^!$ for \locpres\ \coeffsysts, completing the six-functor formalism.
\end{itemize}
\end{summary}

\begin{defn}
\nlabel{functors.a1}
The following notions, which we draw from \cite[4.7.4.13]{Lurie_higher-algebra}, provide a convenient language for discussion base change natural transformations and the like. 
\begin{enumerate}
\item
\nlabel{functors.a1.1}
Consider a square $Q$ of \qcategories
\[
\begin{tikzcd}
\mc{C}
\ar[r, "f^*" above]
\ar[d, "g^*" left]
\ar[dr, "Q" description, phantom]
&
\mc{D}
\ar[d, "g'^*" right]
\\
\mc{C}'
\ar[r, "f'^*" below]
&
\mc{D}'
\end{tikzcd}
\]
that commutes up to a given equivalence $\alpha: g'^*f^* \isom f'^*g^*$.
We say that $Q$ is \emph{left adjointable} provided that left adjoints $f_{\sharp} \dashv f^*$ and $f'_{\sharp} \dashv f'^*$ exist and that the composite \emph{exchange morphism}
\[
\fct{\on{ex}^*_{\sharp}}{\alpha}:
f'_{\sharp}g'^* 
\to
f'_{\sharp} g'^* f^* f_{\sharp} 
\xrightarrow{\alpha}
f'_{\sharp} f'^* g^* f_{\sharp}
\to
g^* f_{\sharp}
\]
is an equivalence, where the first and third arrows are induced by the unit and counit of their respective adjunctions.
\item
\nlabel{functors.a1.2}
Dually, we say $Q$ is \emph{right adjointable} if right adjoints $f^* \dashv f_*$ and $f'^* \dashv f'_*$ exist and the composite exchange morphism
\[
\fct{\on{ex}^*_*}{\alpha}:
g^* f_*
\to
f'_* f'^* g^* f_*
\xrightarrow{\alpha^{-1}}
f'_* g'^* f^* f_*
\to
f'_* g'^*
\]
is an equivalence.
\item
\nlabel{functors.a1.3}
Defining the \emph{transpose} $Q^{\mr{tr}}$ of $Q$ to be the square obtained by reflecting $Q$ across the diagonal, if $f^*$ and $f'^*$ admit left adjoints and $g^*$ and $g'^*$ admit right adjoints, then $Q$ is left adjointable if and only if $Q^{\mr{tr}}$ is right adjointable: left adjoints commute if and only if their right adjoints commute.
\end{enumerate}
\end{defn}

\begin{notation}
\nlabel{functors.a2}
The central objects of discussion in this section are functors 
\[
\coef[um = {*,\otimes}]{M}{}:
  \prns{\sch[ft]{S}}\op \to \calg{\QCATexmon},
\]
satisfying a list of conditions.
Such functors consist of the following homotopy-coherent data:
for each $X \in \sch[ft]{S}$, a possibly large symmetric monoidal \qcategory\  $\coef[um = {\otimes}]{M}{X} \coloneqq \coef[um = {\otimes}]{M}{X}$; and
for each morphism $f: X \to Y$ of $\sch[ft]{S}$, a symmetric monoidal pullback functor
\[
f^{*,\otimes} 
  \coloneqq \coef[um = {*, \otimes}]{M}{f}:
    \coef[um = {\otimes}]{M}{Y} 
      \to \coef[um = {\otimes}]{M}{X}.
\]
\end{notation}

\begin{defn}
[{\cite[1.4.1, 2.3.1]{Ayoub_six-operationsI}, \cite[2.4.45]{Cisinski-Deglise_triangulated-categories}}]
\nlabel{functors.3}
A \emph{\coeffsyst\ over $S$} is a functor 
$
\coef[um = {*, \otimes}]{M}{}:
  \prns{\sch[ft]{S}}\op
    \to \calg{\QCATexmon}
$
satisfying the following seven conditions.
\begin{enumerate}
\item
\nlabel{functors.3.A}
\textbf{Pushforwards:} 
For each morphism $f: X \to Y$ of $\sch[ft]{S}$, the functor $f^*$ admits a right adjoint $f_*:\coef{M}{X} \to \coef{M}{Y}$. 
\item
\nlabel{functors.3.B}
\textbf{Internal morphisms objects:}
For each $X \in \sch[ft]{S}$, the symmetric monoidal \qcategory\ $\coef[um = {\otimes}]{M}{X}$ is closed, i.e., for each $M \in \coef{M}{X}$, the functor $\prns{-}\otimes_{\coef{M}{X}} M$ admits a right adjoint, denoted by $\intmor[dm = {\coef{M}{X}}]{M}{-}$.
\item
\nlabel{functors.3.C}
\textbf{Smooth base change:}
The restriction of the functor $\coef[um = {*}]{M}{}$ to the subcategory $\prns{\sch[sm]{S}}\op$, whose morphisms are the opposites of the smooth morphisms in $\sch[ft]{S}$, factors through the inclusion of the \subqcategory\ $\QCAT[u = {Ex,L,R}] \hookrightarrow \QCAT[u = {Ex,L}]$ spanned by the functors admitting both left and right adjoints.
For each Cartesian square
\[
\begin{tikzcd}
X\times_Y Y'
\ar[r, "g_{\prns{X}}" above]
\ar[d, "f_{\prns{Y'}}" left]
\ar[dr, "Q" description, phantom] 
&
X
\ar[d, "f" right]
\\
Y'
\ar[r, "g" below]
&
Y
\end{tikzcd}
\]
in $\sch[ft]{S}$ with $f$ smooth, the induced square 
\[
\begin{tikzcd}
\coef{M}{Y}
\ar[r, "f^*" above]
\ar[d, "g^*" left]
\ar[dr, "\coef{M}{Q}" description, phantom]
&
\coef{M}{X}
\ar[d, "g^*_{\prns{X}}" right]
\\
\coef{M}{Y'}
\ar[r, "f^*_{\prns{Y'}}" below]
&
\coef{M}{X\times_YY'}
\end{tikzcd}
\]
is left adjointable.
\item
\nlabel{functors.3.D}
\textbf{Smooth projection formula:}
For each smooth morphism $p: X \to Y$ of $\sch[ft]{S}$, the exchange transformation 
\[
p_{\sharp}\prns{- \otimes_{\coef{M}{X}} \fct{p^*}{-}} 
  \to p_{\sharp}\prns{-} \otimes_{\coef{M}{Y}} \prns{-}:
    \coef{M}{X} \times \coef{M}{Y} \to \coef{M}{Y}
\]
is an equivalence, where $p_{\sharp}$ is a left adjoint of $p^*$, the existence of which is guaranteed by \nref[Axiom]{functors.3.C}.
\item
\nlabel{functors.3.E}
\textbf{Localization:}
For each closed immersion $i: Z \hookrightarrow X$ in $\sch[ft]{S}$ with complementary open immersion $j: U \hookrightarrow X$, the square
\[
\begin{tikzcd}
\coef{M}{Z}
\ar[r, "i_*" below]
\ar[d]
&
\coef{M}{X}
\ar[d, "j^*"]
\\
\pt
\ar[r]
&
\coef{M}{U}
\end{tikzcd}
\]
is Cartesian in $\QCATex$.
\item
\nlabel{functors.3.G}
\nlabel{functors.a3.F}
\textbf{$\affine^1$-homotopy invariance:}
For each $X \in \sch[ft]{S}$, if $p:\affine^1_X \to X$ denotes the canonical projection, then $p^*: \coef{M}{X} \to \coef{M}{\affine^1_X}$ is fully faithful.
\item
\nlabel{functors.3.H}
\nlabel{functors.a3.G}
\textbf{$\tatesphere$-stability:}
For each smooth morphism of finite type $f: X \to Y$ in $\sch[ft]{S}$ admitting a section $s: Y \to X$, the \emph{Thom transformation} $\Thom{f,s} \coloneqq f_{\sharp}s_*$ is an equivalence, where $f_{\sharp}$ is a left adjoint of $f^*$, the existence of which is guaranteed by \nref[Axiom]{functors.3.C}.
\end{enumerate}
\end{defn}

\begin{rmk}
\nlabel{rep.a3}
Let $\kappa$ be a small regular cardinal.
Recall that $\pr[u = {L}, dm = {\kappa, \tu{st}}]$ is the very large \qcategory\ of stable \locpres[\kappa]\ \qcategories\ and cocontinuous functors that preserve $\kappa$-presentable objects.
It carries a symmetric monoidal structure $\pr[um = {\tu{L}, \otimes}, dm = {\kappa, \tu{st}}]$.

The objects of the \qcategory\ $\calg{\pr[um = {\tu{L}, \otimes}, dm = {\kappa, \tu{st}}]}$ can be regarded as stable \locpres[\kappa]\ symmetric monoidal \qcategories\ in which the monoidal unit is $\kappa$-presentable unit, and in which the tensor product is cocontinuous separately in each variable and preserves $\kappa$-presentable objects. 
The morphisms of $\calg{\pr[um = {\tu{L}, \otimes}, dm = {\kappa, \tu{st}}]}$ are the cocontinuous symmetric monoidal functors that preserve $\kappa$-presentable objects.
\end{rmk}

\begin{defn}
\nlabel{functors.4}
Let $\coef[um = {*, \otimes}]{M}{}: \prns{\sch[ft]{S}}\op \to \QCAT[u = {Ex}]$ be a functor, and $\kappa$ a small regular cardinal.
We say that $\coef[um = {*, \otimes}]{M}{}$ is:
\begin{enumerate}
\item
\nlabel{functors.4.1}
\emph{essentially small} if it factors through the functor $\calg{\qcatexmon} \to \calg{\QCATexmon}$;
\item
\nlabel{functors.4.2}
\emph{\locpres} if it factors through the functor $\calg{\pr^{\mr{L},\otimes}_{\mr{st}}} \to \calg{\QCATexmon}$; and
\item
\nlabel{functors.4.3}
\emph{\locpres[\kappa]} if it factors through the functor $\calg{\pr^{\mr{L},\otimes}_{\kappa,\mr{st}}} \to \calg{\QCAT^{\mr{Ex},\otimes}}$.
\end{enumerate}
\end{defn}

\begin{defn}
\nlabel{functors.a5}
Let $\mc{V}^{\otimes}$ be a \locpres\ symmetric monoidal \qcategory.
A \emph{$\mc{V}^{\otimes}$-linear \locpres\ \coeffsyst} is a functor
\[
\coef[um = {*, \otimes}]{M}{}:
  \prns{\sch[ft]{S}}\op \to \calg{\pr[um = {\tu{L}, \otimes}, d = st]}_{\mc{V}^{\otimes}/}.
\]
Similarly, if $\kappa$ is a small regular cardinal, then a \emph{$\mc{V}^{\otimes}$-linear \locpres[\kappa]\ \coeffsyst} is a functor $\prns{\sch[ft]{S}}\op \to \calg{\pr[um = {\tu{L}, \otimes}, dm = {\kappa, \tu{st}}]}_{\mc{V}^{\otimes}/}$.
\end{defn}

\begin{rmk}
This is a generalization of the notion of a \locpres\ \coeffsyst\ in the sense that the forgetful functors
\[
\calg{\pr[um = {\tu{L}, \otimes}, d = {st}]}_{\spc{}^{\times}/} 
  \to \calg{\pr[um = {\tu{L}, \otimes}, d = {st}]}
  \leftarrow \calg{\pr[um = {\tu{L}, \otimes}, d = {st}]}_{\spt{}^{\wedge}/}
\]
are equivalences (\nref{scalars.a3.2}), so a \locpres\ \coeffsyst\ is the same as a $\spc{}^{\times}$-linear or $\spt{}^{\wedge}$-linear \locpres\ \coeffsyst.
It is certainly possible to generalize further to the setting of $\mc{V}^{\otimes}$-linear \coeffsysts\ without any \locpresbility\ hypotheses, but we will not need the added generality below.
\end{rmk}

\begin{defn}
\nlabel{functors.a6}
Let $\phi^*: \coef[um = {*, \otimes}]{M}{} \to \coef[um = {*, \otimes}]{N}{}$ be a morphism of $\psh{\sch[ft]{S}}{\calg{\qcatexmon}}$ such that $\coef[um = {*, \otimes}]{M}{}$ and $\coef[um = {*, \otimes}]{N}{}$ are \coeffsysts.
Then $\phi^*$ is a \emph{morphism of \coeffsysts\ over $S$} if, for each smooth morphism $p: X \to Y$ of $\sch[ft]{S}$, the square
\[
\begin{tikzcd}
\coef{M}{Y}
\ar[r, "p^*" below]
\ar[d, "\phi^*_Y" left]
&
\coef{M}{X}
\ar[d, "\phi^*_X" right]
\\
\coef{N}{Y}
\ar[r, "p^*" above]
&
\coef{N}{X}
\end{tikzcd}
\] 
is left adjointable (\nref{functors.a1}), i.e., the exchange transformation $\coef[dm = {\sharp}]{N}{p} \phi^*_X \to \phi^*_Y \coef[dm = {\sharp}]{M}{p}$ is an equivalence.
The notion of a morphism of $\mc{V}^{\otimes}$-linear \locpres\ \coeffsysts\ is similar.
Such morphisms are stable under composition in $\ho{\psh{\sch[ft]{S}}{\calg{\qcatexmon}}}$, and the \emph{\qcategory\ of \coeffsysts\ over $S$} is the \subqcategory\ of $\psh{\sch[ft]{S}}{\calg{\QCATexmon}}$ spanned by the \coeffsysts\ and the morphisms of such.
The \qcategory\ of essentially small \coeffsysts\ over $S$ is defined analogously.

We also define a \emph{morphism of $\mc{V}^{\otimes}$-linear \locpres\ \coeffsysts\ over $S$} as a morphism of $\psh{\sch[ft]{S}}{\calg{\pr[um = {\tu{L}, \otimes}, d = {st}]}_{\mc{V}^{\otimes}/}}$ whose underlying morphism in $\psh{\sch[ft]{S}}{\calg{\QCATexmon}}$ is a morphism of \coeffsysts.
\end{defn}

\begin{rmk}
\nlabel{functors.a7}
Let $\coef[um = {*, \otimes}]{M}{}: \prns{\sch[ft]{S}}\op \to \calg{\QCATexmon}$ be a functor.
We elaborate on the preceding definitions and introduce some notation.
\begin{enumerate}
\item
\nlabel{functors.a7.A}
\textbf{Pushforwards and internal morphisms objects revisited:}
If $\coef[um = {*, \otimes}]{M}{}$ is \locpres, then it satisfies \nref[Axiom]{functors.3.A} guaranteeing the existence of pushforwards.
Indeed, using the equivalence $\pr^{\mr{L}}_{\mr{st}} \simeq \prns{\pr^{\mr{R}}_{\mr{st}}}\op$ deduced from \cite[5.5.3.4]{Lurie_higher-topos}, we obtain a functor $\coef[dm = {*}]{M}{}$ given by the composite
\[
\sch[ft]{S} 
  \xrightarrow{\prns{\coef[um = {*}]{M}{}}\op} \prns{\pr[u = {L}, d = {st}]}\op 
    \simeq \pr[u = {R}, d = {st}]
\]
given informally by assigning to each morphism $f: X \to Y$ in $\sch[ft]{S}$ a functor $\coef[dm = {*}]{M}{f}: \coef{M}{X} \to \coef{M}{Y}$ right adjoint to $f^*$.

\begin{notation*}
For each morphism $f: X \to Y$ in $\sch[ft]{S}$, we set $f_* \coloneqq \coef[dm = {*}]{M}{f}$, so that $f^* \dashv f_*$.
\end{notation*}

\nref[Axiom]{functors.3.B}, guaranteeing the existence of internal morphisms objects, also follows from \locpresbility.
For each $X \in \sch[ft]{S}$ and each $M \in \coef{M}{X}$, the endofunctor $\prns{-} \otimes_{\coef{M}{X}} M: \coef{M}{X} \to \coef{M}{X}$ is cocontinuous provided that $\coef{M}{X}^{\otimes}$ is a \locpres\ symmetric monoidal \qcategory.
The Adjoint Functor Theorem (\cite[5.5.2.9]{Lurie_higher-topos}) then provides the desired right adjoint.
\item
\nlabel{functors.a7.B}
\textbf{Smooth base change revisited:}
\nref[Axiom]{functors.3.C} implies in particular that the restriction of $\coef[um = {*}]{M}{}$ to the subcategory $\prns{\sch[sm]{S}}\op \subseteq \prns{\sch[ft]{S}}\op$, whose morphisms are the smooth morphisms, factors through the inclusion $\pr[u = {R}, d = {st}] \hookrightarrow \QCATex$ of the \subqcategory\ spanned by the stable \locpres\ \qcategories\ and the right-adjoint functors.
Proceeding as in \nref{functors.a7.A}, we obtain a functor $\coef[dm = {\sharp}]{M}{}: \sch[sm]{S} \to \pr[u = {L}, d = {st}]$ given informally by assigning to each smooth morphism $f: X \to Y$ in $\sch[ft]{S}$ a functor $\coef[dm = {\sharp}]{M}{f}: \coef{M}{X} \to \coef{M}{Y}$ left adjoint to $f^*$.

\begin{notation*}
For each smooth morphism $f: X \to Y$ of $\sch[ft]{S}$, we set $f_{\sharp} \coloneqq \coef[dm = {\sharp}]{M}{f}: \coef{M}{X} \to \coef{M}{Y}$, so that $f_{\sharp} \dashv f^*$.
\end{notation*}
\item 
\nlabel{functors.a7.C}
\textbf{Localization revisited:}
\nref[Axiom]{functors.3.E} implies in particular that $\coef{M}{\varnothing} \simeq \pt$.
Indeed, this follows from the case of the closed immersion $i: \varnothing \hookrightarrow X$ for any $X \in \sch[ft]{S}$.
Each complementary open immersion $j: U \hookrightarrow X$ is an isomorphism.
Since $\coef[um = {*}]{M}{}$ is a functor, $j^*$ is an equivalence.
The projection $\coef{M}{\varnothing} \to \pt$ is the pullback of an equivalence, hence an equivalence.

Localization also has the following consequences for each closed immersion $i: Z \
\hookrightarrow X$ and each complementary open immersion $j: U \hookrightarrow X$:
\begin{enumerate}
\item
$i_*$ is fully faithful;
\item
the pair $\prns{i^*, j^*}$ is conservative; and
\item
for each $M \in \coef{M}{X}$, the counit and unit morphisms form a fiber sequence
\[
j_{\sharp}j^*M \to M \to i_*i^*M
\]
of $\coef{M}{X}$ natural in $M$.
\end{enumerate}
In fact, the first two of these conditions are equivalent to the localization axiom by \cite[Proposition~9.4.20]{Robalo_thesis}.
\end{enumerate}
\end{rmk}

\begin{rmk}
\nlabel{rmk:ho}
We have a Quillen adjunction $\ho{}: \sset \rightleftarrows \cat: \nerve{}$ between the Joyal model structure on $\sset$ and the \emph{canonical model structure} on $\cat$, whose weak equivalences $\mf{W}_{\tu{eq}}$ are the equivalences of categories, and whose fibrations are the isofibrations.
We obtain an adjunction of the underlying \qcategories\ $\ho{}: \qcat \rightleftarrows \cat\brk{\mf{W}_{\tu{eq}}^{-1}}: \nerve{}$.
Both $\ho{}$ and $\nerve{}$ preserve finite product, i.e., they underlie symmetric monoidal functors with respect to the Cartesian symmetric monoidal structures (\cite[2.4.1.1]{Lurie_higher-algebra}).
This leads to an adjunction
\begin{equation}
\nlabel{functors.a8.1}
\ho{}_{\calg{}}:
  \calg{\qcatmon} 
    \rightleftarrows \calg{\cat\brk{\mf{W}^{-1}_{\tu{eq}}}^{\times}}:
      \nerve{}_{\calg{}}.
\end{equation}

If we equip the category $\gpd$ of small groupoids with the (symmetric monoidal) model structure it inherits from $\cat$, and if we regard the $(2,1)$-category $\cattwo$ of small categories, functors and natural isomorphisms as a $\gpd^{\times}$-enriched category, then $\cattwo$ is a $\gpd^{\times}$-enriched model category when we equip the underlying category $\cat$ with canonical model structure.

It follows now from \cite[4.8]{Dwyer-Kan_function-complexes} that the hammock localization $\LH{\cat, \mf{W}_{\tu{eq}}}$ is Dwyer-Kan equivalent to the fibrant $\sset$-enriched category $\nerve{}_*\cattwo$ whose objects are those of $\cat$, and in which the mapping space $\map[dm = {\nerve{}_*\cattwo}]{\mc{C}}{\mc{D}}$ is the nerve of the groupoid of functors $\mc{C} \to \mc{D}$ and natural isomorphisms of such.
Taking simplicial nerves, we obtain an equivalence of \qcategories\ $\cat\brk{\mf{W}^{-1}_{\tu{eq}}} \simeq \snerve{\nerve{}_*\cattwo}$.
Furthermore, the simplicial set $\snerve{\nerve{}_*\cattwo}$ is isomorphic to the geometric nerve $\gmnerve{\cattwo}$ of the strict $2$-category $\cattwo$.
By these observations, we may identify \nref{functors.a8.1} with an adjunction $\calg{\qcatmon} \rightleftarrows \calg{\gmnerve{\cattwo}^{\times}}$.
The left adjoint assigns to a $0$-simplex of $\calg{\qcatmon}$ classifying the symmetric monoidal \qcategory\ $\mc{C}^{\otimes}$ a $0$-simplex of $\calg{\gmnerve{\cattwo}^{\times}}$ classifying the symmetric monoidal structure $\ho{\mc{C}}^{\otimes}$ on its homotopy category.
Similarly, the left adjoint sends the symmetric monoidal functor $F^{\otimes}: \mc{C}^{\otimes} \to \mc{D}^{\otimes}$ to the induced symmetric monoidal functor $\ho{F}^{\otimes}: \ho{\mc{C}}^{\otimes} \to \ho{\mc{D}}^{\otimes}$.

The composite induces, for each \qcategory\ $\mc{C}$, a functor
\[
\ho{}:
\psh{\mc{C}}{\calg{\qcatmon}}
  \to \psh{\mc{C}}{\calg{\gmnerve{\cattwo}^{\times}}}
\]
given informally by assigning to each diagram of symmetric monoidal \qcategories\ the induced diagram of symmetric monoidal homotopy categories.
\end{rmk}

\begin{prop}
\nlabel{functors.a8}
The functor $\ho{}$ of the previous remark restricts to a functor from the \qcategory\ of essentially small \coeffsysts\ over $S$ to the geometric nerve of the $(2,1)$-category of symmetric monoidal stable homotopy $2$-functors \textup{(\cite[1.4.1, 2.3.1]{Ayoub_six-operationsI} and \cite[3.1]{Ayoub_operations-de-Grothendieck})} given informally by assigning to $\coef[um = {*, \otimes}]{M}{}: \prns{\sch[ft]{S}}\op \to \calg{\QCATexmon}$ the pseudofunctor $X \mapsto \hocoef{M}{X}^{\otimes}$.
\end{prop}

\begin{proof}
By \nref{rmk:ho}, we have functors
\[
\calg{\qcatexmon}
  \to \calg{\qcatmon}
    \xrightarrow{\ho{}_{\calg{}}}
      \calg{\gmnerve{\cattwo}^{\times}},
\]
where the first arrow is associated with the inclusion $\qcatex \hookrightarrow \qcat$.
 One can check that the composite factors through the forgetful functor $\gmnerve{\ttcat} \to \calg{\gmnerve{\cattwo}^{\times}}$ from the geometric nerve of the $(2,1)$-category of small tensor-triangulated categories.
A fairly routine verification now shows that the resulting functor
\[
\psh{\sch[ft]{S}}{\calg{\qcatexmon}}
  \to \psh{\sch[ft]{S}}{\gmnerve{\ttcat}}
\]
sends (morphisms of) \coeffsysts\ to (morphisms of) symmetric monoidal stable homotopy $2$-functors, as required.
\end{proof}

\begin{rmk}
We have been imprecise here: 
we deal exclusively with unbiased symmetric monoidal structures, whereas \cite{Ayoub_six-operationsI} works with biased ones.
We are not aware of a proof in the literature of the equivalence between the two definitions.
We will address this technical issue elsewhere, but the abuse is not serious.
\end{rmk}

\begin{cor}
\nlabel{cor:homotopy-2-functor}
Let $\phi^{*, \otimes}: \coef[um = {*, \otimes}]{M}{} \to \coef[um = {*, \otimes}]{N}{}$ be a morphism of $\psh{\sm[ft]{S}}{\calg{\pr[um = {\tu{L}, \otimes}, d = {st}]}}$.
\begin{enumerate}
\item
Suppose that, for each smooth morphism $f: X \to Y$ of $\sch[ft]{S}$, $f^*: \coef{M}{Y} \to \coef{M}{X}$ admits a left adjoint.
Then $\coef[um = {*, \otimes}]{M}{}$ is a \coeffsyst\ if and only if $\hocoef[um = {*, \otimes}]{M}{}$ is a stable homotopy $2$-functor.
\item
If $\coef[um = {*, \otimes}]{M}{}$ and $\coef[um = {*, \otimes}]{N}{}$ are \coeffsysts, then $\phi^{*, \otimes}$ is a morphism of such if and only if $\hocoef[um = {*, \otimes}]{\phi}{}$ is a morphism of stable homotopy $2$-functors.
\end{enumerate}
\end{cor}

\begin{proof}
Aside from the existence of the adjoints $f_{\sharp}$, $f_*$ and $\intmor{}{}$, the axioms of \nref{functors.3} can be checked at the level of homotopy categories.
Here, we appeal to the reformulation of the localization axiom (\nref[Axiom]{functors.3.E}) given in \nref{functors.a7.C}.
The adjointability condition in the definition of a morphism of \coeffsysts\ can also be checked at the level of homotopy categories.
\end{proof}

\begin{rmk}
While \nref{cor:homotopy-2-functor} suggests that much of the essential data of a \locpres\ \coeffsyst\ is already encoded at the level of homotopy categories, the advantage of working with the \qcategorically\ enhanced version is that it allows us to employ the techniques of higher algebra.
In particular, these enhancements will be essential for the extension of the \qcategory\ $\DH{\spec{\complex}}$ of \nref{defn:hodge-modules} to a six-functor formalism.
\end{rmk}

\begin{rmk}[$\tatesphere$-stability revisited]
\nlabel{functors.a9}
The proofs of \cite{Ayoub_six-operationsI, Ayoub_six-operationsII} apply in the unbiased setting without essential modification.
Thus, by \nref{functors.a8}, the results of \cite{Ayoub_six-operationsI, Ayoub_six-operationsII} and \cite{Cisinski-Deglise_triangulated-categories} expressed in the language of triangulated categories apply to \coeffsysts. 
With this in mind, we now revisit \nref[Axiom]{functors.3.H} regarding $\tatesphere$-stability.

Suppose that $\coef[um = {*,\otimes}]{M}{}$ satisfies the first six conditions of \nref{functors.3}.
\begin{enumerate}
\item
\nlabel{functors.a9.A}
By \cite[1.5.7]{Ayoub_six-operationsI} or \cite[2.4.14]{Cisinski-Deglise_triangulated-categories}, \nref[Axiom]{functors.3.H} follows from \nref[Axiom]{functors.3.E}, Zariski excision, which will be established below in \nref{excision.2}, and the following ostensibly weaker condition:
for each $X \in \sch[ft]{S}$, if $s: X \to \affine^1_X$ denotes the zero section of the canonical projection $p: \affine^1_X \to X$, then $\Thom{p,s} = p_{\sharp}s_*$ is an equivalence.
\item
\nlabel{functors.a9.B}
According to \cite[2.4.13, 2.3.8]{Cisinski-Deglise_triangulated-categories}, for each $f$ and $s$ as in \nref[Axiom]{functors.3.H}, there is an equivalence $\Thom{f,s} \simeq \Thom{f,s}\1{\coef{M}{Y}} \otimes_{\coef{M}{Y}} \prns{-}$.
Consequently, $\coef[um = {*, \otimes}]{M}{}$ satisfies \nref[Axiom]{functors.3.H} in this situation if and only if $\Thom{f,s}\1{\coef{M}{Y}}$ is $\otimes$-invertible.
In fact, it suffices that $\Thom{f,s}\1{\coef{M}{Y}}$ be $\otimes$-invertible when $f$ ranges over the canonical projections $\affine^1_Y \to Y$ and $s$ ranges over the corresponding zero sections by \nref{functors.a9.A}.
\item
\nlabel{functors.a9.C}
When $f: X = \spec{\sym{\mc{E}^{\vee}}} \to Y$ is a vector bundle corresponding to the locally free $\0_Y$-module $\mc{E}$ of finite rank and $s: Y \hookrightarrow X$ the zero section, we abusively denote the Thom transformation $\Thom{f,s}$ of \nref[Axiom]{functors.3.H} by $\Thom{\mc{E}}$.

\begin{notation*}
When $f$ is the trivial bundle associated with $\0_Y$, we denote the endofunctor $M \mapsto \Thom{\0_Y}M\sus{-2}$ by $M \mapsto \twist{M}{1}$,
and we denote its quasi-inverse $M \mapsto \operator[um = {-1}]{{Th}}{\0_Y}M\sus{2}$ by $M \mapsto \twist{M}{-1}$.
When we apply these functors to the unit object, we obtain the \emph{Tate object $\tate{\coef{M}{Y}}{1}$ of $\coef{M}{Y}$} and its \emph{inverse} $\tate{\coef{M}{Y}}{-1}$, respectively.
\end{notation*}

\noindent
It follows from \nref{functors.a9.B} above that $\tate{\coef{M}{Y}}{1}$ and $\tate{\coef{M}{Y}}{-1}$ are mutual $\otimes$-inverses when $\coef[um = {*, \otimes}]{M}{}$ is a \coeffsyst. 
\item
\nlabel{functors.a9.D}
As explained in \cite[2.4.19]{Cisinski-Deglise_triangulated-categories}, the Tate object $\tate{\coef{M}{X}}{1}$ is equivalent to $K\sus{-2}$, where $K$ denotes the fiber of the counit morphism $p_{\sharp}p^*\1{\coef{M}{X}} \to \1{\coef{M}{X}}$ associated with the canonical projection $p: \projective^1_X \to X$.
\end{enumerate}
\end{rmk}

\begin{ex}
\nlabel{functors.a10}
Our first examples of \coeffsysts\ are \qcategorical\ constructions of the $\tatesphere$-stable motivic homotopy categories.
\begin{enumerate}
\item
\nlabel{functors.a10.1}
Let $\spt[u = {Nis}, dm = {\tatesphere}]{-}^{*, \wedge}: \prns{\sch[ft]{S}}\op \to \calg{\pr^{\mr{L},\otimes}_{\mr{st}}}$ denote the functor denoted by $\SH{}^{\otimes}$ in \cite[\S9.1]{Robalo_thesis}, sending each $S$-scheme $X$ to the symmetric monoidal \qcategory\ ${\spt[u = {Nis}, dm = {\tatesphere}]{X}}^{\wedge}$ of \motivicspectra\ as defined in \nref[Definition]{SH.1.J}.
By \cite[Theorem~9.4.36]{Robalo_thesis}, $\spt[u = {Nis}, dm = {\tatesphere}]{-}^{*, \wedge}$ is a \locpres[\aleph_0]\ \coeffsyst.
Alternatively, one can deduce this from \nref{sheaves.12}, \nref{SH.a8}, \nref{functors.a8}, and \cite[Th\'eor\`eme~4.5.30]{Ayoub_six-operationsII}.
\item
\nlabel{functors.a10.2}
Similar arguments show that the \'etale-local variant $\spt[u = {\'et}, dm = {\tatesphere}]{-}^{*, \wedge}$ is also a \locpres\ \coeffsyst.
If $\mc{V}$ is $\prns{\sm[ft]{S}, \tu{\'et}}$-finite, then it is even \locpres[\aleph_0].
\end{enumerate}
\end{ex}

\begin{prop}[Proper exceptional pullbacks]
\nlabel{functors.a11}
Let $\coef[um = {*, \otimes}]{M}{}: \prns{\sch[ft]{S}}\op \to \calg{\pr[um = {\tu{L}, \otimes}, d = {st}]}$ be a functor satisfying either of the following conditions:
\begin{enumerate}[label=$\tu{(\alph*)}$]
\item 
\nlabel{functors.a11.1}
$\coef[um = {*}]{M}{}$ factors through the inclusion $\pr[u = {L}, dm = {\aleph_0, \tu{st}}] \hookrightarrow \pr[u = {L}, d = {st}]$; or
\item
\nlabel{functors.a11.2}
$\coef[um = {*, \otimes}]{M}{}$ is a \locpres\ \coeffsyst.
\end{enumerate}
For each proper morphism $p: X \to Y$ in $\sch[ft]{S}$, there is a sequence of adjoint functors $p^* \dashv p_* \dashv p^!$.
\end{prop}

\begin{proof}
Suppose that \nref[Hypothesis]{functors.a11.1} is satisfied.
Let $p$ be a proper morphism of $\sch[ft]{S}$.
Since the functor $p^*$ preserves $\aleph_0$-presentable objects, its right adjoint $p_*$ preserves small $\aleph_0$-filtered colimits (\cite[5.5.7.2]{Lurie_higher-topos}).
Since $p_*$ is a right-adjoint functor, it preserves finite limits (\cite[1.1.4.1]{Lurie_higher-algebra}).
Its domain and codomain are stable, so $p_*$ is exact and preserves finite colimits.
It follows that $p_*$ is cocontinuous (\cite[1.4.4.1]{Lurie_higher-algebra}).
By the Adjoint Functor Theorem (\cite[5.5.2.9.]{Lurie_higher-topos}), $p_*$ therefore admits a right adjoint $p^!$. 

Suppose now that \nref[Hypothesis]{functors.a11.2} is satisfied.
We will apply \nref{functors.a8} and the generalization of \cite[1.4.2]{Ayoub_six-operationsI} to the proper nonprojective case given in \cite{Cisinski-Deglise_triangulated-categories}.
The Adjoint Functor Theorem provides an accessible right adjoint $p^* \dashv p_*$.
The claim will now follow once we show that $p_*$ preserves small colimits.
As a continuous functor between stable \qcategories, $p_*$ is exact, so it suffices show that $p_*$ preserves small coproducts by \cite[1.4.4.1.(2)]{Lurie_higher-algebra}.

Since $p_*$ is exact, $\ho{p_*}$ is triangulated.
The homotopy categories of the stable \locpres\ \qcategories\ $\coef{M}{X}$ and $\coef{M}{Y}$ are well generated (\cite[8.1.7]{Neeman_triangulated-categories}).
Using the remark that homotopy groups commute with products, one checks that $p_*$ preserves small coproducts if and only if $\ho{p_*}$ does, so it remains to prove the latter assertion.

By the argument given for \nref{functors.a8}, $\hocoef[um = {*, \otimes}]{M}{}$ is a well-generated motivic triangulated category over $\sch[ft]{S}$ in the sense of \cite[2.4.45]{Cisinski-Deglise_triangulated-categories}.
By \cite[2.4.26, 2.4.28, 2.4.47]{Cisinski-Deglise_triangulated-categories}, $\ho{p_*}$ admits a right adjoint.
Thus, $\ho{p_*}$ does in fact preserve small coproducts.
\end{proof}

\begin{rmk}
\nlabel{functors.a12}
Suppose $\coef[um = {*, \otimes}]{M}{}$ is a \locpres\ \coeffsyst.
By \nref{functors.a11}, the restriction of the composite functor $\coef[dm = {*}]{M}{}: \sch[ft]{S} \to \pr[u = {R}, d = {st}] \hookrightarrow \QCATex$ resulting from \nref[Remark]{functors.a7.A} to the subcategory $\sch[prop]{S} \subseteq \sch[ft]{S}$ spanned by the proper morphisms factors through the inclusion $\pr[u = {L}, d = {st}] \hookrightarrow \QCATex$.
We therefore deduce a composite functor 
\[
\coef[um = {!}]{M}{}: 
  \prns{\sch[prop]{S}}\op
    \xrightarrow{\prns{\coef[dm = {*}]{M}{}}\op} \prns{\pr[u = {L}, d = {st}]}\op
      \simeq \pr[u = {R}, d = {st}]
\]
given informally by assigning to each proper morphism $p: X \to Y$ of $\sch[ft]{S}$ a functor $\coef[um = {!}]{M}{p}: \coef{M}{Y} \to \coef{M}{X}$ right adjoint to $p_*$.

\begin{notation*}
For each proper morphism $p: X \to Y$ of $\sch[ft]{S}$, we set $p^! \coloneqq \coef[um = {!}]{M}{p}: \coef{M}{Y} \to \coef{M}{X}$, so that $p_* \dashv p^!$.
\end{notation*}
\end{rmk}

\begin{thm}
[{Ayoub, Liu-Zheng, Robalo}]
\nlabel{functors.a13}
Let $\coef[um = {*, \otimes}]{M}{}: \prns{\sch[ft]{S}}\op \to \calg{\pr[um = {\tu{L}, \otimes}, d = {st}]}$ be a \locpres\ \coeffsyst.
There exists a functor $\coef[dm = {!}]{M}{}: \sch[sepft]{S} \to \pr[u = {L}, d = {st}]$ satisfying the following properties.
\begin{enumerate}[ref=Theorem~$\thethm.(\arabic*)$]
\item 
\nlabel{functors.a13.2}
\textbf{Gluing property for exceptional pushforwards:}
There is a natural equivalence between the functors $\sch[open]{S} \to \QCATex$ obtained from $\coef[dm = {!}]{M}{}$ and $\coef[dm = {\sharp}]{M}{}$ by restriction.
There is also a natural equivalence between the functors $\sch[prop]{S} \to \QCATex$ obtained from $\coef[dm = {!}]{M}{}$ and $\coef[dm = {*}]{M}{}$ by restriction.
\item
\nlabel{functors.a13.3}
\textbf{Relative purity:}
For each separated, smooth morphism $f: X \to Y$ in $\sch[ft]{S}$, letting $\delta: X \hookrightarrow X\times_YX$ denote the diagonal morphism, $p: \spec{\sym{\normalbundle{\delta}}} \to X$ the normal bundle associated with $\delta$ and $s: X \hookrightarrow \spec{\sym{\normalbundle{\delta}}}$ the zero section, there is a natural equivalence
$
f_{\sharp}
\isom
f_!\Thom{p,s}
$.
\item
\nlabel{functors.a13.4}
\textbf{Proper base change:}
Given morphisms $f: X \to Y$ and $g: Y' \to Y$ in $\sch[ft]{S}$ with $f$ separated, the functors $g^*f_!$ and $f_{\prns{Y'}!}g^*_{\prns{X}}$ are equivalent.
\item
\nlabel{functors.a13.5}
\textbf{Projection formula:}
For each morphism $f: X \to Y$ of $\sch[ft]{S}$, the functors $f_!\prns{-\otimes_{\coef{M}{X}} \fct{f^*}{-}}$ and $\fct{f_!}{-} \otimes_{\coef{M}{Y}} \prns{-}$ are equivalent.
\end{enumerate}
\end{thm}

\begin{proof}
The claims follow from the combination of the following results:
\cite[Scholie 1.4.2]{Ayoub_six-operationsI}, \cite[2.4.50]{Cisinski-Deglise_triangulated-categories}, \cite[Corollary 0.3]{Liu-Zheng_gluing-restricted} and  \cite[\S9.4]{Robalo_thesis}.
\end{proof}

\begin{notation*}
With the notation and hypotheses of \nref{functors.a13}, we set $f_! \coloneqq \coef[dm = {!}]{M}{f}: \coef{M}{X} \to \coef{M}{Y}$ for each morphism $f: X \to Y$ of $\sch[sepft]{S}$, and we refer to it as the \emph{exceptional pushforward along $f$}.
\end{notation*}

\begin{rmk}
\nlabel{functors.a14}
We elaborate on the above conditions and introduce some notation:
\begin{enumerate}
\item 
\nlabel{functors.a14.1}
Let $\coef[um = {*, \otimes}]{M}{}$ be a \locpres\ \coeffsyst.
As in \nref[Remark]{functors.a7.A}, we obtain a functor $\coef[um = {!}]{M}{}: \prns{\sch[sepft]{S}}\op \to \pr[u = {R}]$ such that $f_! \dashv f^! \coloneqq \coef[um = {!}]{M}{f}$.
We refer to $f^!$ as the \emph{exceptional pullback along $f$}.
By \nref[Theorem]{functors.a13.2}, the restriction of $\coef[um = {!}]{M}{}$ to $\prns{\sch[prop]{S}}\op$ is naturally equivalent to the functor denoted by the same symbol in \nref{functors.a12}.
\item
\nlabel{functors.a14.2}
By \nref{functors.a8} and \cite[1.7.4]{Ayoub_six-operationsI}, there is a lax natural transformation $\hocoef[dm = {!}]{M}{} \to \hocoef[dm = {*}]{M}{}$ of pseudofunctors from $\sch[sepft]{S}$ to the $(2,1)$-category of triangulated categories.
We were unable to locate a construction in the literature of an analogous oplax natural transformation $\coef[dm = {!}]{M}{} \to \coef[dm = {*}]{M}{}: \sch[sepft]{S} \to \QCATex$.
On the other hand, we were also unable to find essential applications of the existence of such an oplax natural transformation, even in the language of triangulated categories.
\item 
\nlabel{functors.a14.3}
One often refers to the data and compatibilities described in \nref{functors.3} and \nref{functors.a13} as a \emph{Grothendieck six-functor formalism}, although precisely which compatibilities are included in such a formalism may depend on the context.
\end{enumerate}
\end{rmk}

\begin{prop}
\nlabel{prop:exchange-for-proper-pushforward}
Let $\phi^{*, \otimes}: \coef[um = {*, \otimes}]{M}{} \to \coef[um = {*, \otimes}]{N}{}$ be a morphism of \coeffsysts\ 
and $f: X \to Y$ a proper morphism of $\sch[ft]{S}$.
The exchange morphism $\phi^*_Yf_* \to f_*\phi^*_X$ is an equivalence.
\end{prop}

\begin{proof}
This follows from \cite[2.3.11, 2.4.53]{Cisinski-Deglise_triangulated-categories}.
\end{proof}

\section{Excision properties for \coeffsysts}

\setcounter{thm}{-1}

\begin{notation}
In this section, we fix the following notation:
\begin{itemize}
\item
$S$, a Noetherian scheme of finite dimension; and
\item
$\coef[um = {*, \otimes}]{M}{}: \prns{\sch[ft]{S}}\op \to \calg{\QCAT[um = {\tu{Ex}, \otimes}]}$, a \coeffsyst\ over $S$.
\end{itemize}
\end{notation}

\begin{motivation}
In this short section, we recall the notion of excision in a \coeffsyst\ for Nisnevich distinguished squares and for blow-ups, and we observe that the proofs of \cite[\S3.3]{Cisinski-Deglise_triangulated-categories} readily adapt to our \qcategorical\ framework to show that these excision properties follow from the localization axiom and proper base change.
\end{motivation}

\begin{defn}
\nlabel{excision.1}
A commutative square
\begin{equation}
\nlabel{excision.1.a}
\begin{tikzcd}
W
\ar[r, "j'" below, hookrightarrow]
\ar[d, "e'" left]
&
V
\ar[d, "e" right]
\\
U
\ar[r, "j" above, hookrightarrow]
&
X
\end{tikzcd}
\end{equation}
of schemes is \emph{Nisnevich distinguished} if it satisfies the following conditions:
\begin{enumerate}
\item
the square is Cartesian;
\item
$j$ is an open immersion;
\item
$e$ is \'etale; and
\item
the induced morphism $e'': e^{-1}\prns{X-U}\red \to \prns{X-U}\red$ is an isomorphism.
\end{enumerate}
\end{defn}

\begin{prop}[Nisnevich excision]
\nlabel{excision.2}
Consider a Nisnevich-distinguished square \nref{excision.1.a}, and let $M \in \coef{M}{X}$.
The essentially commutative square
\[
\begin{tikzcd}
\prns{ej'}_{\sharp}\prns{ej'}^*M
\ar[r]
\ar[d]
&
e_{\sharp}e^*M
\ar[d]
\\
j_{\sharp}j^*M
\ar[r]
&
M
\end{tikzcd}
\]
in $\coef{M}{X}$ whose arrows are the counits of the associated adjunctions is coCartesian.
\end{prop}

\begin{proof}
The proof is identical to that of \cite[3.3.4]{Cisinski-Deglise_triangulated-categories}.
\end{proof}

\begin{defn}
\nlabel{excision.3}
A commutative square
\begin{equation}
\nlabel{excision.3.a}
\begin{tikzcd}
E
\ar[r, "i'" below, hookrightarrow]
\ar[d, "p'" left]
&
Y
\ar[d, "p" right]
\\
Z
\ar[r, "i" above, hookrightarrow]
&
X
\end{tikzcd}
\end{equation}
of schemes is \emph{\cdh-distinguished} if it satisfies the following conditions:
\begin{enumerate}
\item
the square is Cartesian;
\item
$i$ is a closed immersion;
\item
$p$ is proper and surjective; and
\item
the induced morphism $p'': p^{-1}\prns{X-Z} \to X-Z$ is an isomorphism.
\end{enumerate}
\end{defn}

\begin{prop}[Excision for blow-ups]
\nlabel{excision.4}
Consider a \cdh-distinguished square \nref{excision.3.a}, and let $M \in \coef{M}{X}$.
The essentially commutative square
\[
\begin{tikzcd}
M
\ar[r]
\ar[d]
&
p_*p^*M
\ar[d]
\\
i_*i^*M
\ar[r]
&
\prns{pi'}_*\prns{pi'}^*M
\end{tikzcd}
\]
in $\coef{M}{X}$ whose arrows are the units of the associated adjunctions is Cartesian.
\end{prop}

\begin{proof}
The proof is identical to that of \cite[3.3.10.$(i)$]{Cisinski-Deglise_triangulated-categories}.
\end{proof}

\section{Constructibility and duality}
\nlabel{construct}

\setcounter{thm}{-1}

\begin{notation}
\nlabel{construct.0}
Throughout this section, we fix:
\begin{itemize}
\item
$\mc{V}^{\otimes}$, a \locpres[\aleph_0]\ symmetric monoidal \qcategory; and
\item
$S$, a Noetherian scheme of finite dimension.
\end{itemize}
\end{notation}

\begin{motivation}
One of the important features of six-functor formalisms is their compatibility with constructible objects.
This is important, for example, in arguments involving d\'evissage, for example. 
In this section, we show that the results of Ayoub and Cisinski-D\'eglise in this direction, as well as their results on the existence of a theory of Verdier duality for constructible objects, can be translated to the \qcategorical\ framework introduced in \nref{construct}.
\end{motivation}

\begin{summary}
\
\begin{itemize}
\item
We begin by introducing the basic definitions of constructible objects, compatibility of a \coeffsyst\ with constructibility, and compatibility of a \coeffsyst\ with duality.
\item
In \nref{prop:constructible-subfunctor}, we show that the \subqcategories\ spanned by constructible objects form a \coeffsyst\ in their own right.
\item
In \nref{construct.a4}, we establish sufficient conditions for a \coeffsyst\ to be compatible with constructibility.
\item
In \nref{construct.a6}, we establish sufficient conditions for a \coeffsyst\ to be compatible with duality.
\item
In \nref{construct.a7}, we show that each \coeffsyst\ compatible with duality admits a theory of Verdier duality intertwining the ordinary and exceptional pullback and pushforward functors.
\item
In \nref{prop:morphisms-compatible-with-six-functors}, we remark that when $S$ is the spectrum of a field of characteristic zero, the techniques of the proof of \cite[4.4.25]{Cisinski-Deglise_triangulated-categories} apply in our \qcategorical\ context to show that morphisms of \coeffsysts\ commute with the six operations when restricted to constructible objects.
\item
In \nref{prop:conservativity}, we establish a sufficient criterion for a morphism of \coeffsysts\ to be conservative when restricted to constructible objects.
\end{itemize}
\end{summary}

\begin{defn}
\nlabel{construct.a1}
\nlabel{construct.a1.A}
Let $\coef[um = {*, \otimes}]{M}{}$ be a $\mc{V}^{\otimes}$-linear \locpres\ \coeffsyst.
For each $X \in \sch[ft]{S}$, the \emph{$\mc{V}^{\otimes}$-constructible objects} of $\coef{M}{X}$ are the objects of the smallest replete, idempotent-complete (\cite[\S4.4.5]{Lurie_higher-topos}) stable \subqcategory\ $\coef[d = {c}]{M}{X} \subseteq \coef{M}{X}$ containing the objects of the form $\twist{f_{\sharp}f^*\1{\coef{M}{X}}}{r} \otimes V$ for each smooth morphism $f: Y \to X$ in $\sch[ft]{S}$, each $r \in \integer$, and each $V \in \mc{V}_{\aleph_0}$.
When $\mc{V} = \spc{}^{\times}$, we speak simply of \emph{constructible objects}.
\end{defn}

\begin{defn}
\nlabel{construct.a1.B}
We say that the \coeffsyst\ $\coef[um = {*, \otimes}]{M}{}$ is \emph{$\mc{V}^{\otimes}$-quasi-constructible} if it satisfies the following conditions:
\begin{enumerate}
\item
\nlabel{construct.a1.B.a}
for each closed immersion $i: Z \hookrightarrow X$ between regular schemes in $\sch[ft]{S}$, each $r \in \integer$, and each $V \in \mc{V}_{\aleph_0}$, the object $i^!\prns{\twist{\1{\coef{M}{X}}}{r} \otimes V}$ is $\mc{V}^{\otimes}$-constructible; and
\item
\nlabel{construct.a1.B.b}
for each $X \in \sch[ft]{S}$, each $r \in \integer$, each $V \in \mc{V}_{\aleph_0}$, and each $M \in \coef[d = {c}]{M}{X}$, the object 
\[
\intmor[dm = {\coef{M}{X}}]{\tate{\coef{M}{X}}{r} \otimes V}{M}
  \in \coef{M}{X}
\]
is $\mc{V}^{\otimes}$-constructible.
\end{enumerate}
When $\mc{V} = \spc{}^{\times}$, we say that $\coef[dm = {*, \otimes}]{M}{}$ is \emph{quasi-constructible}.
\end{defn}

\begin{defn}
\nlabel{construct.a1.C}
We say that $\coef[um = {*, \otimes}]{M}{}$ is \emph{$\mc{V}^{\otimes}$-constructible} if it satisfies the following conditions:
\begin{enumerate}
\item
\nlabel{construct.a1.C.a}
for each morphism $f$, each smooth morphism $p$, and each separated morphism of finite type $g$ in $\sch[ft]{S}$, the functors $f^*$, $f_*$, $p_{\sharp}$, $g_!$ and $g^!$ preserve $\mc{V}^{\otimes}$-constructible objects; and
\item
\nlabel{construct.a1.C.b}
for each $X \in \sch[ft]{S}$, the bifunctors $\prns{-} \otimes_{\coef{M}{X}} \prns{-}$ and $\intmor[dm = {\coef{M}{X}}]{-}{-}$ send pairs of $\mc{V}^{\otimes}$-constructible objects to $\mc{V}^{\otimes}$-constructible objects.
\end{enumerate}
When $\mc{V}^{\otimes} = \spc{}^{\times}$, we say that $\coef[um = {*, \otimes}]{M}{}$ is \emph{constructible}.
\end{defn}

\begin{defn}
\nlabel{construct.a1.D}
We say that $\coef[um = {*, \otimes}]{M}{}$ is \emph{$\mc{V}^{\otimes}$-dualizable} if it satisfies the following conditions:
\begin{enumerate}
\item
\nlabel{construct.a1.D.a}
$\coef[um = {*, \otimes}]{M}{}$ is $\mc{V}^{\otimes}$-quasi-constructible;
\item
\nlabel{construct.a1.D.b}
for each closed immersion $i : Z \hookrightarrow X$ between regular schemes in $\sch[ft]{S}$, the object $i^!\1{\coef{M}{X}}$ is $\otimes$-invertible; and
\item
\nlabel{construct.a1.D.c}
for each regular scheme $X$ in $\sch[ft]{S}$, each $r \in \integer$, and each $V \in \mc{V}_{\aleph_0}$, the morphism
\[
\tate{\coef{M}{X}}{r} \otimes V
  \to \intmor[size = 0, dm = {\coef{M}{X}}]{\intmor[dm = {\coef{M}{X}}]{\tate{\coef{M}{X}}{r} \otimes V}{\1{\coef{M}{X}}}}{\1{\coef{M}{X}}}
\]
adjoint to the evaluation morphism is an equivalence.
\end{enumerate}
When $\mc{V}^{\otimes} = \spc{}^{\times}$, we say that $\coef[um = {*, \otimes}]{M}{}$ is dualizable.
\end{defn}

\begin{prop}
\nlabel{prop:constructible-subfunctor}
\nlabel{construct.a2.B}
Let $\phi^{*, \otimes}: \coef[um = {*, \otimes}]{M}{} \to \coef[um = {*, \otimes}]{N}{}$ be a morphism of $\mc{V}^{\otimes}$-constructible $\mc{V}^{\otimes}$-linear \coeffsysts.
The following properties hold:
\begin{enumerate}
\item
\nlabel{prop:constructible-subfunctor.1}
the assignment $X \mapsto \coef[d = {c}]{M}{X}$ underlies a \coeffsyst\ $\coef[um = {*, \otimes}, d = {c}]{M}{}: \prns{\sch[ft]{S}}\op \to \calg{\qcatexmon}$,
and the inclusions $\iota_X: \coef[d = {c}]{M}{X} \hookrightarrow \coef{M}{X}$ underlie a morphism of \coeffsysts\ $\iota^*_X: \coef[um = {*, \otimes}, d = {c}]{M}{} \to \coef[um = {*, \otimes}]{M}{}$; and
\item
\nlabel{prop:constructible-subfunctor.2}
\nlabel{prop:morphisms-compatible-with-six-functors.2}
$\phi^{*, \otimes}$ restricts to a morphism of \coeffsysts\ $\phi^{*, \otimes}\cnstr: \coef[um = {*, \otimes}, d = {c}]{M}{} \to \coef[um = {*, \otimes}, d = {c}]{N}{}$.
\end{enumerate}
\end{prop}

\begin{proof}
Consider \nref[Claim]{prop:constructible-subfunctor.1}.
Abusing notation slightly, identify $\coef[um = {*, \otimes}]{M}{}$ with the corresponding functor $\prns{\sch[ft]{S}}\op \to \calg{\QCATmon}$.
Let $\hocoef[um = {*, \otimes}]{M}{}: \prns{\sch[ft]{S}}\op \to \calg{\QCATexmon}$ denote the composite of $\coef[um = {*, \otimes}]{M}{}$ with $\ho[dm = {\calg{}}]{}: \calg{\QCATmon} \to \calg{\gmnerve{\CATtwo}^{\times}}$, where $\CATtwo$ denotes the strict $(2,1)$-category of large categories, functors and natural isomorphisms.
We regard $\hocoef[um = {*, \otimes}]{M}{}$ as a pseudofunctor from $\prns{\sch[ft]{S}}\op$ into the strict $\prns{2,1}$-category of large unbiased symmetric monoidal categories.

Since $\mc{V}^{\otimes}$-constructible objects are stable under tensor products and pullbacks, the assignment $X \mapsto \ho{\coef[d = {c}]{M}{X}}$ underlies a full sub-pseudofunctor of $\hocoef[um = {*, \otimes}]{M}{}$,
i.e., there is a pseudonatural transformation $\ho{\iota}^*: \hocoef[um = {*, \otimes}, d = {c}]{M}{} \to \hocoef[um = {*, \otimes}]{M}{}$ such that $\ho{\iota}^*_X$ is the inclusion $\hocoef[d = {c}]{M}{X}^{\otimes} \hookrightarrow \hocoef{M}{X}^{\otimes}$ for each $X \in \sch[ft]{S}$.

Recall from \nref{rmk:ho} that $\ho[dm = {\calg{}}]{}$ admits a right adjoint $\nerve[dm = {\calg{}}]{}$.
For the Cartesian square
\[
\begin{tikzcd}
\coef[um = {*, \otimes}, d = {c}]{M}{}
\ar[r, "\iota^*" below]
\ar[d]
&
\coef[um = {*, \otimes}]{M}{}
\ar[d, "\eta" right]
\\
\nerve[dm = {\calg{}}]{} \hocoef[um = {*, \otimes}, d = {c}]{M}{}
\ar[r, "\ho{\iota}^*" above]
&
\nerve[dm = {\calg{}}]{} \hocoef[um = {*, \otimes}]{M}{}
\end{tikzcd}
\]
in $\psh{\sch[ft]{S}}{\calg{\QCATmon}}$.
Since fiber products in functor \qcategories\ are computed pointwise, $\iota^*_X$ corresponds to the inclusion $\coef[d = {c}]{M}{X} \hookrightarrow \coef{M}{X}$ for each $X \in \sch[ft]{S}$.

It remains to check that $\coef[um = {*, \otimes}, d = {c}]{M}{}$ is a \coeffsyst.
Since $\coef[d = {c}]{M}{X} \subseteq \coef{M}{X}$ is a stable \subqcategory\ by definition, it follows that $\coef[um = {*, \otimes}, d = {c}]{M}{}$ factors through the inclusion $\calg{\QCATexmon} \hookrightarrow \calg{\QCATmon}$.
The axioms of \nref{functors.3} follow readily from the hypothesis that $\coef[um = {*, \otimes}]{M}{}$ is $\mc{V}^{\otimes}$-constructible and \nref{functors.a7.C}.

Consider \nref[Claim]{prop:constructible-subfunctor.2}.
Let $f: X \to Y$ be a smooth morphism of $\sch[ft]{S}$, $V \in \mc{V}_{\aleph_0}$ and $r \in \integer$.
Since $\phi^{*, \otimes}$ is a morphism of $\mc{V}^{\otimes}$-linear \coeffsysts, we have equivalences
\[
\fct{\phi^*_Y}{V \odot f_{\sharp}\tate{\coef{M}{X}}{r}}
  \simeq V \odot \phi^*_Yf_{\sharp}\tate{\coef{M}{X}}{r}
  \simeq V \odot f_{\sharp}\phi^*_X\tate{\coef{M}{X}}{r}
  \simeq V \odot f_{\sharp}\tate{\coef{N}{X}}{r}.
\]
Since $\phi^*_Y$ is exact, \nref[Claim]{prop:constructible-subfunctor.2} follows.
\end{proof}

\begin{rmk}
\nlabel{construct.a2}
Let $\coef[um = {*, \otimes}]{M}{}$ be a $\mc{V}^{\otimes}$-linear, \locpres\ \coeffsyst.
\begin{enumerate}
\item
\nlabel{construct.a2.A}
For each smooth morphism $f: Y \to X$ in $\sch[ft]{S}$, $f_{\sharp}$ preserves $\aleph_0$-presentable objects, since its right adjoint $f^*$ is cocontinuous.
Therefore, if $\1{\coef{M}{Y}} \in \coef{M}{Y}_{\aleph_0}$ for each smooth morphism $f: Y \to X$ in $\sch[ft]{S}$, e.g., if $\coef[um = {*, \otimes}]{M}{}$ is \locpres[\aleph_0], then $\coef[d = {c}]{M}{X} \subseteq \coef{M}{X}_{\aleph_0}$.
\item
\nlabel{construct.a2.C}
\nref[Condition]{construct.a1.B.a} is automatic if $S$ is the spectrum of a perfect field: it is equivalent to the quasi-purity (\cite[2.2.28]{Ayoub_six-operationsI}) of the class of objects of $\coef{M}{S}$ of the form $V \otimes \tate{\coef{M}{S}}{r}$ with $V \in \mc{V}_{\aleph_0}$ and $r \in \integer$, and the claim follows from \cite[2.2.29]{Ayoub_six-operationsI}.
\item
\nlabel{construct.a2.D}
\nref[Condition]{construct.a1.B.b} is automatic if $\mc{V}^{\otimes}$ is ind-rigid:
$\mc{V}_{\aleph_0}$ is stable under taking $\otimes$-duals and, for each $M \in \coef[d = {c}]{M}{X}$, each $r \in \integer$, and each $V \in \mc{V}_{\aleph_0}$, we have 
$
\intmor[dm = {\coef{M}{X}}]{\tate{\coef{M}{X}}{r} \otimes V}{M} 
\simeq
\tate{\coef{M}{X}}{-r} \otimes V^{\vee} \otimes M,
$
and, as explained in step $\tu{(a)}$ of the proof of \nref{construct.a4} below, the tensor product of the $\mc{V}^{\otimes}$-constructible objects $M$ and $\tate{\coef{M}{X}}{-r} \otimes V^{\vee}$ is $\mc{V}^{\otimes}$-constructible.
\item
\nlabel{construct.a2.E}
\nref[Condition]{construct.a1.D.b} is a euphemistic formulation of the ``absolute purity theorem'' in $\coef[um = {*, \otimes}]{M}{}$.
\item
\nlabel{construct.a2.F}
\nref[Condition]{construct.a1.D.c} is automatic if $\mc{V}^{\otimes}$ is ind-rigid:
if $M \coloneqq V \otimes \tate{\coef{M}{X}}{r}$ with $V \in \mc{V}_{\aleph_0}$ and $r \in \integer$, then the desired equivalence is just the canonical equivalence $M \isom \prns{M^{\vee}}^{\vee}$.
\item
\nlabel{construct.a2.G}
If $\phi^{*, \otimes}: \coef[um = {*, \otimes}]{M}{} \to \coef[um = {*, \otimes}]{N}{}$ is a morphism of $\mc{V}^{\otimes}$-linear \locpres\ \coeffsysts, then $\phi^*$ preserves $\mc{V}^{\otimes}$-constructible objects.
Indeed, it suffices to show that $\fct{\phi^*}{f_{\sharp}f^*\tate{\coef{M}{X}}{r} \odot V}$ is a $\mc{V}^{\otimes}$-constructible object of $\coef{N}{X}$ for each smooth morphism $f: Y \to X$, each $r \in \integer$, and each $V \in \mc{V}_{\aleph_0}$.
By $\mc{V}^{\otimes}$-linearity and the adjointability condition of \nref{functors.a6}, we have equivalences
\[
\fct{\phi^*}{f_{\sharp}f^*\tate{\coef{M}{X}}{r} \odot V}
  \simeq \fct{\phi^*}{f_{\sharp}f^*\tate{\coef{M}{X}}{r}} \odot V
  \simeq f_{\sharp}f^*\fct{\phi^*}{\tate{\coef{M}{X}}{r}} \odot V
  \simeq f_{\sharp}f^*\tate{\coef{N}{X}}{r} \odot V,
\]
and the claim follows.
\end{enumerate}
\end{rmk}

\begin{defn}
\nlabel{construct.1}
\nlabel{construct.1.2}
Let $\coef[um = {*, \otimes}]{M}{}: \prns{\sch[ft]{S}}\op \to \calg{\QCATexmon}$ be a functor.
We say that $\coef[um = {*, \otimes}]{M}{}$ is \emph{separated} \resp{\emph{semi-separated}} if, for each surjective morphism of finite type \resp{for each surjective, finite, radicial morphism} $f: X \to Y$ in $\sch[ft]{S}$, the functor $\coef[u = {*}]{M}{f}$ is conservative.
Alternatively, $\coef[um = {*, \otimes}]{M}{}$ is separated \resp{semi-separated} if the associated pseudofunctor $\hocoef[um = {*, \otimes}]{M}{}$ is separated \resp{semi-separated} in the sense of \cite[2.1.60]{Ayoub_six-operationsI}.
\end{defn}

\begin{defn}
\nlabel{construct.a3}
Let $\mc{V}^{\otimes}$ be a \locpres\ symmetric monoidal \qcategory.
We say that the pair $\prns{S, \mc{V}^{\otimes}}$ is \emph{solvent} if it satisfies \emph{one} of the following conditions:
\begin{enumerate}
\item
\nlabel{construct.a3.1}
$S = \spec{\kk}$ is the spectrum of a field $\kk$ of characteristic zero;
or
\item
\nlabel{construct.a3.2}
$S$ is excellent, $\dim{S} \leq 2$, and $\mc{V}^{\otimes}$ is $\rational$-linear.
\end{enumerate}
If $\coef[um = {*, \otimes}]{M}{}$ is a $\mc{V}^{\otimes}$-linear \locpres\ \coeffsyst, then we say that the pair $\prns{S, \coef[um = {*, \otimes}]{M}{}}$ is \emph{solvent} if $S$ satisfies \nref[Condition]{construct.a3.1} or the following condition is satisfied:
\begin{enumerate}
\setcounter{enumi}{2}
\item
\nlabel{construct.a3.3}
$\prns{S, \mc{V^{\otimes}}}$ satisfies \nref[Condition]{construct.a3.2} and $\coef[um = {*, \otimes}]{M}{}$ is separated.
\end{enumerate}
These conditions allow us to apply resolutions of singularities to prove statements about $\coef[um = {*, \otimes}]{M}{}$.
\end{defn}

\begin{rmk}
The restriction to characteristic zero in \nref{construct.a3.1} conflates two phenomena:
it guarantees all extensions of function fields are perfect which implies semi-separatedness (\cite[2.1.161]{Ayoub_six-operationsI});
and it guarantees that $\kk$ admits resolution of singularities by blow-ups.
\end{rmk}

\begin{prop}
\nlabel{construct.a4}
Let $\coef[um = {*, \otimes}]{M}{}$ be a $\mc{V}^{\otimes}$-linear \locpres\ \coeffsyst.
\begin{enumerate}
\item
\nlabel{construct.a4.1}
If $\prns{S, \coef[um = {*, \otimes}]{M}{}}$ is solvent and $\coef[um = {*, \otimes}]{M}{}$ is $\mc{V}^{\otimes}$-quasi-constructible, then $\coef[um = {*, \otimes}]{M}{}$ is $\mc{V}^{\otimes}$-constructible.
\item
\nlabel{construct.a4.2}
If $S = \spec{\kk}$ is the spectrum of a perfect field $\kk$, if $\mc{V}^{\otimes}$ is ind-rigid, and if $\prns{S, \coef[um = {*, \otimes}]{M}{}}$ is solvent, then $\coef[um = {*, \otimes}]{M}{}$ is $\mc{V}^{\otimes}$-constructible.
\end{enumerate}
\end{prop}

\begin{proof}
Consider \nref[Claim]{construct.a4.1}.
If we restrict to the category of quasi-projective $S$-schemes, then this follows from \cite[2.2.34, 2.3.65]{Ayoub_six-operationsI}.
On the other hand, if we work with $S$-schemes of finite type, and if we assume that $\coef[um = {*, \otimes}]{M}{}$ is $\rational$-linear, is separated, and arises from a combinatorial left Quillen presheaf (\cite[2.21]{Barwick_left-and-right}), then this follows from \cite[4.2.29]{Cisinski-Deglise_triangulated-categories}.

Out of an abundance of caution, let us extract from these sources a proof of our variation on this theme.
Let $f: X \to Y$ be a morphism of $\sch[ft]{S}$.
\begin{enumerate}[label=\tu{(\alph*)}, itemsep=0ex, topsep=0ex]
\item
The proofs of \cite[4.2.3, 4.2.4]{Cisinski-Deglise_triangulated-categories} apply without modification to show that the bifunctor $\prns{-}\otimes_{\coef{M}{X}}\prns{-}$, the functor $f^*$ and, if $f$ is smooth, the functor $f_{\sharp}$ all preserve $\mc{V}^{\otimes}$-constructible objects.
\item
If $f$ is separated of finite type, the proof of \cite[3.3.10.$(i)$]{Cisinski-Deglise_triangulated-categories}, and therefore also of \cite[4.2.12]{Cisinski-Deglise_triangulated-categories}, translates directly to our situation, from which we deduce that $f_!$ preserves $\mc{V}^{\otimes}$-constructible objects.
\item
The proofs of \cite[2.2.30, 2.2.32]{Ayoub_six-operationsI} also go through without change to show that if $i: Z \hookrightarrow X$ is a closed immersion and $j: U \hookrightarrow X$ an open immersion in $\sch[ft]{S}$, then $i^!$ and $j_*$ preserve $\mc{V}^{\otimes}$-constructible objects.
Note that schemes in \cite{Ayoub_six-operationsI} are assumed to be quasi-projective, but $\mc{V}^{\otimes}$-constructibility is a Zariski-local property and $i$ is affine, so we can work affine-locally on $X$, thereby reducing to the quasi-projective case.
\item
Using $\tu{(c)}$, the proof of \cite[4.2.28]{Cisinski-Deglise_triangulated-categories} now shows that $f^!$ preserves $\mc{V}^{\otimes}$-constructible objects if $f$ is separated of finite type.
\item
If $f$ is separated of finite type, then it admits a Nagata compactification, i.e., it factors as $pj$ with $j$ an open immersion and $p$ a proper morphism.
The cases $\tu{(b)}$ and $\tu{(c)}$ show that $f_* \simeq p_*j_* \simeq p_!j_*$ preserves $\mc{V}^{\otimes}$-constructible objects.
\item
Let $f$ be arbitrary and let $M \in \coef[d = {c}]{M}{X}$.
We claim that $f_*M$ is $\mc{V}^{\otimes}$-constructible.
Let $\brc{j_{\alpha}: U_{\alpha} \hookrightarrow Y}_{\alpha \in A}$ be a Zariski cover of $Y$ by finitely many affine open subschemes.
As the proof of \cite[4.2.6]{Cisinski-Deglise_triangulated-categories} shows, it suffices to prove that $j^*_{\alpha}f_*M$ is $\mc{V}^{\otimes}$-constructible for each $\alpha \in A$.
By the smooth base change property (\nref[Axiom]{functors.3.C}), $j^*_{\alpha}f_*M \simeq f_{\prns{U_{\alpha}}*} j^*_{\alpha \prns{X}}M$ and we may henceforth assume $Y$ is affine.

We proceed by induction on $\dim{X}$.
If $\dim{X} = 0$, then the Noetherian scheme $X$ is affine.
Since $Y$ is also affine, $f$ is separated and the claim follows from $\tu{(e)}$.
Suppose that $g_*$ preserves $\mc{V}^{\otimes}$-constructible objects for each morphism $g: W \to Y$ in $\sch[ft]{S}$ such that $\dim{W} < \dim{X}$.
As $X$ is Noetherian, it admits a dense, affine open $U \subseteq X$.
Indeed, if $\brc{Z_k}_{1 \leq k \leq r}$ are the irreducible components of $X$, for each $1 \leq k \leq r$, there exists an affine open $\varnothing \neq U_k \hookrightarrow X$ such that $U_k \subseteq Z_k$ and $U_k \cap Z_{\ell} = \varnothing$ for each $1 \leq \ell \neq k \leq r$.
As $U_k \cap U_{\ell} = \varnothing$ for $k \neq \ell$, the union $U \coloneqq \bigcup_{1 \leq k \leq r} U_k$ is the desired dense, affine open subscheme of $X$.
Let $j: U \hookrightarrow X$ denote the inclusion and $i: Z \coloneqq X - U \hookrightarrow X$ the complementary reduced closed subscheme.
By construction, $\dim{Z} < \dim{X}$.
The localization property (\nref[Axiom]{functors.3.E}) gives a fiber sequence
$ f_*i_*i^!M 
\to
f_*M
\to
f_*j_*j^*M $.
Since $fj$ is separated, the $\mc{V}^{\otimes}$-constructibility of $f_*M$ follows from the equivalences $f_*i_* \simeq \prns{fi}_*$ and $f_*j_* \simeq \prns{fj}_*$, the inductive hypothesis and the cases $\tu{(a)}$, $\tu{(c)}$ and $\tu{(e)}$.
\item
Finally, $\intmor[dm = {\coef{M}{X}}]{-}{-}$ preserves $\mc{V}^{\otimes}$-constructible objects by \cite[4.2.25]{Cisinski-Deglise_triangulated-categories}.
\end{enumerate}

\nref[Claim]{construct.a4.2} follows from \nref[Claim]{construct.a4.1} and \nref[Remarks]{construct.a2.C} and \nref{construct.a2.D}.
\end{proof}

\begin{lemma}
\nlabel{construct.a5}
Suppose that $\coef[um = {*, \otimes}]{M}{}$ is a \locpres\ \coeffsyst, and that $f: X \to Y$ is a morphism of $\sch[ft]{S}$.
For each $\otimes$-dualizable object $M \in \coef{M}{Y}$, we have equivalences of the following forms: 
\begin{enumerate}
\item
\nlabel{construct.a5.1}
$f^*M \otimes_{\coef{M}{X}} \fct{f^*}{-} \simeq \fct{f^*}{M \otimes_{\coef{M}{Y}} \prns{-}}$;
\item
\nlabel{construct.a5.2}
$\fct{f_{\sharp}}{ \prns{-} \otimes_{\coef{M}{X}} f^*M } 
  \simeq \fct{f_{\sharp}}{-} \otimes_{\coef{M}{Y}} M$;
\item
\nlabel{construct.a5.3}
$\fct{f_*}{-} \otimes_{\coef{M}{Y}} M \simeq \fct{f_*}{ \prns{-} \otimes_{\coef{M}{X}} f^*M }$;
\item
\nlabel{construct.a5.4}
$\fct{f_!}{-} \otimes_{\coef{M}{Y}} M \simeq \fct{f_!}{ \prns{-} \otimes_{\coef{M}{X}} f^*M }$;
\item
\nlabel{construct.a5.5}
$f^*M \otimes_{\coef{M}{X}} \fct{f^!}{-} \simeq \fct{f^!}{ M \otimes_{\coef{M}{Y}} \prns{-} }$.
\end{enumerate}
In particular, if $\coef[um = {*, \otimes}]{M}{}$ is $\mc{V}^{\otimes}$-linear, then $f_{\sharp}$, $f^*$, $f_*$, $f_!$ and $f^!$ commute with $\prns{V \otimes \tate{\coef{M}{X}}{r}} \otimes \prns{-}$ for each $\otimes$-dualizable object $V \in \mc{V}$, and each $r \in \integer$. 
\end{lemma}

\begin{proof}
The equivalence of \nref[Claim]{construct.a5.1} follows from the fact that $f^*$ is a symmetric monoidal functor, while \nref[Claim]{construct.a5.2} \resp{\nref[Claim]{construct.a5.4}} is a special case of \nref[Axiom]{functors.3.D} \resp{\nref{functors.a13.5}}.
Only \nref[Claims]{construct.a5.3} and \nref{construct.a5.5} require that $M$ be $\otimes$-dualizable:
\[
\begin{tikzcd}
f_*N \otimes_{\coef{M}{Y}} M
\ar[d]
\ar[r, "\sim"]
&
\intmor[dm = {\coef{M}{Y}}]{M^{\vee}}{f_*N}
\ar[d, "\sim"]
\\
\fct{f_*}{ N \otimes_{\coef{M}{X}} f^*M}
\ar[r, "\sim"]
&
f_*\intmor[dm = {\coef{M}{X}}]{f^*M^{\vee}}{N}
\end{tikzcd}
\]
and, similarly,
$
f^*M \otimes_{\coef{M}{X}} f^!N 
\simeq
\intmor[dm = {\coef{M}{X}}]{f^*M^{\vee}}{f^!N}
\simeq
f^!\intmor[dm = {\coef{M}{Y}}]{M^{\vee}}{N}
\simeq
\fct{f^!}{M \otimes_{\coef{M}{Y}} N}
$
by \nref{functors.a13.5}.
\end{proof}

\begin{prop}
[{\cite[4.4.16]{Cisinski-Deglise_triangulated-categories}}]
\nlabel{construct.a6}
Suppose that $\mc{V}^{\otimes}$ is ind-rigid \tu{(\nref{SH.10})}.
If $\coef[um = {*, \otimes}]{M}{}$ is a $\mc{V}^{\otimes}$-linear \locpres\ \coeffsyst, and if $S = \spec{\kk}$ is the spectrum of a perfect field $\kk$, then $\coef[um = {*, \otimes}]{M}{}$ is $\mc{V}^{\otimes}$-dualizable.
\end{prop}

\begin{proof}
\nref[Conditions]{construct.a1.D.a} and \nref{construct.a1.D.c} follow from \nref[Remarks]{construct.a2.C}, \nref{construct.a2.D} and \nref{construct.a2.F}.
It remains to check \nref[Condition]{construct.a1.D.b}.
Let $i: Z \hookrightarrow X$ be a closed immersion between regular schemes in $\sch[ft]{S}$, and let $p: Z \to S$ and $q: X \to Z$ denote the structure morphisms.
Let $j: U \hookrightarrow X$ be an open immersion in $\sm[ft]{S}$ with $U$ affine, and consider the Cartesian square
\[
\begin{tikzcd}
V
\ar[r, "i'" below]
\ar[d, "j'" left]
&
U
\ar[d, "j" right]
\\
Z
\ar[r, "i" above]
&
X
\end{tikzcd}
\]
in $\sm[ft]{S}$.
By \nref{functors.a13.2}, we have equivalences $j_{\sharp} \simeq j_!$ and $j'_{\sharp} \simeq j'_!$, whence equivalences of the associated right adjoints $j^* \simeq j^!$ and $j'^* \simeq j'^!$.
We therefore have equivalences
\[
j'^*i^!\1{\coef{M}{X}}
  \simeq j'^!i^!\1{\coef{M}{X}}
  \simeq i'^!j^!\1{\coef{M}{X}}
  \simeq i'^!j^*\1{\coef{M}{X}}
  \simeq i'^!\1{\coef{M}{U}}
\]
The condition that $i^!\1{\coef{M}{X}}$ be $\otimes$-invertible is Zariski-local on $Z$ by \cite[4.4.12]{Cisinski-Deglise_triangulated-categories}, so, replacing $i$ by $i'$, we may therefore assume that $X$ and, hence, $Z$ are quasi-projective.

Since $\kk$ is perfect and $X$ and $Z$ are regular, $p$ and $q$ are smooth.
By \nref{functors.a8} and \cite[1.6.19]{Ayoub_six-operationsI}, we have
\[i^!\1{\coef{M}{X}}
  \simeq i^!1^*\1{\coef{M}{S}}
  \simeq \Thom[um = {-1}]{\normalbundle{i}}p^*\1{\coef{M}{S}}
  \simeq \Thom[um = {-1}]{\normalbundle{i}}\1{\coef{M}{Z}},
\]
where $\normalbundle{i}$ is the normal bundle of the regular closed immersion $i$.
As explained in \nref{functors.a9.B}, $\Thom[um = {-1}]{\normalbundle{i}}\1{\coef{M}{Z}}$ is $\otimes$-invertible.
\end{proof}

\begin{thm}
[{\cite[2.3.75]{Ayoub_six-operationsI}, \cite[4.4.24]{Cisinski-Deglise_triangulated-categories}}]
\nlabel{construct.a7}
Let $\coef[um = {*, \otimes}]{M}{}$ be a $\mc{V}^{\otimes}$-dualizable $\mc{V}^{\otimes}$-linear \locpres\ \coeffsyst\ such that $\prns{S, \coef[um = {*, \otimes}]{M}{}}$ is solvent \tu{(\nref{construct.a3})}.
For each morphism $\pi: X \to T$ in $\sch[sepft]{S}$ with $T$ regular, and each $\otimes$-invertible, $\mc{V}^{\otimes}$-constructible object $R \in \coef{M}{T}$, the Verdier duality functor $\verdier[dm = {X}]{} \coloneqq \intmor[dm = {\coef{M}{X}}]{-}{\pi^!R}$ induces an equivalence $\coef[d = {c}]{M}{X}\op \to \coef[d = {c}]{M}{X}$ with the following properties:
\begin{enumerate}
\item
\nlabel{construct.a7.1}
the canonical morphism $M \to \verdier[dm = {X}]{\verdier[dm = {X}]{M}}$ is an equivalence for each $M \in \coef[d = {c}]{M}{X}$;
\item
\nlabel{construct.a7.2}
for each morphism $f: Y \to X$ in $\sch[sepft]{S}$, there are natural equivalences 
\[
\verdier[dm = {Y}]{}f^* 
  \simeq f^!\verdier[dm = {X}]{}:
    \coef[d = {c}]{M}{Y}\op 
      \to \coef[d = {c}]{M}{X}
\quad\text{and}\quad
\verdier[dm = {X}]{}f_! 
  \simeq f_* \verdier[dm = {Y}]{}:
    \coef[d = {c}]{M}{X}\op 
      \to \coef[d = {c}]{M}{Y};
\quad\text{and}
\]
\item
\nlabel{construct.a7.3}
there is a natural equivalence $\verdier[dm = {X}]{ \prns{-} \otimes \verdier[dm = {X}]{-}} \simeq \intmor[dm = {\coef{M}{X}}]{-}{-}: \coef{M}{X}\op \times \coef[d = {c}]{M}{X} \to \coef{M}{X}$.
\end{enumerate}
\end{thm}

\begin{proof}
As was the case for \nref{construct.a4}, the statement falls under the purview of neither \cite[2.3.75]{Ayoub_six-operationsI} nor \cite[4.4.24]{Cisinski-Deglise_triangulated-categories} unless we play fast and loose with technical hypotheses. 
Following \cite[2.3.66]{Ayoub_six-operationsI}, we say that $R \in \coef[d = {c}]{M}{X}$ is \emph{$\mc{V}^{\otimes}$-dualizing} if, for each $M \in \coef[d = {c}]{M}{X}$, the canonical morphism $M \to \intmor[dm = {\coef{M}{X}}]{\intmor[dm = {\coef{M}{X}}]{M}{R}}{R}$ is an equivalence.
If we show that, for each morphism $\pi: X \to T$ in $\sch[sepft]{S}$ with $T$ regular, the object $\pi^!\1{\coef{M}{T}}$ is $\mc{V}^{\otimes}$-dualizing, then the proof of \cite[4.4.24]{Cisinski-Deglise_triangulated-categories} goes through in our setting without modification.

In order to show that $\pi^!\1{\coef{M}{T}}$ is $\mc{V}^{\otimes}$-dualizing, we follow \cite[4.4.21]{Cisinski-Deglise_triangulated-categories}.
In following the trail of logical dependencies of \cite[4.4.21]{Cisinski-Deglise_triangulated-categories}, the only point at which the hypothesis that $\coef[um = {*, \otimes}]{M}{}$ be $\rational$-linear and given by a combinatorial left Quillen presheaf intervenes is in the proof \cite[4.4.3]{Cisinski-Deglise_triangulated-categories}.
As a workaround, one may appeal to \cite[2.2.27]{Ayoub_six-operationsI} after using Chow's lemma (\cite[5.6.1]{EGAII}) and localization (\nref[Axiom]{functors.3.E}) to eliminate the quasi-projectivity hypothesis thereof, which minor chore we entrust to the reader.
\end{proof}

\begin{prop}
\nlabel{prop:morphisms-compatible-with-six-functors}
Let $\phi^{*, \otimes}: \coef[um = {*, \otimes}]{M}{} \to \coef[um = {*, \otimes}]{N}{}$ be a morphism of $\mc{V}^{\otimes}$-linear \locpres\ \coeffsysts.
Suppose that the following conditions are satisfied:
\begin{itemize}
\item
$S = \spec{\kk}$ is the spectrum of a field of characteristic zero; and
\item
$\mc{V}^{\otimes}$ is ind-rigid.
\end{itemize}
Then the following hold:
\begin{enumerate}
\item
\nlabel{prop:morphisms-compatible-with-six-functors.1}
$\coef[um = {*, \otimes}]{M}{}$ and $\coef[um = {*, \otimes}]{N}{}$ are $\mc{V}^{\otimes}$-constructible; and
\item
\nlabel{prop:morphisms-compatible-with-six-functors.3}
$\phi^*_{\tu{c}}$ commutes up to natural equivalence with the six functors $f^*$, $f_*$, $f_!$, $f^!$, $\prns{-} \otimes \prns{-}$ and $\intmor{-}{-}$.
\end{enumerate}
\end{prop}

\begin{proof}
\nref[Claim]{prop:morphisms-compatible-with-six-functors.1} follows from \nref{construct.a4}.

If we work over quasi-projective $S$-schemes, then \nref[Claim]{prop:morphisms-compatible-with-six-functors.3} follows from the arguments of \cite[\S3]{Ayoub_operations-de-Grothendieck}.
In general, it suffices to proceed as in the proof of \cite[4.4.25]{Cisinski-Deglise_triangulated-categories} with the following modification:
in the last paragraph of that proof, replace the de\thinspace Jong alteration with a resolution of singularities in the sense of Hironaka, and replace the appeal to \cite[4.4.1]{Cisinski-Deglise_triangulated-categories} with $\cdh$-excision (\nref{excision.4}).
\end{proof}

\begin{prop}
\nlabel{prop:conservativity}
Let $\phi^{*, \otimes}: \coef[um = {*, \otimes}]{M}{} \to \coef[um = {*, \otimes}]{N}{}$ be a morphism of $\mc{V}^{\otimes}$-linear \locpres\ \coeffsysts.
Suppose that the following conditions are satisfied:
\begin{itemize}
\item
$S = \spec{\kk}$ is the spectrum of a field of characteristic zero;
\item
$\mc{V}^{\otimes}$ is ind-rigid; and
\item
$\phi^*_{\tu{c},S}: \coef[d = {c}]{M}{S} \to \coef[d = {c}]{N}{S}$ is conservative.
\end{itemize}
Then $\phi^*_{\tu{c},X}: \coef[d = {c}]{M}{X} \to \coef[d = {c}]{N}{X}$ is conservative for each $X \in \sch[ft]{S}$.
\end{prop}

\begin{proof}
Let $X \in \sch[ft]{S}$ and $M \in \coef[d = {c}]{M}{X}$ such that $\phi^*_XM \simeq 0$.
We claim that $M \simeq 0$.
By definition, the objects $V \odot f_{\sharp}\state{\coef{M}{Y}}{r}{s} \in \coef{M}{X}$ with $V \in \mc{V}_{\aleph_0}$, $f: Y \to X$ smooth and $\prns{r,s} \in \integer^2$ generate $\coef[d = {c}]{M}{X}$ under iterated finite colimits and retracts.
It therefore suffices to show that 
\[
\pi_0\map[dm = {\coef{M}{X}}]{V \odot f_{\sharp}\state{\coef{M}{Y}}{r}{s}}{M}
  \simeq 0
\]
for each $V$, $f: Y \to X$ and $\prns{r,s}$ as above.

Applying \nref{construct.a5} liberally, we have 
\begin{align*}
\map[dm = {\coef{M}{X}}]{V \odot f_{\sharp}\state{\coef{M}{Y}}{r}{s}}{M}
&\simeq \map[dm = {\coef{M}{X}}]{f_{\sharp}\prns{V \odot \state{\coef{M}{Y}}{r}{s}}}{M} \\
&\simeq \map[dm = {\coef{M}{Y}}]{V \odot \state{\coef{M}{Y}}{r}{s}}{f^*M} \\
&\simeq \map[dm = {\coef{M}{Y}}]{\1{\coef{M}{Y}}}{V^{\vee} \odot \stwist{f^*M}{-r}{-s}} \\
&\simeq \map[dm = {\coef{M}{S}}]{\1{\coef{M}{S}}}{\pi_*\prns{V^{\vee} \odot \stwist{f^*M}{-r}{-s}}} \\
&\simeq \map[dm = {\coef{M}{S}}]{\1{\coef{M}{S}}}{V^{\vee} \odot \pi_*\stwist{f^*M}{-r}{-s}},
\end{align*}
where $\pi: Y \to S$ is the structural morphism.
It therefore suffices to show that $\pi_*f^*M \simeq 0$.

By \nref{prop:morphisms-compatible-with-six-functors}, we have
$\phi^*_S\pi_*f^*M \simeq \pi_*f^*\phi^*_XM$, and $\phi^*_XM \simeq 0$ by hypothesis.
Thus, $\phi^*_S\pi_*f^*M \simeq 0$.
By hypothesis, $\phi^*_{\tu{c},S}$ is conservative, so $\pi_*f^*M \simeq 0$.
\end{proof}

\section{Scalar extension of Grothendieck's six operations}
\nlabel{scalars}

\setcounter{thm}{-1}

\begin{notation}
In this section, we fix the following notation:
\begin{itemize}
\item
$S$, a Noetherian scheme of finite dimension; and
\item
$\tau$, either the Nisnevich or the \'etale topology.
\end{itemize}
\end{notation}

\begin{motivation}
In \nref{functors.a10}, we observed that $\spt[um = {\tau}, dm = {\tatesphere}]{-}^{\wedge}$ is a \coeffsyst, and therefore admits a six-functor formalism.
In particular, it is a functor
\[
\spt[um = {\tau}, dm = {\tatesphere}]{-}^{\wedge}:
  \prns{\sch[ft]{S}}\op 
    \to \calg{\pr[um = {\tu{L}, \otimes}, d = {st}]}.
\]
We will be interested in studying two natural ways of constructing new functors from $\spt[um = {\tau}, dm = {\tatesphere}]{-}^{\wedge}$.

First, for each \locpres\ symmetric monoidal \qcategory\ $\mc{V}^{\otimes}$, one can form the $\mc{V}^{\otimes}$-linearizaion
\[
\spt[um = {\tau}, dm = {\tatesphere}]{-, \mc{V}}^{\otimes}
  \coloneqq \spt[um = {\tau}, dm = {\tatesphere}]{-}^{\wedge} \otimes \mc{V}^{\otimes},
\]
which is the natural continuation of the ideas of \nref{SH}. 
Second, for each $A \in \calg{\spt[um = {\tau}, dm = {\tatesphere}]{S}^{\wedge}}$, one has a functor
\[
\mod[dm = {A}]{\spt[um = {\tau}, dm = {\tatesphere}]{-}}^{\otimes}
  \coloneqq \spt[um = {\tau}, dm = {\tatesphere}]{-}^{\wedge} \otimes_{\spt[um = {\tau}, dm = {\tatesphere}]{S}^{\wedge}} \mod[dm = {A}]{\spt[um = {\tau}, dm = {\tatesphere}]{S}}^{\otimes}
\]
of $A$-modules in $\spt[um = {\tau}, dm = {\tatesphere}]{-}^{\wedge}$.
For example, if $A = \EM[d = {mot}]{\rational}$ is the object representing rational motivic cohomology, then $\mod[dm = {\EM[d = {mot}]{\rational}}]{\spt[um = {\tau}, dm = {\tatesphere}]{-}}^{\otimes}$ is equivalent to the symmetric monoidal \qcategory\ of Be\u{\i}linson motives of \cite{Cisinski-Deglise_triangulated-categories}.
If $A = \hodgespectrum$ is the absolute Hodge spectrum of \nref{defn:hodge-spectrum}, then
\[
\DH{-}^{\otimes}
  \coloneqq \mod[dm = {\hodgespectrum}]{\spt[u = {\'et}, dm = {\tatesphere}]{-}}^{\otimes}
\]
provides a \coeffsyst\ for mixed Hodge theory closely related to M.~Saito's derived category of mixed Hodge modules, as discussed further in \nref{intro}.

The goal of this section is to establish that $\spt[um = {\tau}, dm = {\tatesphere}]{-, \mc{V}}^{\otimes}$ and $\mod[dm = {A}]{\spt[um = {\tau}, dm = {\tatesphere}]{-}}^{\otimes}$ inherit six-functor formalisms from $\spt[dm = {\tatesphere}]{-}^{\wedge}$.
\end{motivation}

\begin{summary}
\
\begin{itemize}
\item
We begin with the general definition of scalar extension of \coeffsysts.
\item
In \nref{scalars.3}, we show that \coeffsysts\ are preserved by scalar extensions of the form $\coef[um = {*, \otimes}]{M}{} \otimes \mc{V}^{\otimes}$, i.e., that $\coef[um = {*, \otimes}]{M}{} \otimes \mc{V}^{\otimes}$ inherits a partial six-functor formalism from $\coef[um = {*, \otimes}]{M}{}$, and that the canonical morphism $\coef[um = {*, \otimes}]{M}{} \to \coef[um = {*, \otimes}]{M}{} \otimes \mc{V}^{\otimes}$ is a morphism of \coeffsysts.
\item
In \nref{scalars.5}, we list sufficient conditions for \coeffsysts\ to be preserved by scalar extensions of the form $\coef[um = {*, \otimes}]{M}{} \otimes_{\mc{V}^{\otimes}} \mc{W}^{\otimes}$.
\item
In \nref{scalars.6} and \nref{scalars.7}, we show that the conditions of \nref{scalars.5} are satisfied in particular by scalar extensions along free $A$-module functors $\mc{V}^{\otimes} \to \mod[dm = {A}]{\mc{V}}^{\otimes}$ with $A \in \calg{\mc{V}^{\otimes}}$.
\item
In \nref{ex:DH-coefficient-syst}, we introduce the six-functor formalism $\DH{-}$ and observe that it admits canonical realization functors from $\spt[u = {\'et}, dm = {\tatesphere}]{-}^{\wedge}$.
\end{itemize}
\end{summary}

\begin{defn}
\nlabel{scalars.1}
Let $v^{\otimes} : \mc{V}^{\otimes} \to \mc{W}^{\otimes}$ be a cocontinuous symmetric monoidal functor of \locpres\ symmetric monoidal \qcategories.
By \cite[4.5.3.1]{Lurie_higher-algebra}, $v^{\otimes}$ induces a symmetric monoidal functor $\mod[dm = {\mc{V}^{\otimes}}]{\pr[u = {L}]}^{\otimes} \to \mod[dm = {\mc{W}^{\otimes}}]{\pr[u = {L}]}^{\otimes}$, right adjoint to the forgetful functor.
Taking commutative algebra objects, we obtain a left adjoint $\calg{\mod[dm = {\mc{V}^{\otimes}}]{\pr[u = {L}]}^{\otimes}} \to \calg{\mod[dm = {\mc{W}^{\otimes}}]{\pr[u = {L}]}^{\otimes}}$, which, by \cite[3.4.1.7]{Lurie_higher-algebra}, corresponds to a left adjoint $\calg{\pr[um = {\tu{L}, \otimes}]}_{\mc{V}^{\otimes}/} \to \calg{\pr[um = {\tu{L}, \otimes}]}_{\mc{W}^{\otimes}/}$ to the forgetful functor.

By \cite[4.3.3]{Riehl-Verity_2-category-theory}, this induces yet another adjunction 
\begin{equation}
\nlabel{scalars.1.a}
\psh{\sch[ft]{S}}{\calg{\pr[um = {\tu{L}, \otimes}]}_{\mc{V}^{\otimes}/}}
  \rightleftarrows
  \psh{\sch[ft]{S}}{\calg{\pr[um = {\tu{L}, \otimes}]}_{\mc{W}^{\otimes}/}}.
\end{equation}
We denote the image of the functor $\coef[um = {*, \otimes}]{M}{}: \psh{\sch[ft]{S}}{\calg{\pr[um = {\tu{L}, \otimes}]}_{\mc{V}^{\otimes}/}}$ under the left adjoint in \nref{scalars.1.a} by $\coef[um = {*, \otimes}]{M}{} \otimes_{\mc{V}^{\otimes}} \mc{W}^{\otimes}$, and we refer to it as the \emph{scalar extension of $\coef[um = {*, \otimes}]{M}{}$ along $v^{\otimes}$}.
\end{defn}

\begin{rmk}
The scalar-extension construction has the following basic properties.
\begin{enumerate}
\item
\nlabel{scalars.a2}
The unit of the adjunction \nref{scalars.1.a} is a natural transformation whose component over the functor $\coef[um = {*, \otimes}]{M}{}: \prns{\sch[ft]{S}}\op \to \calg{\pr[um = {\tu{L}, \otimes}]}_{\mc{V}^{\otimes}/}$ is the natural transformation
\begin{equation}
\nlabel{scalars.a2.1}
\phi^{*, \otimes}:
  \coef[um = {*, \otimes}]{M}{} 
    \to \coef[um = {*, \otimes}]{M}{} \otimes_{\mc{V}^{\otimes}} \mc{W}^{\otimes}:
      \prns{\sch[ft]{S}}\op
        \to \calg{\pr[um = {\tu{L}, \otimes}]}_{\mc{V}^{\otimes}/}
\end{equation}
given over each $X \in \sch[ft]{S}$ by the canonical symmetric monoidal functor 
$
\id_{\coef{M}{X}^{\otimes}}:
  \coef{M}{X}^{\otimes}
    \to \coef{M}{X}^{\otimes} \otimes_{\mc{V}^{\otimes}} \mc{W}^{\otimes}
$
to the coproduct in $\calg{\mod[dm = {\mc{V}^{\otimes}}]{\pr[u = {L}]}^{\otimes}}$ (\cite[3.2.4.8]{Lurie_higher-algebra}).
\item
It follows from \nref{scalars.a3} that, for each cocontinuous symmetric monoidal functor $v^{\otimes}: \mc{V}^{\otimes} \to \mc{W}^{\otimes}$ of \locpres\ symmetric monoidal \qcategories, the adjunction \nref{scalars.1.a} restricts to an adjunction
\begin{equation}
\nlabel{scalars.a3.a}
\psh{\sch[ft]{S}}{\calg{\pr[um = {\tu{L}, \otimes}, d = {st}]}_{\mc{V}^{\otimes}/}}
  \rightleftarrows
  \psh{\sch[ft]{S}}{\calg{\pr[um = {\tu{L}, \otimes}, d = {st}]}_{\mc{W}^{\otimes}/}}.
\end{equation}
We need not require $\mc{V}^{\otimes}$ or $\mc{W}^{\otimes}$ to be stable here.
\end{enumerate}
\end{rmk}

\begin{ex}
\nlabel{ex:hodge-module-6ff}
The examples of principal interest here are the following.
\begin{enumerate}
\item
\nlabel{ex:hodge-module-6ff.1}
Scalar extension of the \coeffsyst\ $\spt[um = {\tau}, dm = {\tatesphere}]{-}^{*, \wedge}$ of \nref{functors.a10} along the unit $\spc{}^{\times} \to \mc{V}^{\otimes}$ of the \locpres\ symmetric monoidal \qcategory\ $\mc{V}^{\otimes}$ gives a presheaf 
\[
\spt[um = {\tau}, dm = {\tatesphere}]{-, \mc{V}}^{*, \otimes}:
\prns{\sch[ft]{S}}\op \to \calg{\pr[um = {\tu{L}, \otimes}, d = {st}]}_{\mc{V}^{\otimes}/}
\]
whose value at $X \in \sch[ft]{S}$ is the symmetric monoidal \qcategory\ $\spt[um = {\tau}, dm = {\tatesphere}]{X, \mc{V}}^{\otimes}$ of \nref{SH.1}.
\item
\nlabel{ex:hodge-module-6ff.2}
Scalar extension of $\spt[um = {\tau}, dm = {\tatesphere}]{-}^{*, \wedge}: \prns{\sch[ft]{S}}\op \to \calg{\pr[um = {\tu{L}, \otimes}, d = {st}]}_{\spt[um = {\tau}, dm = {\tatesphere}]{S}^{\wedge}/}$ along the realization functors $\rho^{*, \otimes}_{\tu{B}}$, $\rho^{*, \otimes}_{\tu{Hdg}}$ and $\rho^{*, \otimes}_{\tb{Hdg}}$ of \nref{ex:realizations} provides us with presheaves that we shall denote by $\DB{-}^{\otimes}$, $\DH{-}^{\otimes}$ and $\varDH{-}^{\otimes}$, respectively.
\end{enumerate}
The principal goal of this section is to show that these presheaves obtained by extension of scalars are \coeffsysts.
\end{ex}

\begin{lemma}
\nlabel{functors.10}
Let $\mc{C}$, $\mc{D}$ and $\mc{E}$ be \locpres\ \qcategories\ and $f^* : \mc{C} \rightleftarrows \mc{D} : f_*$ adjoint functors.
If $f_*$ is cocontinuous, then $f^* \otimes \id_{\mc{E}}: \mc{C} \otimes \mc{E} \to \mc{D} \otimes \mc{E}$ is left adjoint to $f_* \otimes \id_{\mc{E}}: \mc{D} \otimes \mc{E} \to \mc{C} \otimes \mc{E}$.
\end{lemma}

\begin{proof}
By the Adjoint Functor Theorem (\cite[5.5.2.9]{Lurie_higher-topos}), $f_*$ and $f^*$ both belong to $\pr^{\mr{L}}$ and $f^* \otimes \id_{\mc{E}}$ and $f_* \otimes \id_{\mc{E}}$ are well-defined.
Let $f^!$ be right adjoint to $f_*$.
By \cite[4.8.1.17]{Lurie_higher-algebra}, there are canonical equivalences $\mc{C} \otimes \mc{E} \simeq \fun[u = {R}]{\mc{C}\op}{\mc{E}}$ and $\mc{D} \otimes \mc{E} \simeq \fun[u = {R}]{\mc{D}\op}{\mc{E}}$ under which $f^* \otimes \id_{\mc{E}}$  and $f_* \otimes \id_{\mc{E}}$ correspond to the functors given by composition with $f_*\op$ and $f^{!,\varop}$, respectively.
Since $f^{!,\varop} \dashv f_*\op$, $\prns{-} \circ f_*\op$ is left adjoint to $\prns{-} \circ f^{!,\varop}$ as functors between $\fun{\mc{C}\op}{\mc{E}}$ and $\fun{\mc{D}\op}{\mc{E}}$ by \cite[4.3.3]{Riehl-Verity_2-category-theory}.
As these adjoint functors respect the full \subqcategories\ $\fun[u = {R}]{\mc{C}\op}{\mc{E}}$ and $\fun[u = {R}]{\mc{D}\op}{\mc{E}}$ and the claim follows.
\end{proof}

\begin{prop}
\nlabel{scalars.3}
Let $\coef[um = {*, \otimes}]{M}{}$ be a \locpres\ \coeffsyst, and $\mc{V}^{\otimes}$ a \locpres\ symmetric monoidal \qcategory.
Then $\coef[um = {*, \otimes}]{N}{} \coloneqq \coef[um = {*, \otimes}]{M}{} \otimes \mc{V}^{\otimes}$ is a $\mc{V}^{\otimes}$-linear \locpres\ \coeffsyst\ and the morphism $\phi^{*, \otimes}: \coef[um = {*, \otimes}]{M}{} \to \coef[um = {*, \otimes}]{N}{}$ of \nref{scalars.1.a} is a morphism of \coeffsysts. 
\end{prop}

\begin{proof}
It follows from the techniques of \cite[4.8.2.18]{Lurie_higher-algebra} that, for each $X \in \sch[ft]{S}$, the \locpres\ symmetric monoidal \qcategory\ $\coef{M}{X}^{\otimes} \otimes \mc{V}^{\otimes}$ inherits stability from $\coef{M}{X}$.
Thus, $\coef[um = {*, \otimes}]{N}{}$ is a functor $\prns{\sch[ft]{S}}\op \to \calg{\pr[um = {\tu{L}, \otimes}, d = {st}]}_{\mc{V}^{\otimes}/}$.

\nref[Axioms]{functors.3.A} and \nref{functors.3.B} follow from the \locpresbility\ hypothesis as explained in \nref[Remarks]{functors.a7.A} and \nref{functors.a7.B}.

By \nref{functors.10}, $\coef[dm = {\sharp}]{M}{f} \otimes \id_{\mc{V}} \dashv \coef[um = {*}]{M}{f} \otimes \id_{\mc{V}}$, which proves the first condition of \nref[Axiom]{functors.3.C}.
This construction of the left adjoint $\coef[dm = {\sharp}]{N}{}{f}$ of $\coef[um = {*}]{N}{f}$ for each smooth morphism $f$ of $\sch[ft]{S}$ implies that the exchange transformation $\coef[dm = {\sharp}]{N}{f_{\prns{Y'}}} \coef[um = {*}]{N}{g_{\prns{X}}} \to \coef[um = {*}]{N}{g} \coef[dm = {\sharp}]{N}{f}$ arising in the left adjointability condition of \nref[Axiom]{functors.3.C}, interpreted as a natural transformation between functors $\fun[u = {R}]{\coef{M}{X}\op}{\mc{V}} \to \fun[u = {R}]{\coef{M}{Y'}\op}{\mc{V}}$ is obtained by composing with the transpose (\nref{functors.a1.3}) of the analogous exchange transformation in $\coef[um = {*, \otimes}]{M}{}$, and is thus an equivalence.
The same argument establishes \nref[Axiom]{functors.3.D} for $\coef[um = {*, \otimes}]{N}{}$.

Limits in $\pr^{\mr{R}}$ correspond to colimits in $\pr^{\mr{L}}$ via the equivalence $\pr[u = {L}] \simeq \prns{\pr[u = {R}]}\op$ of \cite[5.5.3.4]{Lurie_higher-topos}, and so \nref[Axiom]{functors.3.E} is equivalent to the requirement that the square
\[
\begin{tikzcd}
\coef{N}{U}
\ar[r, "j_{\sharp}" below]
\ar[d]
&
\coef{N}{X}
\ar[d, "i^*" right]
\\
{*}
\ar[r]
&
\coef{N}{Z}
\end{tikzcd}
\]
be coCartesian for each closed immersion $i: Z \hookrightarrow X$ with complementary open immersion $j: U \hookrightarrow X$.
By \nref{functors.10}, this square is the image of the analogous square in $\coef[um = {*, \otimes}]{M}{}$ under the cocontinuous functor $\prns{-} \otimes \mc{V}$, so it is indeed coCartesian.

\nref[Axiom]{functors.3.G} follows from the observation that if $p: \affine^1_X \to X$ is the canonical projection, then the counit $\coef[dm = {\sharp}]{N}{p} \coef[um = {*}]{N}{p} \to \id_{\coef{N}{X}}$ is equivalent to $\prns{\coef[dm = {\sharp}]{M}{p} \coef[um = {*}]{M}{p}} \otimes \id_{\mc{V}} \to \id_{\coef{M}{X}} \otimes \id_{\mc{V}}$, which is an equivalence since $\coef[um = {*}]{M}{p}$ is fully faithful.

\nref[Axiom]{functors.3.H} holds for $\coef[um = {*, \otimes}]{N}{}$ because $\coef[dm = {\sharp}]{N}{f} \coef[dm = {*}]{N}{s} \simeq \prns{\coef[dm = {\sharp}]{M}{f} \coef[dm = {*}]{M}{s}} \otimes \id_{\mc{V}}$ is an equivalence whenever $\coef[dm = {\sharp}]{M}{f} \coef[dm = {*}]{M}{s}$ is and so, in particular, whenever $f: X \to Y$ is a smooth morphism with a section $s: Y \to X$.

Finally, $\phi^{*, \otimes}$ is a morphism of \coeffsysts\ because, for each smooth morphism $p: X \to Y$ of $\sch[ft]{S}$, $\coef[dm = {\sharp}]{N}{p} \phi^*_X \simeq \phi^*_Y\coef[dm = {\sharp}]{M}{p}$ by definition of $\coef[dm = {\sharp}]{N}{p}$.
\end{proof}

\begin{ex}
\nlabel{scalars.4}
By \nref{functors.a10} and \nref{scalars.3}, for each \locpres\ symmetric monoidal \qcategory\ $\mc{V}^{\otimes}$, the presheaf $\spt[um = {\tau}, dm = {\tatesphere}]{-, \mc{V}}^{*, \otimes} \coloneqq \spt[um = {\tau}, dm = {\tatesphere}]{-}^{*,\wedge} \otimes \mc{V}^{\otimes}$ of \nref{ex:hodge-module-6ff.1} is a $\mc{V}^{\otimes}$-linear \locpres\ \coeffsyst. 
\end{ex}

\begin{prop}
\nlabel{scalars.a6}
Let $\coef[um = {*, \otimes}]{M}{}$ be a \locpres\ \coeffsyst, $\mc{V}^{\otimes}$ a \locpres\ symmetric monoidal \qcategory, and $\tau$ a Grothendieck topology on $\sch[ft]{S}$.
Suppose that $\coef[um = {*, \otimes}]{M}{}$ is $\tau$-local.
\begin{enumerate}
\item
\nlabel{scalars.a6.1}
If $\tau$ is the Zariski, Nisnevich or \'etale topology, then $\coef[um = {*, \otimes}]{M}{} \otimes \mc{V}^{\otimes}$ is $\tau$-local.
\item
\nlabel{scalars.a6.2}
If $\coef[um = {*}]{M}{}: \prns{\sch[ft]{S}}\op \to \pr[u = {L}]$ factors through the inclusion $\pr[u = {L}, dm = {\aleph_0, \tu{st}}]$, then $\coef[um = {*, \otimes}]{M}{} \otimes \mc{V}^{\otimes}$ is $\tau$-local.
\end{enumerate}
\end{prop}

\begin{proof}
Consider \nref[Claim]{scalars.a6.1}.
We claim that $\coef[um = {*}]{M}{} \otimes \mc{V}$ is $p$-local for each $\tau$-hypercover $p: U_{\bullet} \to X$.
For that, it suffices to show that the restriction of $\coef[um = {*}]{M}{} \otimes \mc{V}$ to the small \'etale site $X_{\tu{\'et}}$ is $p'$-local, where $p'$ is the restriction of $p$ to $X_{\tu{\'et}}$.
Each morphism in $X_{\tu{\'et}}$ is smooth, so the composite $\coef[um = {*}]{M}{} \otimes \mc{V}: X_{\tu{\'et}}\op \to \pr[u = {L}, d = {st}] \to \QCATex$ factors through the inclusion $\pr[u = {R}, d = {st}] \hookrightarrow \QCATex$ by \nref[Axiom]{functors.3.C}.
Using the equivalence $\pr[u = {L}] \simeq \prns{\pr[u = {R}]}\op$ of \cite[5.5.3.4]{Lurie_higher-topos} and \nref{functors.10}, it is equivalent to show that the associated functor $\coef[dm = {\sharp}]{M}{} \otimes \mc{V}: X_{\tu{\'et}} \to \pr[u = {L}, d = {st}]$ is $p'$-local, i.e., that its left Kan extension along the Yoneda embedding $X_{\tu{\'et}} \hookrightarrow \psh{X_{\tu{\'et}}}{\spc{}}$ sends the augmented simplicial diagram $p'_+$ to a colimit diagram.
Since $\coef[um = {*}]{M}{}$ is $p$-local, and its restriction to $X_{\tu{\'et}}$ is therefore $p'$-local, the preceding remarks imply that the restriction of $\coef[dm = {\sharp}]{M}{}$ to $X_{\tu{\'et}}$ is $p'$-local.
It therefore remains to note that $\prns{-} \otimes \mc{V}: \pr[u = {L}, d = {st}] \to \pr[u = {L}, d = {st}]$ is cocontinuous.

Consider \nref[Claim]{scalars.a6.2}.
The hypothesis that $\coef[um = {*}]{M}{}$ factors through $\pr[u = {L}, dm = {\aleph_0, \tu{st}}]$ implies that $\coef[dm = {*}]{M}{}$ factors through $\pr[u = {L}, d = {st}]$.
Indeed, a left-adjoint functor between stable \locpres[\aleph_0]\ \qcategories\ preserves $\aleph_0$-presentable objects if and only if its right adjoint preserves small $\aleph_0$-filtered colimits.
That right adjoint also preserves finite colimits by the stability hypothesis, so it is cocontinuous.
Since $\coef[um = {*}]{M}{}$ is $p$-local and $\pr[u = {L}] \simeq \prns{\pr[u = {R}]}\op$, it follows that $\coef[dm = {*}]{M}{}$ is $p$-local.
Thus, $\coef[dm = {*}]{M}{} \otimes \mc{V}$ is $p$-local.
Using the equivalence $\pr[u = {L}] \simeq \prns{\pr[u = {R}]}\op$ and \nref{functors.10}, we deduce that $\coef[um = {*}]{M}{} \otimes \mc{V}$ is $p$-local.
\end{proof}

\begin{prop}
\nlabel{scalars.5}
Assume the following:
\begin{enumerate}
\item
\nlabel{scalars.5.1}
$v^{\otimes}: \mc{V}^{\otimes} \to \mc{W}^{\otimes}$ is a cocontinuous symmetric monoidal functor of \locpres\ symmetric monoidal \qcategories;
\item
\nlabel{scalars.5.2}
$\coef[um = {*, \otimes}]{M}{}: \prns{\sch[ft]{S}}\op \to \calg{\pr[um = {\tu{L}, \otimes}, d = {st}]}$ is a \locpres\ \coeffsyst;
\item
\nlabel{scalars.5.3}
$\phi^{*, \otimes}: \coef[um = {*, \otimes}]{M}{} \to \coef[um = {*, \otimes}]{N}{} \coloneqq \coef[um = {*, \otimes}]{M}{} \otimes_{\mc{V}^{\otimes}} \mc{W}^{\otimes}$ is the morphism of \tu{\nref{scalars.a2.1}};
\item
\nlabel{scalars.5.4}
$\phi_*: \coef[dm = {*}]{N}{} \to \coef[dm = {*}]{M}{}$ is the natural transformation corresponding to $\phi^*$ under the equivalence
\[
\psh{\sch[ft]{S}}{\pr[u = {L}, d = {st}]}
  \simeq \psh{\sch[ft]{S}}{\prns{\pr[u = {R}, d = {st}]}\op}
  \simeq \fun{\sch[ft]{S}}{\pr[u = {R}, d = {st}]}\op;
\]
\item 
\nlabel{scalars.5.5}
the commutative square
\begin{equation}
\nlabel{scalars.5.a}
\begin{tikzcd}
\coef{M}{Y}
\ar[r, "\phi^*_Y" below]
\ar[d, "f^*" left]
&
\coef{M}{Y} \otimes_{\mc{V}} \mc{W}
\ar[d, "f^*" right]
\\
\coef{M}{X}
\ar[r, "\phi^*_X" above]
&
\coef{M}{X} \otimes_{\mc{V}} \mc{W}
\end{tikzcd}
\end{equation}
is right adjointable whenever the morphism $f:X \to Y$ of $\sch[ft]{S}$ is smooth or a closed immersion;
\item
\nlabel{scalars.5.6}
for each $X \in \sch[ft]{S}$, $\coef{N}{X}$ is generated under small colimits by objects of the form $\phi^*_XM$ with $M \in \coef{M}{X}$.
\end{enumerate}
Then $\coef[um = {*, \otimes}]{N}{}$ is a $\mc{W}^{\otimes}$-linear \locpres\ \coeffsyst, and $\phi^{*, \otimes}$ is a morphism of such.
\end{prop}

\begin{proof}
For future reference, note that \nref[Hypothesis]{scalars.5.6} has the following consequence:
\begin{enumerate}
\setcounter{enumi}{6}
\item
\nlabel{scalars.5.7}
\emph{for each $X \in \sch[ft]{S}$, the functor $\phi_{X*}$ is conservative.}
\end{enumerate}
Indeed, an exact functor between stable \qcategories\ is conservative if and only if it reflects zero objects.
Suppose $\phi_{X*}N \simeq 0$ for some $N \in \coef{N}{X}$.
We claim that $N \simeq 0$.
By \nref[Hypothesis]{scalars.5.6}, we may write each object of $\coef{N}{X}$ as a colimit of the form $\colim_{\alpha \in A}\phi^*_XM_{\alpha}$ for some small diagram $\alpha \mapsto M_{\alpha}: A \to \coef{M}{X}$.
Then 
\[
\map[dm = {\coef{N}{X}}]{\colim_{\alpha \in A}\phi^*_X M_{\alpha}}{N} 
\simeq 
\lim_{\alpha \in A} \map[dm = {\coef{N}{X}}]{\phi^*_X M_{\alpha}}{N} 
\simeq
\lim_{\alpha \in A} \map[dm = {\coef{M}{X}}]{M_{\alpha}}{\phi_{X*}N}
\simeq 
\pt
\]
and so, by the \qcategorical\ Yoneda lemma (\cite[5.1.3.1]{Lurie_higher-topos}), $N \simeq 0$.

By \nref{scalars.1} and \nref[Remark]{scalars.a3.3}, $\coef[um = {*, \otimes}]{N}{}$ is a diagram of stable \locpres\ symmetric monoidal \qcategories.
\nref[Remark]{functors.a7.A} therefore implies \nref[Axioms]{functors.3.A} and \nref{functors.3.B}. 

Consider \nref[Axiom]{functors.3.C}.
First, we claim that $\coef[um = {*}]{N}{f}$ admits a left adjoint for each smooth morphism $f: X \to Y$ of $\sch[ft]{S}$.
Since $\coef[um = {*}]{N}{f}$ is a morphism of $\pr^{\mr{L}}$, it is accessible, so by the Adjoint Functor Theorem (\cite[5.5.2.9]{Lurie_higher-topos}), it suffices to check that $\coef[um = {*}]{N}{f}$ is continuous.
Let $\alpha \mapsto N_{\alpha}: A \to \coef{N}{Y}$ be a small diagram.
We have a commutative diagram
\[
\begin{tikzcd}
\phi_{X*} \coef[um = {*}]{N}{f} \lim_{\alpha \in A} N_{\alpha}
\ar[r, "a" below]
&
\phi_{X*} \lim_{\alpha \in A} \coef[um = {*}]{N}{f} N_{\alpha}
\ar[r, "b" below]
&
\lim_{\alpha \in A} \phi_{X*} \coef[um = {*}]{N}{f} N_{\alpha}
\\
\coef[um={*}]{M}{f} \phi_{Y*} \lim_{\alpha \in A} N_{\alpha}
\ar[u, "c" left]
\ar[r, "a'" above]
&
\coef[um={*}]{M}{f} \lim_{\alpha \in A} \phi_{Y*} N_{\alpha}
\ar[r, "b'" above]
&
\lim_{\alpha \in A} \coef[um={*}]{M}{f} \phi_{Y*} N_{\alpha}
\ar[u, "c'" right]
\end{tikzcd}
\]
in which $a$ is the image of the canonical morphism under $\phi_{X*}$, $c$ and $c'$ are the equivalences by \nref[Hypothesis]{scalars.5.5}, and $a'$ and $b'$ are equivalences, as $\coef[um={*}]{M}{f}$ and $\phi_{X*}$ are both right adjoints and thus continuous.
It now follows from \nref[Hypothesis]{scalars.5.7} that $\coef[um = {*}]{N}{f}$ is continuous, and $\coef[um = {*}]{N}{f}$ is thus a right adjoint.

Let $f: X \to Y$ and $g: Y' \to Y$ be morphisms in $\sch[ft]{S}$ with $f$ smooth.
As noted in \nref{functors.a1}, the exchange transformation $\coef[dm = {\sharp}]{N}{f_{\prns{Y'}}} \coef[um = {*}]{N}{g_{\prns{X}}} \to \coef[um = {*}]{N}{g} \coef[dm = {\sharp}]{N}{f}$ is an equivalence if and only if the transpose (\nref[Definition]{functors.a1.3}) $\alpha: \coef[um = {*}]{N}{f} \coef[dm = {*}]{N}{g} \to \coef[dm = {*}]{N}{g_{\prns{X}}'} \coef[um = {*}]{N}{f_{\prns{Y'}}}$ is.
It suffices to check this after passing to homotopy categories via \nref{functors.a8}.
Applying \cite[1.1.12]{Ayoub_six-operationsI}, we find that the diagram of homotopy categories underlying
\[
\begin{tikzcd}
\coef[um={*}]{M}{f} \phi_{Y*} \coef[dm = {*}]{N}{g}
\ar[r, "a" below]
\ar[d, "c" left]
&
\phi_{X*} \coef[um = {*}]{N}{f} \coef[dm = {*}]{N}{g}
\ar[r, "\phi_{X*}\alpha" below]
&
\phi_{X*} \coef[dm = {*}]{N}{g_{\prns{X}}} \coef[um = {*}]{N}{f_{\prns{Y'}}}
\ar[d, "c'" right]
\\
\coef[um={*}]{M}{f} \coef[dm={*}]{M}{g} \phi_{Y'*}
\ar[r, "a'" above]
&
\coef[dm={*}]{M}{g_{\prns{X}}} \coef[um={*}]{M}{f_{\prns{Y'}}} \phi_{Y*}
\ar[r, "b'" above]
&
\coef[dm={*}]{M}{g_{\prns{X}}} \phi_{X\times_SY*} \coef[um = {*}]{N}{f_{\prns{Y'}}}
\end{tikzcd}
\]
is essentially commutative.
By \nref[Hypothesis]{scalars.5.5}, $a$ and $b'$ are an equivalences;
$c$ and $c'$ are equivalences because $\phi_*$ is a natural transformation;
and $a'$ is an equivalence by \nref[Hypothesis]{scalars.5.2}.
Thus, $\phi_{X*}\alpha$ is an equivalence.
By \nref[Hypothesis]{scalars.5.7}, $\alpha$ is as well, and \nref[Axiom]{functors.3.C} holds in $\coef[um = {*, \otimes}]{N}{}$.

The proof of \nref[Axiom]{functors.3.D} is in the same spirit.
Let $f: X \to Y$ be a smooth morphism of $\sch[ft]{S}$.
We claim that, for each $\prns{M,N} \in \coef{N}{Y}^2$, the transposed exchange transformation 
\[
\coef[um = {*}]{N}{f} \intmor[dm = {\coef{N}{Y}}]{M}{N} 
  \to \intmor[dm = {\coef{N}{X}}]{\coef[um = {*}]{N}{f}M}{\coef[um = {*}]{N}{f}N}
\]
is an equivalence. 
For each $Y' \in \sch[ft]{S}$ and each $M \in \coef{M}{Y'}$, let $\eta_M$ denote the endofunctor $\intmor[dm = {\coef{M}{Y'}}]{M}{-}$ of $\coef{M}{Y'}$ and, similarly, $\eta_N \coloneqq \intmor[dm = {\coef{N}{Y'}}]{N}{-}$ for $N \in \coef{N}{Y'}$.
With this notation, we claim that the natural transformation $\beta_M: \coef[um = {*}]{N}{f}\eta_M \to \eta_{\coef[um = {*}]{N}{f}M}\coef[um = {*}]{N}{f}$ is an equivalence for each $M \in \coef{N}{Y}$.
Since $\coef[um = {*}]{N}{f}$ is bicontinuous and $\eta_{\colim_{\alpha \in A}M_{\alpha}} \simeq \lim_{\alpha \in A} \eta_{M_{\alpha}}$ for each small diagram $\alpha \mapsto M_{\alpha}: A \to \coef{N}{Y}$, it suffices by \nref[Hypothesis]{scalars.5.6} to show that $\beta_{\phi^*_YM}$ is an equivalence for each $M \in \coef{M}{Y}$.
Since $\phi^*_{Y'}$ is a symmetric monoidal functor for each $Y' \in \sch[ft]{S}$, we have a canonical equivalence
\[
\intmor[dm = {\coef{M}{Y'}}]{-}{\fct{\phi_{Y'*}}{-}} 
  \isom \phi_{Y'*}\intmor[dm = {\coef{N}{Y'}}]{\fct{\phi^*_{Y'}}{-}}{-},
\]
and thus $\eta_M \phi_{Y'*} \isom \phi_{Y'*}\eta_{\phi^*_{Y'}M}$ for each $M \in \coef{M}{Y'}$.
For each $M \in \coef{M}{Y}$, we therefore have an essentially commutative diagram
\[
\begin{tikzcd}
\coef[um={*}]{M}{f} \eta_{M} \phi_{Y*}
\ar[r]
\ar[d]
&
\coef[um={*}]{M}{f} \phi_{Y*} \eta_{\phi^*M}
\ar[r]
&
\phi_{X*} \coef[um = {*}]{N}{f} \eta_{\phi^*_YM}
\ar[d, "a" right]
\\
\eta_{\coef[um={*}]{M}{f}M} \coef[um={*}]{M}{f} \phi_{Y*}
\ar[r]
&
\eta_{\coef[um={*}]{M}{f}M} \phi_{X*} \coef[um = {*}]{N}{f}
\ar[r]
&
\phi_{X*} \eta_{\phi^*_X\coef[um={*}]{M}{f}M} \coef[um = {*}]{N}{f}
\end{tikzcd}
\]
in which $a$ is equivalent to $\phi_{X*}\beta_{\phi^*_YM}$ when we identify $\eta_{\phi^*_X\coef[um={*}]{M}{f}M}$ with $\eta_{\coef[um={*}]{M}{f}\phi^*_YM}$.
The preceding remarks imply under the given hypotheses that each of the other arrows in this diagram is an equivalence.
By \nref[Hypothesis]{scalars.5.7}, $\beta_{\phi^*_Y M}$ is an equivalence, which establishes \nref[Axiom]{functors.3.D}.

Consider \nref[Axiom]{functors.3.E}.
By \cite[2.3.3]{Cisinski-Deglise_triangulated-categories} and \cite[Proposition 9.4.20]{Robalo_thesis}, it will suffice to check that, for each closed immersion $i: Z \hookrightarrow X$ and complementary open immersion $j: U \hookrightarrow X$ in $\sch[ft]{S}$, the functors $\coef[um = {*}]{N}{i}$ and $\coef[um = {*}]{N}{j}$ are jointly conservative, and that $\coef[dm = {*}]{N}{i}$ is fully faithful.
Let $M \in \coef{N}{X}$ and suppose that $\coef[um = {*}]{N}{i}M = 0$ and that $\coef[um = {*}]{N}{j}M = 0$.
Then 
\[
\phi_{Z*} \coef[um = {*}]{N}{i} M 
  \simeq \coef[um={*}]{M}{i} \phi_{X*} M 
  = 0
\quad\text{and}\quad
\phi_{U*} \coef[um = {*}]{N}{j} M 
  \simeq \coef[um={*}]{M}{j} \phi_{X*} M 
  = 0,
\]
which implies that $\phi_{X*} M = 0$ and therefore $M = 0$.
This is the desired joint conservativity.
We now claim that the counit $\varepsilon: \coef[um = {*}]{N}{i} \coef[dm = {*}]{N}{i} \to \id_{\coef{N}{Z}}$ is an equivalence.
It suffices to check this after passing to homotopy categories via \nref{functors.a8}.
By \cite[1.1.9]{Ayoub_six-operationsI}, we have an essentially commutative square of homotopy categories underlying
\[
\begin{tikzcd}
\coef[um={*}]{M}{i} \phi_{X*} \coef[dm = {*}]{N}{i}
\ar[r, "a" above]
\ar[d, "b" left]
&
\coef[um={*}]{M}{i} \coef[dm={*}]{M}{i} \phi_{Z*}
\ar[d, "b'" right]
\\
\phi_{Z*} \coef[um = {*}]{N}{i} \coef[dm = {*}]{N}{i}
\ar[r, "\phi_{Z*}\varepsilon" above]
&
\phi_{Z*}
\end{tikzcd}
\]
in which $b$ and $b'$ are equivalences by \nref[Hypotheses]{scalars.5.2} and \nref{scalars.5.5}, and in which $a$ is an equivalence because $\phi_*$ is a natural transformation.
By \nref[Hypothesis]{scalars.5.7}, $\varepsilon$ is an equivalence.

An argument analogous but dual to that furnished for the full faithfulness of $\coef[dm = {*}]{N}{i}$ above shows that $\coef[um = {*}]{N}{p}$ is fully faithful for $p: \affine^1_X \to X$ the canonical projection, so $\coef[um = {*, \otimes}]{N}{}$ satisfies \nref[Axiom]{functors.3.G}.

While we need to establish \nref[Axiom]{functors.3.H} before we can say that its codomain is a \coeffsyst, we can at this point remark that $\phi^{*, \otimes}$ satisfies the adjointability condition of \nref{functors.a6}: this follows from \nref[Claim]{scalars.5.5} and the fact, mentioned in \nref{functors.a1}, that the square of \nref{scalars.5.a} is right adjointable if and only if its transpose is left adjointable.
We will use this remark to establish \nref[Axiom]{functors.3.H}, which will complete the proof.

As explained in \nref[Remark]{functors.a9.B}, in our situation, \nref[Axiom]{functors.3.H} is equivalent to the $\otimes$-invertibility of $\Thom{f,s}\1{\coef{N}{Y}}$ for each smooth morphism $f: X \to Y$ with a section $s: Y \to X$ in $\sch[ft]{S}$.
Since $\phi^{*,\otimes}_Y$ is symmetric monoidal and therefore preserves $\otimes$-invertibility, it will suffice to check that $\phi^*_Y \Thom{f,s}\1{\coef{M}{Y}} \simeq \Thom{f,s} \1{\coef{N}{Y}}$.
We have
\begin{align*}
\Thom{f,s}\1{\coef{N}{Y}}
&\coloneqq
\coef[dm = {\sharp}]{N}{} \coef[dm = {*}]{N}{s} \1{\coef{N}{Y}} \\
&\simeq
\coef[dm = {\sharp}]{N}{} \coef[dm = {*}]{N}{s} \phi^*_Y \1{\coef{M}{Y}} \\
&\xleftarrow{\alpha}
\coef[dm = {\sharp}]{N}{}{f} \phi^*_X \coef[dm={*}]{M}{s} \1{\coef{M}{Y}} \\
&\xrightarrow{\beta}
\phi^*_Y \coef[dm={\sharp}]{M}{f} \coef[dm={*}]{M}{s} \1{\coef{M}{Y}}
\eqqcolon
\phi^*_Y \Thom{f,s}\1{\coef{M}{Y}},
\end{align*}
where $\alpha$ is an equivalence by \cite[2.3.11]{Cisinski-Deglise_triangulated-categories} and $\beta$ is the equivalence in the adjointability condition of \nref{functors.a6}.
\end{proof}

\begin{lemma}
\nlabel{scalars.6}
Let $\mc{V}^{\otimes}$ be a \locpres\ symmetric monoidal \qcategory, let $A \in \calg{\mc{V}^{\otimes}}$, and let $\phi^{*,\otimes}: \mc{V}^{\otimes} \to \mod[dm = {A}]{\mc{V}^{\otimes}}$ be the symmetric monoidal free $A$-module functor.
\begin{enumerate}
\item
\nlabel{scalars.6.1}
The essential image of $\phi^*$ generates $\mod[dm = {A}]{\mc{V}}$ under small colimits.
\item
\nlabel{scalars.6.2}
The right adjoint $\phi_*$ of $\phi^*$ underlies a morphism of $\mod[dm = {\mc{V}^{\otimes}}]{\pr[u = {L}]}$.
\item
\nlabel{scalars.6.3}
If $\mc{M}^{\otimes} \in \mod[dm = {\mc{V}^{\otimes}}]{\pr[u = {L}]}$, then $\id_{\mc{M}} \otimes_{\mc{V}} \phi^*: \mc{M} \otimes_{\mc{V}} \mc{V} \to \mc{M} \otimes_{\mc{V}} \mod[dm = {A}]{\mc{V}}$ is left adjoint to $\id_{\mc{M}} \otimes_{\mc{V}} \phi_*$.
\item
\nlabel{scalars.6.4}
If $f^{*, \otimes}: \mc{M}^{\otimes} \to \mc{N}^{\otimes}$ is a morphism of $\mod[dm = {\mc{V}^{\otimes}}]{\pr[u = {L}]}$, then the commutative square
\[
\begin{tikzcd}
\mc{M}
\ar[rr, "\id_{\mc{M}} \otimes_{\mc{V}} \phi^*" below]
\ar[d, "f^*" left]
&
&
\mc{M} \otimes_{\mc{V}} \mod[dm = {A}]{\mc{V}}
\ar[d, "f^* \otimes_{\mc{V}} \id_{\mod[dm = {A}]{\mc{V}}}" right]
\\\
\mc{N}
\ar[rr, "\id_{\mc{N}} \otimes_{\mc{V}} \phi^*" above]
&
&
\mc{N} \otimes_{\mc{V}} \mod[dm = {A}]{\mc{V}}
\end{tikzcd}
\]
is right adjointable.
\end{enumerate}
\end{lemma}

\begin{proof}
\nref[Claim]{scalars.6.1} follows from the \qcategorical\ Barr-Beck Theorem (\cite[4.7.3.14, 4.7.3.5, 4.5.1.5]{Lurie_higher-algebra}).

By \cite[4.5.1.6, 4.5.2.1]{Lurie_higher-algebra}, there is a canonical equivalence $\psi_*: \mod[dm = {A}]{\mc{V}}^{\otimes} \isom \rmod[dm = {A}]{\mc{V}}^{\otimes}$ of \qcategories\ left tensored over $\mc{V}^{\otimes}$.
Let $\rho^*: \mc{V} \to \rmod[dm = {A}]{\mc{V}}^{\otimes}$ denote the free-right-$A$-module functor, $\rho_*$ its right adjoint, and $\psi^*$ a functor quasi-inverse to $\psi_*$.
By \cite[4.8.4.10]{Lurie_higher-algebra}, $\rho^* \dashv \rho_*$ underlies a $\mc{V}^{\otimes}$-linear adjunction and, in particular, $\rho_*$ underlies a morphism of $\lmod[dm = {\mc{V}^{\otimes}}]{\pr[u = {L}]}$.
As $\phi_* \simeq \rho_* \psi_*$, the functor $\phi_*$ underlies a morphism of $\lmod[dm = {\mc{V}^{\otimes}}]{\pr[u = {L}]}$.
Using \cite[4.5.1.6]{Lurie_higher-algebra} again, $\mod[dm = {\mc{V}^{\otimes}}]{\pr[u = {L}]} \simeq \lmod[dm = {\mc{V}^{\otimes}}]{\pr[u = {L}]}$, proving \nref[Claim]{scalars.6.2}. 

As explained in part $(b)$ of the proof of \cite[4.8.4.6]{Lurie_higher-algebra}, $\id_{\mc{M}} \otimes_{\mc{V}} \rho^* \dashv \id_{\mc{M}} \otimes_{\mc{V}} \rho_*$.
Moreover, $\id_{\mc{M}} \otimes_{\mc{V}} \psi^* \dashv \id_{\mc{M}} \otimes_{\mc{V}} \psi_*$ is a pair of quasi-inverse adjoint functors.
As 
\[
\id_{\mc{M}} \otimes_{\mc{V}} \phi^* 
  \simeq \prns{\id_{\mc{M}} \otimes_{\mc{V}} \psi^*}\prns{\id_{\mc{M}} \otimes_{\mc{V}} \rho^*}
\quad\text{and}\quad
\id_{\mc{M}} \otimes_{\mc{V}} \phi_* 
  \simeq \prns{\id_{\mc{M}} \otimes_{\mc{V}} \rho_*}\prns{\id_{\mc{M}} \otimes_{\mc{V}} \psi_*},
\]
\nref[Claim]{scalars.6.3} follows from the stability of adjunctions under composition (\cite[5.2.2.6]{Lurie_higher-topos}).

\nref[Claim]{scalars.6.4} follows from \nref[Claim]{scalars.6.3} after identifying $f^*: \mc{M} \to \mc{N}$ with $f^* \otimes_{\mc{V}} \id_{\mc{V}}: \mc{M} \otimes_{\mc{V}} \mc{V} \to \mc{N} \otimes_{\mc{V}} \mc{V}$.
\end{proof}

\begin{thm}
\nlabel{scalars.7}
Consider the following data:
\begin{itemize}
\item
$\mc{V}^{\otimes}$, a \locpres\ symmetric monoidal \qcategory;
\item
$\coef[um={*,\otimes}]{M}{}$, a $\mc{V}^{\otimes}$-linear \locpres\ \coeffsyst;
\item
$A \in \calg{\mc{V}^{\otimes}}$;
\item
$v^{*,\otimes}: \mc{V}^{\otimes} \to \mod[dm = {A}]{\mc{V}}^{\otimes}$, the symmetric monoidal free-$A$-module functor \tu{(\cite[4.5.3.1]{Lurie_higher-algebra})}; and 
\item
$\mod[dm = {A}]{\coef[um={*}]{M}{}}^{\otimes} 
  \coloneqq \coef[um={*,\otimes}]{M}{} \otimes_{\mc{V}^{\otimes}} \mod[dm = {A}]{\mc{V}}^{\otimes}$.
\end{itemize}
Then $\mod[dm = {A}]{\coef[um={*}]{M}{}}^{\otimes}$ is a $\mod[dm = {A}]{\mc{V}}^{\otimes}$-linear \locpres\ \coeffsyst, and the morphism $\phi^{*, \otimes}: \coef[um={*, \otimes}]{M}{} \to \mod[dm = {A}]{\coef[um={*}]{M}{}}^{\otimes}$ of \tu{\nref{scalars.a2}} is a morphism of \coeffsysts.
\end{thm}

\begin{proof}
By \nref{scalars.6}, each hypothesis of \nref{scalars.5} is satisfied.
\end{proof}

\begin{prop}
\nlabel{construct.3}
With the notation and hypotheses of \tu{\nref{scalars.5}}, suppose that the square \nref{scalars.5.a} is right adjointable whenever $f$ is surjective of finite type \resp{surjective, finite, radicial}.
\begin{enumerate}
\item
\nlabel{construct.3.1}
If $\coef[um={*, \otimes}]{M}{}{}$ is separated \resp{semi-separated}, then so is $\coef[um = {*, \otimes}]{N}{}$.
\item
\nlabel{construct.3.2}
If $\prns{S,\coef[um = {*, \otimes}]{M}{}}$ is solvent \tu{(\nref{construct.a3})}, then $\prns{S, \coef[um = {*, \otimes}]{N}{}}$ is solvent.
\end{enumerate}
\end{prop}

\begin{proof}
By \nref{construct.1}, $\coef[um = {*, \otimes}]{N}{}$ being separated \resp{semi-separated} amounts to $\coef[um = {*}]{N}{f}$ being conservative for each $f$ as in the statement of the proposition.
However, right adjointability of \nref{scalars.5.a} and conservativity $\phi_{Y*}$ imply that $\coef[um = {*}]{N}{f}$ reflects equivalences whenever $\coef[um={*}]{M}{f}$ does, proving \nref[Claim]{construct.3.1}.
\nref[Claim]{construct.3.2} follows.
\end{proof}

\begin{ex}
\nlabel{ex:DH-coefficient-syst}
It follows from \nref{scalars.4} and \nref{scalars.7} that the presheaves $\DB{-}^{\otimes}$, $\DH{-}^{\otimes}$ and $\varDH{-}^{\otimes}$ of \nref{ex:hodge-module-6ff.2} are \locpres\ \coeffsysts. 
In particular, they admit six-functor formalisms and morphisms of \coeffsysts\ 
\begin{align*}
\rho^{*, \otimes}_{\tu{B}}:
&\spt[um = {\tau}, dm = {\tatesphere}]{-}^{*, \wedge}
  \to \DB{-}^{\otimes} \\
\rho^{*, \otimes}_{\tu{Hdg}}:
&\spt[um = {\tau}, dm = {\tatesphere}]{-}^{*, \wedge}
  \to \DH{-}^{\otimes} \\
\rho^{*, \otimes}_{\tb{Hdg}}:
&\spt[um = {\tau}, dm = {\tatesphere}]{-, \D{\ind{\mhsp[\rational]}}}^{*, \otimes}
  \to \varDH{-}^{\otimes}.
\end{align*}
By \nref{prop:morphisms-compatible-with-six-functors}, $\rho^{*, \otimes}_{\tu{B}}$ and $\rho^{*, \otimes}_{\tu{Hdg}}$ commute with each of the six functors when restricted to constructible objects, as does $\rho^{*, \otimes}_{\tb{Hdg}}$ when restricted to $\D{\ind{\mhsp[\rational]}}^{\otimes}$-constructible objects.

The symmetric monoidal functors
\[
\DH{\spec{\complex}}^{\otimes}
  \hookrightarrow \varDH{\spec{\complex}}^{\otimes}
  \to \DB{\spec{\complex}}^{\otimes}.
\]
of \nref{ex:naturality} induce morphisms of \coeffsysts\ 
\[
\DH{-}^{\otimes}
  \xrightarrow{\iota^{*, \otimes}} \varDH{-}^{\otimes}
  \xrightarrow{\omega^{*, \otimes}} \DB{-}^{\otimes},
\]
and $\omega^{*, \otimes}$ and $\omega^{*, \otimes}\iota^{*, \otimes}$ are conservative when restricted to $\D{\ind{\mhsp[\rational]}}^{\otimes}$-constructible objects and constructible objects, respectively, by \nref{prop:conservativity}.
By \cite[17.1.7]{Cisinski-Deglise_triangulated-categories}, there is a fully faithful morphism of \coeffsysts\ $\DB[d = {c}]{-}^{\otimes} \hookrightarrow \D[u = {b}, d = {c}]{\prns{-}^{\tu{an}}, \rational}^{\otimes}$.

Combining these observations, we find that both $\DH[d = {c}]{-}^{\otimes}$ and $\varDH[d = {c}]{-}^{\otimes}$ have the properties predicted in \nref[Desiderata]{desideratum.1} through \varnref{desideratum.4}.
They also establish part of \varnref{desideratum.6}.
It remains to establish a theory of weights for these \coeffsysts, which we shall turn to next in \nref{weight}.
\end{ex}

\begin{rmk}
\nlabel{rmk:representables}
It follows from \cite[4.5.10]{Ayoub_six-operationsII} and the subsequent constructions that $\Sigma^{\infty}_{\tatesphere}\yona[um = {\tau}]{Y}{X}_+ \simeq f_{\sharp}\1{X}$ in $\spt[um = {\tau}, dm = {\tatesphere}]{Y}$ for each smooth morphism $f: X \to Y$ of $\sch[ft]{S}$.
If $\rho^{*, \otimes}: \spt[um = {\tau}, dm = {\tatesphere}]{-}^{\wedge} \to \coef[um = {*, \otimes}]{M}{}$ is a morphism of \coeffsysts, then
\[
\rho^*_Y \Sigma^{\infty}_{\tatesphere} \yona[um = {\tau}]{Y}{X}
  \simeq \rho^*_Y f_{\sharp} \1{\coef{M}{X}}
  \simeq f_{\sharp} \rho^*_X \1{\coef{M}{X}}
  \simeq f_{\sharp} \1{\coef{M}{X}}.
\]
In particular, absolute Hodge cohomology is represented by extensions of Tate objects in $\DH{X}$ and $\varDH{X}$ for each $X \in \sm[ft]{\complex}$.
Indeed, if $f: X \to \spec{\complex}$ denotes the structural morphism, then
\begin{align*}
\pi_0\map[dm = {\DH{X}}]{\1{\DH{X}}}{\state{\DH{X}}{r}{s}}
& \simeq \pi_0\map[dm = {\DH{X}}]{f^*\1{\DH{\spec{\complex}}}}{f^*\state{\DH{\spec{\complex}}}{r}{s}} \\
& \simeq \pi_0\map[dm = {\DH{\spec{\complex}}}]{f_{\sharp}f^*\1{\DH{\spec{\complex}}}}{\state{\DH{\spec{\complex}}}{r}{s}} \\
& \simeq \pi_0\map[dm = {\DH{\spec{\complex}}}]{\rho^*_{\tu{Hdg}} \Sigma^{\infty}_{\tatesphere} \yona[u = {\'et}]{\spec{\complex}}{X}}{\state{\DH{\spec{\complex}}}{r}{s}} \\
& \simeq \absolutehodge[um = {s}]{X, \twist{\rational}{r}}
\end{align*}
by \nref{rep.3.2b}.
\end{rmk}

\section{Weight structures}
\nlabel{weight}

\setcounter{thm}{-1}

\begin{notation}
Throughout this section, we fix an \emph{excellent} Noetherian scheme $S$ of finite dimension.
\end{notation}

\begin{motivation}
The stable \qcategories\ $\spt[u = {\'et}, dm = {\tatesphere}]{X, \D{\mod[dm = {\rational}]{}}}$ and $\D[u = {b}]{\mhm{X}}$ both admit weight structures compatible with the associated six-functor formalisms.
The conjectural realization functor $\DA[u = {\'et}]{X, \rational}_{\aleph_0} \to \D[u = {b}]{\mhm{X}}$ should preserve weights.
In this section, we establish the analogous statement for $\DH{X}$ by providing a general criterion for a \coeffsyst\ to admit a weight structure compatible with the six-functor formalism.

The idea is to generalize the arguments of \cite[\S3]{Hebert_structure-de-poids} in the case of $\spt[u = {\'et}, dm = {\tatesphere}]{X, \D{\mod[dm = {\rational}]{}}}$.
H\'ebert's proof hinges on the identification
\begin{equation}
\nlabel{weight.0.a}
\pi_0\map[dm = {\spt[u = {\'et}, dm = {\tatesphere}]{X, \D{\mod[dm = {\rational}]{}}}}]{\1{X}}{\state{X}{r}{s}}
  \simeq \gr[um = {r}, dm = {\gamma}]{\operator[dm = {2r-s}]{K}{X}_{\rational}}
\end{equation}
and the vanishing of $\gr[um = {r}, dm = {\gamma}]{\operator[dm = {2r-s}]{K}{Y}_{\rational}}$ for $s > 2r$.
An analogous vanishing statement holds for absolute Hodge cohomology.
\end{motivation}

\begin{summary}
\
\begin{itemize}
\item
In \nref{weight.2a}, we recall the expected compatibilities between weight structures and Grothendieck's six functors.
\item
In \nref{weight.1}, we introduce a key lemma that allows us to deduce a fundamental vanishing result for separated morphisms of finite type to the analogous vanishing result for closed immersions between regular schemes.
\item
In \nref{weight.2}, we establish the main result regarding the existence of weight structures on \coeffsysts\ satisfying a vanishing condition analogous to that of \eqref{weight.0.a}, and the compatibility of these weight structures with the six functors.
\item
In \nref{weight.5}, we remark that the weight structures resulting from \nref{weight.2} are suitably natural with respect to morphisms of \coeffsysts.
\item
In \nref{prop:hodge-weights}, we conclude by observing that $\DH{-}$ satisfies the conditions of \nref{weight.2} and that the realization $\rho^*_X: \spt[u = {\'et}, dm = {\tatesphere, \tu{c}}]{X, \D{\mod[dm = {\rational}]{}}} \to \DH[d = {c}]{X}$ is therefore weight exact for each $X \in \sch[ft]{S}$.
\end{itemize}
\end{summary}

\begin{defn}
\nlabel{defn:purity-structure}
Let $\mc{V}^{\otimes}$ be a stable \locpres[\aleph_0]\ symmetric monoidal \qcategory.
A \emph{purity structure on $\mc{V}^{\otimes}$} consists of a set $\mf{G}$ of equivalence classes of $\aleph_0$-presentable, $\otimes$-dualizable objects of $\mc{V}^{\otimes}$ satisfying the following conditions: 
\begin{enumerate}
\item
$\mf{G}$ contains $\1{\mc{V}}$;
\item
$\mf{G}$ is stable under $\prns{-} \otimes \prns{-}$ and $\prns{-}^{\vee}$; and
\item
$\mf{G}$ generates $\mc{V}_{\aleph_0}$ under $\sphere^1$-suspensions, $\sphere^1$-desuspensions, iterated finite colimits, and retracts.
\end{enumerate}
\end{defn}

\begin{defn}
\nlabel{weight.2a}
Let $\coef[um = {*, \otimes}]{M}{}: \sch[ft]{S} \to \calg{\QCATexmon}$ be a \coeffsyst.
\begin{enumerate}
\item
\nlabel{weight.2a.1}
A \emph{weight structure $\prns{\coef[um = {\mf{w} \leq 0}]{M}{}, \coef[um = {\mf{w} \geq 0}]{M}{}}$ on $\coef[um = {*, \otimes}]{M}{}$} is the data of a weight structure $\prns{\coef{M}{X}^{\mf{w} \leq 0}, \coef{M}{X}^{\mf{w} \geq 0}}$ on the stable \qcategory\ $\coef{M}{X}$ for each $X \in \sch[ft]{S}$.
\item
\nlabel{weight.2a.2}
Let $\mc{V}^{\otimes}$ be a stable \locpres[\aleph_0]\ symmetric monoidal \qcategory,
$\mf{G}$ a purity structure on $\mc{V}^{\otimes}$,
and suppose that  $\coef[um = {*, \otimes}]{M}{}$ is a $\mc{V}^{\otimes}$-linear \locpres\ \coeffsyst.
A weight structure $\prns{\coef[um = {\mf{w} \leq 0}, d = {c}]{M}{}, \coef[um = {\mf{w} \geq 0}, d = {c}]{M}{}}$ on $\coef[um = {*, \otimes}, d = {c}]{M}{}$ is \emph{$\mf{G}$-constructible} if it satisfies the following conditions:
\begin{enumerate}
\item
for each regular $X \in \sch[ft]{S}$ and each $V \in \mf{G}$, $V \odot \1{\coef{M}{X}} \in \coef[d = {c}]{M}{X}^{\mf{w} = 0}$;
\item
for each morphism $f: X \to Y$ of $\sch[ft]{S}$, $f^*$ preserves objects of weight $\leq 0$, and $f_*$ preserves objects of weight $\geq 0$;
\item
if $f: X \to Y$ is a morphism of $\sch[sepft]{S}$, then $f_!$ preserves objects of weight $\leq 0$, and $f^!$ preserves objects of weight $\geq 0$;
\item
for each $X \in \sch[ft]{S}$, the tensor product and internal morphisms-object bifunctors restrict to bifunctors
\begin{align*}
\prns{-} \otimes_{\coef[d = {c}]{M}{X}} \prns{-}
  &: \coef[d = {c}]{M}{X}^{\mf{w} \leq 0} \times \coef[d = {c}]{M}{X}^{\mf{w} \leq 0}
    \to \coef[d = {c}]{M}{X}^{\mf{w} \leq 0} \\
\intmor[dm = {\coef[d = {c}]{M}{X}}]{-}{-}
  &: \coef[d = {c}]{M}{X}^{\mf{w} \leq 0} \times \coef[d = {c}]{M}{X}^{\mf{w} \geq 0}
    \to \coef[d = {c}]{M}{X}^{\mf{w} \geq 0}; \quad \text{and}
\end{align*}
\item
for each $r \in \integer$, the functor $\stwist{\prns{-}}{r}{2r}: \coef[d = {c}]{M}{X} \to \coef[d = {c}]{M}{X}$ is \emph{weight exact}, i.e., it preserves objects of weight $\leq 0$ and objects of weight $\leq 0$.
\end{enumerate}
\end{enumerate}
\end{defn}

\begin{lemma}
\nlabel{weight.1}
Let $S$ be an excellent, Noetherian scheme, 
$\mc{V}^{\otimes}$ a \locpres\ symmetric monoidal \qcategory,
and $\coef[um = {*, \otimes}]{M}{}$ a $\mc{V}^{\otimes}$-linear \locpres\ \coeffsyst.
Consider the following property of separated morphisms $f: X \to Y$ of $\sch[ft]{S}$ and pairs $\prns{V, W}$ of $\otimes$-dualizable objects of $\mc{V}^{\otimes}$:
\begin{enumerate}
\item[\tu{$\boldoperator{van}{f,V,W}$:}]
if $\prns{r, s} \in \integer^2$, and if $s > 2r$, then 
\[
\pi_0\map[dm = {\coef{M}{Y}}]{V \odot f_!\1{\coef{M}{X}}}{W \odot \state{\coef{M}{Y}}{r}{s}} = 0.
\]
\end{enumerate}
Fix a pair $\prns{V,W}$ of $\otimes$-dualizable objects of $\mc{V}^{\otimes}$.
If $\boldoperator{van}{i,V,W}$ holds for each closed immersion $i: Z \hookrightarrow T$ between regular schemes in $\sch[ft]{S}$ such that the normal bundle $\normalbundle{i}$ is trivial,
then $\boldoperator{van}{f,V,W}$ holds for each separated morphism $f: X \to Y$ in $\sch[ft]{S}$ with $Y$ regular.
\end{lemma}

\begin{proof}
Let $f: X \to Y$ be a separated morphism of $\sch[ft]{S}$, and let $\prns{V,W} \in \mc{V}^2$ be a pair of $\otimes$-dualizable objects.
We claim that $\boldoperator{van}{f,V,W}$ holds.
Since
\[
\pi_0\map[dm = {\coef{M}{Y}}]{V \odot f_!\1{\coef{M}{X}}}{W \odot \state{\coef{M}{Y}}{r}{s}}
  \simeq \pi_0\map[dm = {\coef{M}{Y}}]{f_!\1{\coef{M}{X}}}{\prns{V^{\vee} \otimes W} \odot \state{\coef{M}{Y}}{r}{s}},
\]
it suffices to establish the equivalent condition $\boldoperator{van}{f,\1{\mc{V}},V}$.

If $X = X_1 \amalg X_2$ is a disjoint union of open subsets, and if $j_{\alpha}: X_{\alpha} \hookrightarrow X$ denotes the associated open immersion for $\alpha \in \brc{1,2}$, then $\boldoperator{van}{fj_1,\1{\mc{V}},V}$ and $\boldoperator{van}{fj_2,\1{\mc{V}},V}$ together imply $\boldoperator{van}{f,\1{\mc{V}},V}$.
Indeed, Nisnevich excision (\nref{excision.2}) implies that $\coef{M}{X} \simeq \coef{M}{X_1} \times \coef{M}{X_2}$.
We may therefore assume $X$ is connected.

We may assume $X$ is irreducible.
Indeed, if $X = \bigcup^n_{\alpha = 1} X_{\alpha}$ is a finite union of irreducible components, then the Cartesian square
\[
\begin{tikzcd}
\coprod^n_{\alpha = 1} \prns{X_{\alpha} \cap X_n}
\ar[r, "i'_n" below]
\ar[d, "p'" right]
&
\coprod^n_{\alpha = 1} X_{\alpha}
\ar[d, "p" left]
\\
X_n
\ar[r, "i_n" above]
&
X
\end{tikzcd}
\]
is $\cdh$-distinguished (\nref{excision.3}).
By $\cdh$-excision (\nref{excision.4}), we have a fiber sequence
\[
f_!\1{\coef{M}{X}}
  \to f_!i_{n,*}i^*_n\1{\coef{M}{X}} \oplus f_!p_*p^*\1{\coef{M}{X}}
  \to f_!q_*q^*\1{\coef{M}{X}},
\]
where $q = pi'_n = i_np': \coprod^n_{\alpha = 1}\prns{X_{\alpha} \cap X_n} \to X$. 
Since $i_n$, $p$, and $q$ are proper, the long exact sequence associated with the image of this fiber sequence under $\map[dm = {\coef{M}{Y}}]{-}{V \odot \state{\coef{M}{Y}}{r}{s}}$ shows that $\boldoperator{van}{f,\1{\mc{V}},V}$ follows from $\boldoperator{van}{fi_n,\1{\mc{V}},V}$, $\boldoperator{van}{fp,\1{\mc{V}},V}$ and $\boldoperator{van}{fq,\1{\mc{V}},V}$.
Since the domains of $i_n$, $p$, and $q$ are disjoint unions of their respective irreducible components, we may assume that $X$ is irreducible by our previous reduction to the case in which $X$ is connected.

We may furthermore assume that $X$ is reduced, hence integral.
Indeed, if $i: X_{\tu{red}} \hookrightarrow X$ is the reduction, then $i^*$ is an equivalence by \cite[2.1.163]{Ayoub_six-operationsI}, so 
\[
\prns{fi}_!\1{\coef{M}{X_{\tu{red}}}} 
  \simeq f_!i_!\1{\coef{M}{X_{\tu{red}}}}
  \simeq f_!i_!i^*\1{\coef{M}{X}}
  \simeq f_!\1{\coef{M}{X}}
\]
and we may assume $X$ is reduced.

Since $S$ is excellent, the regular locus $\operator{Reg}{X} \subseteq X$ is open by \cite[Scholie~7.8.3.(iv)]{EGAIVb}.
Since $X$ is integral, its generic point belongs to $\operator{Reg}{X}$, which is therefore nonempty, hence dense and open in $X$.

Finally, if $i: \operator{Sing}{X} \hookrightarrow X$ is the reduced subscheme structure on the singular locus of $X$, then the localization axiom (\nref[Axiom]{functors.3.E}) applied to $i$ and its complementary open immersion $j: \operator{Reg}{X} \hookrightarrow X$ provides us with a fiber sequence
\[
f_!j_!\1{\coef{M}{\operator{Reg}{X}}}
  \to f_!\1{\coef{M}{X}}
  \to f_!i_*\1{\coef{M}{\operator{Sing}{X}}}.
\]
The long exact sequence associated to the image of this fiber under $\map[dm = {\coef{M}{Y}}]{-}{V \odot \state{\coef{M}{Y}}{r}{s}}$ shows that $\boldoperator{van}{f,\1{\mc{V}},V}$ follows from $\boldoperator{van}{fj,\1{\mc{V}},V}$ and $\boldoperator{van}{fi,\1{\mc{V}},V}$.
The latter is true by our inductive hypothesis. 
We may therefore replace $X$ by $\operator{Reg}{X}$, i.e., we may assume that $X$ is regular and integral.

By the same argument that allows us to replace $X$ by $\operator{Reg}{X}$, we may replace $X$ by any nonempty open subscheme.
In particular, we may assume that $X$ is affine, regular and integral.

Let $h: U' \hookrightarrow Y$ be an affine open subscheme such that $X \times_Y U' \neq \varnothing$.
Replacing $X$ by any nonempty affine open $j: U \hookrightarrow X$ whose image is contained in $X \times_Y U'$, we may assume $X$ is regular, affine and integral, and that $f$ factors through $h$. 
We may thus factor $f$ as 
\[
X 
  \xrightarrow{i} \affine^n_{U'} 
  \xrightarrow{p} U' 
  \xrightarrow{h} Y,
\]
where $i$ is a closed immersion, and $p$ is the canonical projection for some $n \in \integer_{\geq0}$.

We may also choose an open subscheme $j': U'' \hookrightarrow \affine^n_{U'}$ such that $X' \coloneqq X \times_{\affine^n_{U'}} U'' \neq \varnothing$ and the restriction $j'^*\normalbundle{i}$ of the normal bundle of $i$ to $X'$ is trivial.
By \cite[16.2.2.(iii)]{EGAIVd}, $j'^*\normalbundle{i}$ is isomorphic to the normal bundle of the closed immersion $i': X' \hookrightarrow U''$.
Replacing $X$ by $X'$, we may assume that $f$ factors as
\[
X
  \xrightarrow{i} U''
  \xrightarrow{j} \affine^n_{U'}
  \xrightarrow{p} U'
  \xrightarrow{h} Y
\] 
with $X$ regular, $i$ a closed immersion with trivial normal bundle, $j$ and $h$ open immersions, and $p$ the projection.

We have equivalences
\begin{align*}
&\map[dm = {\coef{M}{Y}}]{f_!\1{\coef{M}{X}}}{V \odot \state{\coef{M}{Y}}{r}{s}} \\
&\qquad \simeq
\map[dm = {\coef{M}{Y}}]{h_!p_!j_!i_!\1{\coef{M}{X}}}{V \odot \state{\coef{M}{Y}}{r}{s}}
\\
& \qquad \simeq
\map[dm = {\coef{M}{U''}}]{i_!\1{\coef{M}{X}}}{j^*p^!h^*\prns{V \odot  \state{\coef{M}{Y}}{r}{s}}}
&&
\text{adjunction}
\\
& \qquad \simeq
\map[dm = {\coef{M}{U''}}]{i_!\1{\coef{M}{X}}}{j^*p^*h^*\prns{V \odot  \state{\coef{M}{Y}}{r + n}{s + 2n}}}
&&
\text{\nref{functors.a13.3}}
\\
& \qquad \simeq
\map[dm = {\coef{M}{U''}}]{i_!\1{\coef{M}{X}}}{V \odot \state{\coef{M}{U''}}{r + n}{s + 2n}}
&&
\text{$\mc{V}^{\otimes}$-linearity}.
\end{align*}
Since $X$ and $U''$ are regular and $\normalbundle{i}$ is trivial, the claim now follows, as $\boldoperator{van}{i,\1{\mc{V}},V}$ holds by fiat.
\end{proof}

\begin{rmk}
\nlabel{weight.a4}
Let $S = \spec{\kk}$ be the spectrum of a perfect field,
$\coef[um = {*, \otimes}]{M}{}$ a \locpres\ \coeffsyst,
$i: Z \hookrightarrow X$ a closed immersion of $\sch[ft]{S}$ with $X$ and $Z$ regular.
Suppose the normal bundle $\normalbundle{i}$ is trivial.
Absolute purity in $\coef[um = {*, \otimes}]{M}{}$ is automatic in this case.
In other words, letting $p: Z \to S$ and $q: X \to S$ denote the structure morphisms, by the results of \cite[\S1.6.1]{Ayoub_six-operationsI}, there are equivalences
\[
i^!\1{\coef{M}{X}}
\simeq i^!q^*\1{\coef{M}{S}} 
\simeq \Thom[um = {-1}]{\normalbundle{i}}p^*\1{\coef{M}{S}}
\simeq \state{\coef{M}{Z}}{-c}{-2c}
\]
where $c$ is the codimension of $i$.
\end{rmk}

\begin{thm}
[D.~H\'ebert]
\nlabel{weight.2}
Consider the following data:
\begin{itemize}
\item
$S = \spec{\kk}$, the spectrum of a field $\kk$ of characteristic zero;
\item
$\mc{V}^{\otimes}$, a stable \locpres[\aleph_0]\ symmetric monoidal \qcategory;
\item
$\mf{G}$, a purity structure on $\mc{V}^{\otimes}$; and
\item
$\coef[um = {*, \otimes}]{M}{}$, a $\mc{V}^{\otimes}$-linear \locpres\ \coeffsyst.
\end{itemize}
Suppose that $\prns{S, \coef[um = {*, \otimes}]{M}{}}$ is solvent, $\coef[um = {*, \otimes}]{M}{}$ is $\mc{V}^{\otimes}$-constructible, and the following condition is satisfied:
\begin{enumerate}[ref=$\boldoperator{van}{A}$]
\item[$\boldoperator{van}{\coef[um = {*, \otimes}]{M}{},\mf{G}}$:]
\nlabel{weight.2.1}
for each object $f: Y \to S$ of $\sm[ft]{S}$, 
each $V \in \mf{G}$, and
each $\prns{r,s} \in \integer^2$ with $s > 2r$,
\[
\pi_0 \map[dm = {\coef{M}{S}}]{V \odot \1{S}}{f_*f^*\state{S}{r}{s}}
  = 0.
\]
\end{enumerate}
Then $\coef[um = {*, \otimes}, d = {c}]{M}{}$ admits a unique $\mf{G}$-constructible weight structure.
\end{thm}

\begin{proof}
Let $\prns{V,W} \in \mf{G}^2$, and let $i: Z \hookrightarrow X$ be a closed immersion of codimension $c$ between regular schemes in $\sch[ft]{S}$ such that the normal bundle $\normalbundle{i}$ is trivial.
We claim that the condition $\boldoperator{van}{i,V,W}$ of \nref{weight.1} is satisfied.
By the dual of the argument given at the beginning of the proof of \nref{weight.1}, this is equivalent to proving $\boldoperator{van}{i,V,\1{\mc{V}}}$ for each $V \in \mf{G}$.

If $\prns{r, s} \in \integer^2$ with $s > 2r$ and $f: Z \to S$ is the structural morphism, which is smooth under our assumptions, then 
\begin{align*}
\map[dm = {\coef{M}{X}}]{V \odot i_!\1{\coef{M}{Z}}}{\state{\coef{M}{X}}{r}{s}}
  &\simeq \map[dm = {\coef{M}{X}}]{i_!\prns{V \odot \1{\coef{M}{Z}}}}{\state{\coef{M}{X}}{r}{s}} 
  && \nref{construct.a5} \\
  &\simeq \map[dm = {\coef{M}{Z}}]{V \odot \1{\coef{M}{Z}}}{i^!\state{\coef{M}{X}}{r}{s}} \\
&\simeq \map[dm = {\coef{M}{Z}}]{V \odot \1{\coef{M}{Z}}}{\state{\coef{M}{Z}}{r-c}{s-2c}} 
&&\text{\nref{weight.a4}} 
\\
  &\simeq \map[dm = {\coef{M}{Z}}]{f^*\prns{V \odot \1{\coef{M}{S}}}}{f^*\state{\coef{M}{S}}{r-c}{s-2c}} 
\\
  &\simeq \map[dm = {\coef{M}{S}}]{V \odot \1{\coef{M}{S}}}{f_*f^*\state{\coef{M}{S}}{r-c}{s-2c}}
\end{align*}
It remains to note that $s-2c > 2(r-c)$, so $\boldoperator{van}{\coef[um = {*, \otimes}]{M}{},\mf{G}}$ implies $\boldoperator{van}{i,V,\1{\mc{V}}}$.

Thus, $\boldoperator{van}{i,V,W}$ holds for $i$ as above and $\prns{V,W} \in \mf{G}^2$.
The assumptions of \nref{weight.1} are therefore satisfied, and we deduce that $\boldoperator{van}{f,V,W}$ holds for each separated morphism $f: X \to Y$ in $\sch[ft]{S}$ with $Y$ regular and each $\prns{V,W} \in \mf{G}^2$.

The existence of the desired weight structure now follows from this observation by the arguments of \cite[3.3, 3.8]{Hebert_structure-de-poids} after substituting \nref{weight.1} for \cite[3.2]{Hebert_structure-de-poids}, and \nref{excision.4} and Hironaka's resolution of singularities for \htop-descent and de\thinspace Jong alterations.
The uniqueness of this weight structure follows from \cite[4.3.2.II.1]{Bondarko_weight-structures-vs} (see also \cite[1.9]{Hebert_structure-de-poids}).
\end{proof}

\begin{rmk}
\nlabel{rmk:weight-generators}
With the notation and hypotheses of \nref{weight.2}, the arguments of \cite[3.2, 3.6]{Hebert_structure-de-poids} provides the following explicit description of the weight structure on $\coef[d = {c}]{M}{X}$ for each $X \in \sch[ft]{S}$:
\begin{enumerate}
\item
the heart is the smallest replete, idempotent-complete full \subqcategory\ of $\coef[d = {c}]{M}{X}$ containing each finite coproduct of objects of the form $V \odot f_*\state{\coef{M}{Y}}{r}{2r}$ for each $V \in \mf{G}$, each proper morphism $f: Y \to X$ with $Y$ regular, and each $r \in \integer$; and
\item
$\coef[um = {\mf{w} \leq 0}, d = {c}]{M}{X} \subseteq \coef[d = {c}]{M}{X}$ \resp{$\coef[um = {\mf{w} \geq 0}, d = {c}]{M}{X} \subseteq \coef[d = {c}]{M}{X}$} is the smallest replete, idempotent-complete full \subqcategory\ stable under finitely iterated extensions containing the objects of the form $V \odot f_*\state{\coef{M}{Y}}{r}{s}$ for each $V \in \mf{G}$, each proper morphism $f: Y \to X$ with $Y$ regular, and each $\prns{r,s} \in \integer^2$ such that $s \leq 2r$ \resp{$s \geq 2r$}.
\end{enumerate}
\end{rmk}

\begin{prop}
\nlabel{weight.5}
Consider the following data:
\begin{itemize}
\item
$S = \spec{\kk}$, the spectrum of a field $\kk$ of characteristic zero;
\item
$v^{\otimes}: \mc{V}^{\otimes} \to \mc{W}^{\otimes}$, a cocontinuous symmetric monoidal functor between stable \locpres[\aleph_0]\ symmetric monoidal \qcategories;
\item
$\mf{G}$ and $\mf{H}$, purity structures on $\mc{V}^{\otimes}$ and $\mc{W}^{\otimes}$, respectively;
\item
$\coef[um = {*, \otimes}]{M}{}$, a $\mc{V}^{\otimes}$-constructible $\mc{V}^{\otimes}$-linear \locpres\ \coeffsyst;
\item
$\coef[um = {*, \otimes}]{N}{}$, a $\mc{W}^{\otimes}$-constructible $\mc{W}^{\otimes}$-linear \locpres\ \coeffsyst; and
\item
$\phi^{*, \otimes}: \coef[um = {*, \otimes}]{M}{} \to \coef[um = {*, \otimes}]{N}{}$, a morphism of $\mc{V}^{\otimes}$-linear \coeffsysts.
\end{itemize}
If the conditions $\boldoperator{van}{\coef[um = {*, \otimes}]{M}{}, \mf{G}}$ and $\boldoperator{van}{\coef[um = {*, \otimes}]{N}{}, \mf{H}}$ of \nref{weight.2} are satisfied, and if $v$ sends $\mf{G}$ into $\mf{H}$, then $\phi^*_{\tu{c},X}: \coef[d = {c}]{M}{X} \to \coef[d = {c}]{N}{X}$ is weight exact for each $X \in \sch[ft]{S}$.
\end{prop}

\begin{proof}
Under our hypotheses, $\mc{V}^{\otimes}$ is ind-rigid.
The claim therefore follows from the compatibility of $\phi^*$ with $f_*$ for $f$ proper (\nref{prop:exchange-for-proper-pushforward}) and the description of the generators of the weight structures given in \nref{rmk:weight-generators}.
\end{proof}

\begin{prop}
\nlabel{prop:hodge-weights}
Let $S = \spec{\kk}$ be the spectrum of a field $\kk$ of characteristic zero,
let $\mf{G} \coloneqq \brc{\brk{\1{\mod[dm = {\rational}]{}}}}$, and let $\mf{H}$ denote the set of equivalence classes of the objects of $\D{\ind{\mhsp[\rational]}}$ of the form $V\brk{-r}$, where $V$ is a polarizable pure Hodge $\rational$-structure of weight $r$ regarded as a cochain complex in degree $0$.
Then we have the following:
\begin{enumerate}
\item
\nlabel{prop:hodge-weights.1}
$\boldoperator{van}{\spt[u = {\'et}, dm = {\tatesphere}]{-, \D{\mod[dm = {\rational}]{}}}^{\otimes}, \mf{G}}$, $\boldoperator{van}{\DH{-}^{\otimes}, \mf{G}}$, and $\boldoperator{van}{\varDH{-}^{\otimes}, \mf{H}}$ hold;
\item
\nlabel{prop:hodge-weights.2}
the \coeffsysts\ $\spt[u = {\'et}, dm = {\tatesphere, \tu{c}}]{-, \D{\mod[dm = {\rational}]{}}}^{\otimes}$ and $\DH[d = {c}]{-}^{\otimes}$ admit unique $\mf{G}$-constructible weight structures, and the morphism of \coeffsysts\ $\rho^{*, \otimes}_{\tu{Hdg}, \tu{c}}$ is weight exact; and
\item
\nlabel{prop:hodge-weights.3}
the \coeffsyst\ $\varDH[d = {c}]{-}^{\otimes}$ admits a unique $\mf{H}$-constructible weight structure, and the morphism of \coeffsysts\ $\iota^{*, \otimes}\cnstr: \DH[d = {c}]{-}^{\otimes} \to \varDH[d = {c}]{-}^{\otimes}$ is weight exact.
\end{enumerate}
\end{prop}

\begin{proof}
Consider $\boldoperator{van}{\spt[u = {\'et}, dm = {\tatesphere}]{-, \D{\mod[dm = {\rational}]{}}}^{\otimes}, \mf{G}}$.
It follows from the arguments of \cite[5.3.35, 16.2.18, 14.2.14]{Cisinski-Deglise_triangulated-categories} that
\[
\pi_0\map[dm = {\spt[u = {\'et}, dm = {\tatesphere}]{X, \D{\mod[dm = {\rational}]{}}}}]{\1{X}}{\state{X}{r}{s}}
  \simeq \gr[um = {r}, dm = {\gamma}]{\Ktheory[dm = {2r-s}]{X}_{\rational}}
\]
for each regular $X \in \sch[ft]{S}$,
where $\gr[um = {r}, dm = {\gamma}]{\Ktheory[dm = {2r-s}]{X}_{\rational}}$ denotes the graded piece of rationalized algebraic $\Ktheory{}$-theory of $X$ with respect to the $\gamma$-filtration.
Since $\Ktheory[dm = {n}]{X} = 0$ for $n < 0$, $\boldoperator{van}{\spt[u = {\'et}, dm = {\tatesphere}]{-, \D{\mod[dm = {\rational}]{}}}^{\otimes}, \mf{G}}$ follows.

By \nref{rmk:representables}, $\boldoperator{van}{\DH{-}^{\otimes}, \mf{G}}$ and $\boldoperator{van}{\varDH{-}^{\otimes}, \mf{H}}$ follow from \nref{rep.3.2b} and the fact that the mixed Hodge structure on $\h^n\enhancedbetti{X}$ is of weight $\geq n$ for each $n \in \mb{Z}$ and each $X \in \sm[ft]{\complex}$.

This proves \nref[Claim]{prop:hodge-weights.1}, and \nref[Claims]{prop:hodge-weights.2} and \varnref{prop:hodge-weights.3} follow immediately from \nref{weight.5}.
\end{proof}

\begin{rmk}
\nref{prop:hodge-weights} establishes \nref[Desideratum]{desideratum.5} for $\DH[d = {c}]{-}$ and $\varDH[d = {c}]{-}$.
Combined with \nref{ex:DH-coefficient-syst}, it also established \nref[Desideratum]{desideratum.6}.
This completes the proof of \nref{mainthm}:
both $\DH[d = {c}]{-}$ and $\varDH[d = {c}]{-}$ satisfy the properties expected of any reasonable theory of constructible coefficients for mixed Hodge theory, with the exception of those properties that involve the existence of a t-structure lifting the perverse t-structure on $\D[u = {b}, d = {c}]{\prns{-}^{\tu{an}}, \rational}$. 
\end{rmk}

\bibliographystyle{alpha}
\bibliography{all}
\end{document}